\documentclass{article}
\usepackage[utf8]{inputenc}

\usepackage{amsfonts}
\usepackage{amssymb}
\usepackage{amsmath}
\usepackage{amsthm}
\usepackage{mathtools}
\usepackage{cleveref}
\usepackage{url}
\usepackage{tikz}
\usepackage{tikz-cd}
\usepackage{adjustbox}
\usepackage{enumitem}
\usepackage{subcaption}
\usepackage{graphicx}
\usepackage{afterpage}
\usepackage{changepage}
\usepackage{authblk}
\usepackage{longtable}
\usepackage{booktabs}
\usepackage{bm}
\usepackage{multicol}
\usepackage{wrapfig}
\usepackage{todonotes}

\usetikzlibrary{knots}

\usepackage[
    backend=biber,
    style=alphabetic,
    sorting=nyt,
    doi=false,
    url=false,
    isbn=false,
    maxbibnames=10,
]{biblatex}
\DeclareFieldFormat[misc]{title}{``#1''}

\theoremstyle{plain}
    \newtheorem{maintheorem}{Theorem}
    \newtheorem{theorem}{Theorem}[section]
    \newtheorem{proposition}[theorem]{Proposition}
    \newtheorem{corollary}[theorem]{Corollary}
    \newtheorem{lemma}[theorem]{Lemma}
    \newtheorem{claim}[theorem]{Claim}

\theoremstyle{definition}
    \newtheorem{definition}[theorem]{Definition}
    \newtheorem{algorithm}[theorem]{Algorithm}
    
    \newtheorem{example}[theorem]{Example}
    \newtheorem{question}[theorem]{Question}

\theoremstyle{remark}
	\newtheorem{remark}[theorem]{Remark}%

\crefname{maintheorem}{Theorem}{Theorems}
\crefname{claim}{Claim}{Claims}

\newtheoremstyle{restated}
    {\topsep}{\topsep} 
    {\itshape}         
    {}                 
    {\bfseries}        
    {.}                
    { }                
    {\thmname{#1} \ref{#3} {\normalfont(Restated)}}
    
\theoremstyle{restated}
    \newtheorem{restate-theorem}{Theorem}
    \newtheorem{restate-proposition}{Proposition}
    \newtheorem{restate-corollary}{Corollary}

\AddToHook{env/proposition/begin}{\crefalias{theorem}{proposition}}
\AddToHook{env/corollary/begin}{\crefalias{theorem}{corollary}}
\AddToHook{env/lemma/begin}{\crefalias{theorem}{lemma}}
\AddToHook{env/claim/begin}{\crefalias{theorem}{claim}}
\AddToHook{env/convention/begin}{\crefalias{theorem}{convention}}
\AddToHook{env/definition/begin}{\crefalias{theorem}{definition}}
\AddToHook{env/algorithm/begin}{\crefalias{theorem}{algorithm}}
\AddToHook{env/condition/begin}{\crefalias{theorem}{condition}}
\AddToHook{env/example/begin}{\crefalias{theorem}{example}}
\AddToHook{env/question/begin}{\crefalias{theorem}{question}}
\AddToHook{env/remark/begin}{\crefalias{theorem}{remark}}
\AddToHook{env/note/begin}{\crefalias{theorem}{note}}
    
\numberwithin{equation}{section}

\newcommand{\ZZ}{\mathbb{Z}}

\newcommand{\RR}{\mathbb{R}}

\newcommand{\FF}{\mathbb{F}}

\newcommand{\id}{\mathit{id}}
\newcommand{\isom}{\cong}
\newcommand{\htpy}{\simeq}

\newcommand{\del}{\partial}

\renewcommand{\emptyset}{\varnothing}

\renewcommand{\epsilon}{\varepsilon}

\DeclareMathOperator{\Hom}{Hom}

\DeclareMathOperator{\Tor}{Tor}

\DeclareMathOperator{\qdeg}{qdeg}

\DeclareMathOperator{\Cone}{Cone}

\DeclareMathOperator{\Cob}{Cob}
\DeclareMathOperator{\Mat}{Mat}
\DeclareMathOperator{\Kom}{Kom}
\DeclareMathOperator{\Kob}{Kob}
\DeclareMathOperator{\Diag}{Diag}
\DeclareMathOperator{\GrMod}{GrMod}
\DeclareMathOperator{\Wh}{Wh}

\newcommand{\CKh}{\mathit{CKh}}
\newcommand{\CKhI}{\mathit{CKhI}}
\newcommand{\Kh}{\mathit{Kh}}
\newcommand{\KhI}{\mathit{KhI}}

\newcommand{\ca}{\alpha}

\newcommand{\uca}{\overline{\ca}}
\newcommand{\dca}{\underline{\ca}}

\newcommand{\ud}{\overline{d}}
\newcommand{\dd}{\underline{d}}
\newcommand{\ds}{\underline{s}}
\newcommand{\us}{\overline{s}}

\newcommand{\lin}{\text{in}}
\newcommand{\lout}{\text{out}}

\renewcommand{\uca}{\overline{\ca}}

\renewcommand{\us}{\overline{s}}

\renewcommand{\sf}{\mathsf}
\renewcommand{\cal}{\mathcal}

\title{
    Involutive Khovanov homology and equivariant knots II
}

\author{Taketo Sano}

\date{August 7, 2026}

\newcommand{\addresses}{{
  \bigskip
  \noindent
  Taketo Sano, \textsc{RIKEN iTHEMS, Wako, Saitama 351-0198, Japan }\par\nopagebreak
  \noindent
  \textit{E-mail address}: \url{taketo.sano@riken.jp}
}}

\begin{document}
    \maketitle
    \begin{abstract}
    In the spirit of Bar-Natan's formulation of Khovanov homology for tangles, we extend the framework of involutive Khovanov homology to \textit{involutive tangles}. This enables a divide-and-conquer computation of involutive Khovanov homology and the equivariant Rasmussen invariant, which results in a significant speedup for the algorithmic computation. With this, we obtain new examples of strongly invertible knots for which no slice disk is smoothly isotopic rel boundary to its symmetric counterpart. In particular, we show that the Whitehead doubles of the pretzel knots $P(-3, 3, -3)$ and $P(-5, 5, -5)$ admit exotic pairs of slice disks.
\end{abstract}
    
    \setcounter{tocdepth}{2}
    
    \section{Introduction}
\label{sec:intro}

\textit{Strongly invertible knots} have been studied for decades; see \cite{Sak:1986} and references therein. Recent developments in the area have given rise to new directions in their research and applications, in particular in the context of exotic 4-manifolds and exotically knotted surfaces. These include equivariant refinements of classical concordance invariants \cite{BI:2022, AB:2024, MP:2023b, Pri:2022, Pri:2024, PF:2023, BC:2026, Zam:2026}, as well as equivariant invariants built from knot homology theories, namely Khovanov homology \cite{Wat:2017, Sna:2018, LS:2022a, HS:2024, Sano:2025b, BDM+:2026a}, knot Floer homology (grid homology) \cite{DMS:2023, Mal:2024, Par:2024, Fra:2026}, and instanton Floer homology \cite{ADM+:2023, DMT:2026}.

In \cite{Sano:2025b}, the author introduced an involutive version of Khovanov homology for strongly invertible knots (more generally, for \textit{involutive knots and links}), following the formalism of \textit{involutive knot Floer homology} \cite{HM:2017,DMS:2023}. 
We further defined the \textit{equivariant Rasmussen invariant}, a pair of integer-valued invariants $(\ds, \us)$ that obstructs \textit{isotopy-equivariant sliceness} in the $4$-ball $D^4$. That is, if a strongly invertible slice knot $K$ has non-trivial equivariant Rasmussen invariant (i.e.\ $(\ds(K), \us(K)) \neq (0, 0)$), then we may immediately conclude that $K$ is not isotopy-equivariantly slice in $D^4$; in other words, \textit{no} slice disk $D$ of $K$ is smoothly isotopic rel $K$ to its symmetric counterpart $D'$ (\cite[Corollary 1.11]{Sano:2025b}). 

In that paper, we improved the program \verb|yui| \cite{Sano:YUI} so that it computes these invariants, and performed direct computation to show that the three strongly invertible slice knots appearing in Hayden and Sundberg's paper \cite{HS:2024} (namely, $m(9_{46})$, $15n_{103488}$, and $17nh_{73}$) have non-trivial equivariant Rasmussen invariants, $(\ds, \us) = (0, 2)$. We expected more interesting examples to be found by direct computation, but at that time the computational capacity was limited. 

In this paper, with the aim of making involutive Khovanov homology and the equivariant Rasmussen invariant more computable, we extend the framework to \textit{involutive tangles}. This extension is based on Bar-Natan's reformulation of Khovanov homology via tangles and cobordisms \cite{BN:2005}, which both clarifies the construction and improves the efficiency of computing ordinary Khovanov homology and the Rasmussen invariant \cite{BN:2007,Sch:2021}. 

\subsection*{Motivating problem}

It is an open question whether a knot $K$ being (smoothly) slice is equivalent to its (positive, untwisted) Whitehead double $\Wh(K)$ being slice \cite[Problem 1.38]{Kir:1997}. In \cite[Conjecture 1.2]{GHK+:2023}, the authors further ask whether the (smooth) isotopy class of a slice disk is determined by that of its Whitehead double. It is natural to ask an equivariant or isotopy-equivariant analogue of these questions. Here, we focus on the latter. Throughout, $\Wh$ denotes the positive untwisted Whitehead doubling operation, and for a strongly invertible knot $K$ with strong inversion $\tau$ of $S^3$, we also denote its obvious extension to $D^4$ by $\tau$. 

\begin{question}
    For a strongly invertible slice knot $K$, is $K$ isotopy-equivariantly slice if and only if $\Wh(K)$ is so?
\end{question}

\begin{question}
\label{q:wh-slice-disks}
    For a slice disk $D$ of a strongly invertible slice knot $K$, are $D$ and $\tau D$ isotopic if and only if $\Wh(D)$ and $\tau \Wh(D)$ are so?
\end{question}

Here, we give evidence supporting affirmative answers to these questions. Using the program \cite{Sano:YUI}, which we improved based on the algorithm introduced in this paper, we obtain the following result by direct computation. 

\begin{maintheorem}
\label{thm:wh-pretzel-exotic}
    Let $K$ be the (strongly invertible slice) pretzel knot $P(-3, 3, -3)$ or $P(-5, 5, -5)$, and let $D$ be any slice disk of $K$. Then $K$ is not isotopy-equivariantly slice, and in particular the two slice disks $D$ and $\tau D$ are not smoothly isotopic rel $K$. Furthermore, the Whitehead double $\Wh(K)$ is not isotopy-equivariantly slice, and in particular the slice disks $\Wh(D)$ and $\tau \Wh(D)$ are not smoothly isotopic rel $\Wh(K)$; in fact, they are topologically isotopic rel $\Wh(K)$ and hence form an exotic pair. 
\end{maintheorem}

Note that $P(-3, 3, -3)$ is the knot $m(9_{46})$, and the corresponding statement is a special case of \cite[Theorem 7.8]{DMT:2026}. To the best of our knowledge, the result for $P(-5, 5, -5)$ has not appeared in the literature.

\subsection*{Involutive Khovanov homology via tangles}

Here, we briefly sketch the reformulation of involutive Khovanov homology. Let $\tau$ denote the $180$-degree rotation about the $y$-axis, and let $B \subset \partial D^2$ be a $\tau$-invariant set of an even number of points, i.e.\ $\tau B = B$. An \textit{involutive tangle} $T$ is a tangle in the $3$-ball $D^3$ with $\partial T = B$ satisfying $\tau(T) = T$. An \textit{involutive tangle diagram} is a tangle diagram $T$ with $\partial T = B$ satisfying $\tau T = T$. In both cases, we ignore orientation; that is, $\tau$ is not required to preserve the orientation of $T$. (In particular, if $\tau$ \textit{reverses} the orientation, $T$ is called \textit{strongly invertible}.) For such $T$, let $[T]$ denote the formal Khovanov complex of $T$, considered over $\FF_2$. Then $[T]$ is naturally equipped with an isomorphism $i\colon [T] \to \tau [T]$ such that $\tau(i) = i^{-1}$. We call such a pair a \textit{$\tau$-equivariant complex}, which is an object of the \textit{category of $\tau$-equivariant complexes} $\Kob_{\bullet/l}^\tau(B)$. In particular, $[T]^\tau := ([T], i)$ is called the \textit{$\tau$-equivariant Khovanov complex} of $T$.

\begin{maintheorem}
\label{thm:inv-tang-invariance}
    For an involutive tangle diagram $T$ with $\partial T = B$, the chain homotopy type of the $\tau$-equivariant Khovanov complex $[T]^\tau$ is invariant under the involutive Reidemeister moves, and also the $I$-move when $B = \emptyset$. 
\end{maintheorem}

Here, the \textit{involutive Reidemeister moves} (see \Cref{fig:inv-reidemeister}), introduced by Lobb and Watson \cite{LW:2021}, are the equivariant analogues of the ordinary Reidemeister moves, and the $I$-move (see \Cref{fig:I-move}), introduced by Borodzik, Dai, Mallick and Stoffregen \cite{BDM+:2026b}, is a non-local move that is additionally required when the ambient space is $S^3$ rather than $\RR^3$. \Cref{thm:inv-tang-invariance} implies that the chain homotopy type of $[T]^\tau$ is an invariant of the involutive tangle. In particular, when $B = \emptyset$, $T = L$ is an involutive link diagram. In this case, there is an induced involution $I_\tau$ on the complex $[L]$. We define 
\[
    \cal{Q}(L) := \Cone([L] \xrightarrow{I - I_\tau} [L])
\]
which is a complex in the (non-$\tau$-equivariant) category $\Kob_{\bullet/l}(\emptyset)$, and call it the \textit{formal involutive Khovanov complex} of $L$. We shall prove that the correspondence $[L]^\tau \mapsto \cal{Q}(L)$ preserves chain homotopy equivalences, and hence obtain 

\begin{maintheorem}
\label{thm:inv-kh-invariance}
    For an involutive link diagram $L$, the chain homotopy type of the formal involutive Khovanov complex $\cal{Q}(L)$ is invariant under the involutive Reidemeister moves and the $I$-move. 
\end{maintheorem}

After applying the TQFT $\cal{F}$ for ordinary Khovanov homology over $\FF_2$, we recover the involutive Khovanov complex $\CKhI(L)$ defined in \cite{Sano:2025b}, and its invariance follows for free. (\cite[Theorem 2]{Sano:2025b} did not account for the $I$-move. See \Cref{rem:I-move}.)

\subsection*{Symmetric tangle decomposition}

One advantage of this reformulation is that it allows a \textit{divide-and-conquer} computation for the involutive Khovanov homology $\KhI(L)$. Take a \textit{symmetric tangle decomposition} of $L$, that is, write $L = D(T, T', \tau T')$, where $T$ is an involutive tangle on the axis, $T'$ is an off-axis tangle with counterpart $\tau T'$, and $D$ is a symmetric planar arc diagram along which the three are glued (see \Cref{fig:planar-diagram} in \Cref{sec:reducing}). With such a description, we may apply equivariant versions of \textit{delooping} and \textit{Gaussian elimination}, originally introduced in \cite{BN:2007}, to the $\tau$-equivariant complex $[L]^\tau$. This provides an effective method for computing $\KhI(L)$, both by hand and by computer. 

This also lets us apply the \textit{diagrammatic approach to the Rasmussen invariant}, introduced in \cite{KimSano:2025}, in the involutive setting. As with the computation of the (ordinary) Rasmussen invariant for pretzel knots in \cite[Theorem 3]{KimSano:2025}, we obtain the following without computer assistance. 

\begin{maintheorem}
\label{thm:pretzel-equiv-s}
    For any odd positive integer $p \geq 3$, the strongly invertible pretzel knot $K = P(-p, p, -p)$ has 
    \[
        (\ds(K), \us(K)) = (0, 2).
    \]
\end{maintheorem}

\subsection*{Direct computations}

The program \verb|yui| \cite{Sano:YUI} has been improved so that it now computes the involutive Khovanov homology and the equivariant Rasmussen invariant using the computation algorithm described above. The improvement is significant: the naive computation of $\KhI$ for the positron knot $17nh_{73}$ takes about 10 minutes on a local machine, whereas the new implementation finishes in 200 msec. This enables us to carry out more extensive computational experiments. 

In \cite{Lam:2025}, Lamm gives a table of symmetric diagrams for all strongly invertible knots up to 10 crossings. Moreover, Appendix B therein lists strongly invertible knots that have two equivalence classes of strong inversions, described in the form of doubly transvergent diagrams.\footnote{
    \cite[Appendix B]{Lam:2025} does not cover all knots that have two equivalence classes of strong inversions. Moreover, for some knots, the two axes of a doubly transvergent diagram give the same class. 
} 
We refer to these diagrams as the \textit{dataset of \cite{Lam:2025}}. Transforming these data into PD-codes, we performed direct computation using our program \verb|yui|. From \cite[Proposition 1.3]{Sano:2025b} we have 
\[
    (\ds(-K), \us(-K)) = (-\us(K), -\ds(K)),
\]
where $-K$ denotes the concordance inverse (mirror-reverse) of $K$, so it is enough to compute one from each mirror pair. For knots that have two strong inversions, such as $8_{21}$, we distinguish them as $8_{21a}$ and $8_{21b}$. We let $s$ denote the $\FF_2$-Rasmussen invariant. We say that the equivariant Rasmussen invariant of $K$ is \textit{anomalous} if $(\ds(K), \us(K)) \neq (s(K), s(K))$.

\begin{table}[t]
    \centering
    \caption{Knots with anomalous equivariant Rasmussen invariant}
    \label{tab:anomalous}
    \begin{tabular}[t]{lcc}
        \toprule
        $K$ & $(\ds, \us)$ & $s$ \\
        \midrule
        $8_{21b}$   & $(2, 4)$  & $2$ \\
        $9_{46a}$   & $(-2, 0)$ & $0$ \\
        $9_{46b}$   & $(-2, 0)$ & $0$ \\
        $10_{141b}$ & $(0, 2)$  & $0$ \\
        \bottomrule
    \end{tabular}
    \qquad
    \begin{tabular}[t]{lcc}
        \toprule
        $K$ & $(\ds, \us)$ & $s$ \\
        \midrule
        $10_{144b}$ & $(2, 4)$   & $2$ \\
        $10_{160}$  & $(2, 4)$   & $4$ \\
        $10_{162}$  & $(-4, -2)$ & $-2$ \\
        $10_{165}$  & $(-4, -2)$ & $-2$ \\
        \bottomrule
    \end{tabular}
\end{table}

\begin{proposition}
\label{prop:experiment1}
    Within the dataset of \cite{Lam:2025}, the knots listed in \Cref{tab:anomalous} are the only strongly invertible knots with anomalous equivariant Rasmussen invariant.
\end{proposition}

Next, we consider strongly invertible \textit{slice} knots. For a knot $K$ with two strong inversions $K_a = (K, \tau_a)$ and $K_b = (K, \tau_b)$, we obtain a strongly invertible slice knot by the equivariant connected sum 
\[
    K_a \# -K_b := (K \# -K, \tau_a \# -\tau_b).
\]

\begin{proposition}
\label{prop:experiment2}
    Among the knots in the dataset of \cite{Lam:2025} that have two strong inversions, the following are the only $K_a \# -K_b$ having non-trivial equivariant Rasmussen invariant, all equal to $(-2, 0)$.
    \[
        8_{21a} \# -8_{21b},\quad 
        10_{141a} \# -10_{141b},\quad
        10_{144a} \# -10_{144b}.
    \]
\end{proposition}

There is another interesting way of producing a strongly invertible slice knot which appears in \cite{DMS:2023,GHK+:2023,DMT:2026}. For a strongly invertible knot $(K, \tau)$, consider the slice knot
\[
    K \# K \# (-K \# -K)
\]
equipped with a strong inversion
\[
    \tau \# \tau \# (-\tau_f)
\]
where $\tau_f$ is the flipping involution on $K \# K$ that interchanges the two factors (see \Cref{fig:flip-const}). We call this the \textit{DMS construction}. 

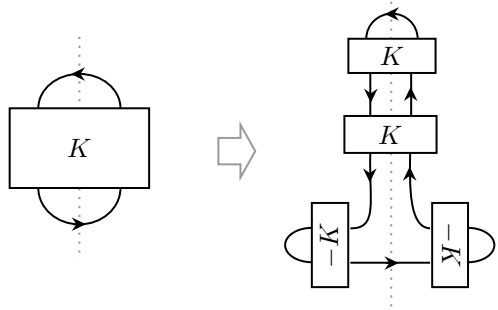
\begin{figure}[h]
    \begin{center}
        \tikzset{every picture/.style={line width=0.75pt}} 

\begin{tikzpicture}[x=0.75pt,y=0.75pt,yscale=-1,xscale=1]

\draw [color={rgb, 255:red, 155; green, 155; blue, 155 }  ,draw opacity=1 ][line width=0.75]  [dash pattern={on 0.84pt off 2.51pt}]  (205,2) -- (205,156.75) ;
\draw  [fill={rgb, 255:red, 255; green, 255; blue, 255 }  ,fill opacity=1 ] (183.5,20.6) -- (227.25,20.6) -- (227.25,37.77) -- (183.5,37.77) -- cycle ;

\draw  [draw opacity=0] (192.48,20.49) .. controls (192.55,13.58) and (198.29,8) .. (205.38,8) .. controls (212.45,8) and (218.19,13.57) .. (218.26,20.48) -- (205.38,20.6) -- cycle ; \draw   (192.48,20.49) .. controls (192.55,13.58) and (198.29,8) .. (205.38,8) .. controls (212.45,8) and (218.19,13.57) .. (218.26,20.48) ;  
\draw  [fill={rgb, 255:red, 255; green, 255; blue, 255 }  ,fill opacity=1 ] (181.5,59.35) -- (227.75,59.35) -- (227.75,77.9) -- (181.5,77.9) -- cycle ;

\draw  [fill={rgb, 255:red, 255; green, 255; blue, 255 }  ,fill opacity=1 ] (165.16,145.25) -- (165.16,103) -- (183.42,103) -- (183.42,145.25) -- cycle ;

\draw  [draw opacity=0] (165.09,132.75) .. controls (157.71,132.72) and (151.75,128.87) .. (151.75,124.13) .. controls (151.75,119.38) and (157.71,115.53) .. (165.08,115.5) -- (165.16,124.13) -- cycle ; \draw   (165.09,132.75) .. controls (157.71,132.72) and (151.75,128.87) .. (151.75,124.13) .. controls (151.75,119.38) and (157.71,115.53) .. (165.08,115.5) ;  
\draw  [fill={rgb, 255:red, 255; green, 255; blue, 255 }  ,fill opacity=1 ] (243.54,103) -- (243.54,145.25) -- (225.55,145.25) -- (225.55,103) -- cycle ;

\draw  [draw opacity=0] (243.61,115.5) .. controls (250.88,115.53) and (256.75,119.38) .. (256.75,124.13) .. controls (256.75,128.87) and (250.88,132.72) .. (243.63,132.75) -- (243.54,124.13) -- cycle ; \draw   (243.61,115.5) .. controls (250.88,115.53) and (256.75,119.38) .. (256.75,124.13) .. controls (256.75,128.87) and (250.88,132.72) .. (243.63,132.75) ;  
\draw    (184.5,133) -- (224.5,133) ;
\draw    (184.5,115.74) .. controls (198.75,116.06) and (193.75,101.59) .. (194.5,78) ;
\draw    (224.5,115.74) .. controls (213.75,115.43) and (214.25,94.04) .. (214.5,78) ;

\draw    (194.5,59) -- (194.5,38) ;
\draw    (215,59) -- (215,38) ;
\draw    (202.25,133.15) -- (205.75,133.15) ;
\draw [shift={(208.75,133.15)}, rotate = 180] [fill={rgb, 255:red, 0; green, 0; blue, 0 }  ][line width=0.08]  [draw opacity=0] (7.14,-3.43) -- (0,0) -- (7.14,3.43) -- (4.74,0) -- cycle    ;
\draw    (207.75,8.15) -- (204.75,8.15) ;
\draw [shift={(201.75,8.15)}, rotate = 360] [fill={rgb, 255:red, 0; green, 0; blue, 0 }  ][line width=0.08]  [draw opacity=0] (7.14,-3.43) -- (0,0) -- (7.14,3.43) -- (4.74,0) -- cycle    ;
\draw    (214.25,89.65) -- (214.25,87.75) ;
\draw [shift={(214.25,84.75)}, rotate = 90] [fill={rgb, 255:red, 0; green, 0; blue, 0 }  ][line width=0.08]  [draw opacity=0] (7.14,-3.43) -- (0,0) -- (7.14,3.43) -- (4.74,0) -- cycle    ;
\draw    (194.25,88.25) -- (194.25,89.75) ;
\draw [shift={(194.25,92.75)}, rotate = 270] [fill={rgb, 255:red, 0; green, 0; blue, 0 }  ][line width=0.08]  [draw opacity=0] (7.14,-3.43) -- (0,0) -- (7.14,3.43) -- (4.74,0) -- cycle    ;
\draw    (214.75,49.65) -- (214.75,47.75) ;
\draw [shift={(214.75,44.75)}, rotate = 90] [fill={rgb, 255:red, 0; green, 0; blue, 0 }  ][line width=0.08]  [draw opacity=0] (7.14,-3.43) -- (0,0) -- (7.14,3.43) -- (4.74,0) -- cycle    ;
\draw    (194.75,48.25) -- (194.75,49.75) ;
\draw [shift={(194.75,52.75)}, rotate = 270] [fill={rgb, 255:red, 0; green, 0; blue, 0 }  ][line width=0.08]  [draw opacity=0] (7.14,-3.43) -- (0,0) -- (7.14,3.43) -- (4.74,0) -- cycle    ;

\draw [color={rgb, 255:red, 155; green, 155; blue, 155 }  ,draw opacity=1 ][line width=0.75]  [dash pattern={on 0.84pt off 2.51pt}]  (48.5,19.62) -- (48.5,131.38) ;
\draw  [fill={rgb, 255:red, 255; green, 255; blue, 255 }  ,fill opacity=1 ] (13.5,55.5) -- (83.5,55.5) -- (83.5,95.5) -- (13.5,95.5) -- cycle ;

\draw  [draw opacity=0] (27.88,55.32) .. controls (27.99,45.88) and (37.18,38.25) .. (48.5,38.25) .. controls (59.81,38.25) and (68.99,45.86) .. (69.12,55.29) -- (48.5,55.5) -- cycle ; \draw   (27.88,55.32) .. controls (27.99,45.88) and (37.18,38.25) .. (48.5,38.25) .. controls (59.81,38.25) and (68.99,45.86) .. (69.12,55.29) ;  
\draw    (51.25,38.32) -- (47.25,38.32) ;
\draw [shift={(44.25,38.32)}, rotate = 360] [fill={rgb, 255:red, 0; green, 0; blue, 0 }  ][line width=0.08]  [draw opacity=0] (7.14,-3.43) -- (0,0) -- (7.14,3.43) -- (4.74,0) -- cycle    ;

\draw  [draw opacity=0] (69.12,95.68) .. controls (69.02,105.67) and (59.82,113.75) .. (48.5,113.75) .. controls (37.19,113.75) and (28,105.69) .. (27.88,95.71) -- (48.5,95.5) -- cycle ; \draw   (69.12,95.68) .. controls (69.02,105.67) and (59.82,113.75) .. (48.5,113.75) .. controls (37.19,113.75) and (28,105.69) .. (27.88,95.71) ;  
\draw    (45.25,113.67) -- (48.75,113.67) ;
\draw [shift={(51.75,113.67)}, rotate = 180] [fill={rgb, 255:red, 0; green, 0; blue, 0 }  ][line width=0.08]  [draw opacity=0] (7.14,-3.43) -- (0,0) -- (7.14,3.43) -- (4.74,0) -- cycle    ;

\draw  [color={rgb, 255:red, 155; green, 155; blue, 155 }  ,draw opacity=1 ] (118.25,70.38) -- (129.35,70.38) -- (129.35,63.75) -- (136.75,77) -- (129.35,90.25) -- (129.35,83.63) -- (118.25,83.63) -- cycle ;

\draw (48.5,75.5) node    {$K$};
\draw (205.38,29.19) node    {$K$};
\draw (204.63,68.63) node    {$K$};
\draw (174.29,124.13) node  [rotate=-270]  {$-K$};
\draw (234.55,124.13) node  [rotate=-90]  {$-K$};

\end{tikzpicture}
    \end{center}
    \caption{DMS construction}
    \label{fig:flip-const}
\end{figure}

\begin{proposition}
\label{prop:experiment3}
    Within the dataset of \cite{Lam:2025}, the eight strongly invertible knots listed in \Cref{tab:anomalous} are the only ones for which the corresponding DMS construction gives a strongly invertible slice knot with non-trivial equivariant Rasmussen invariant; these invariants are each equal to either $(-2, 0)$ or $(0, 2)$.
\end{proposition}

\begin{question}
    For a strongly invertible knot $K$ with anomalous equivariant Rasmussen invariant, does the corresponding DMS construction give a strongly invertible slice knot with non-trivial equivariant Rasmussen invariant?
\end{question}

Finally, we consider Whitehead doubles. Since applying $\Wh$ to a diagram at least quadruples the number of crossings, computations are out of reach for most of the strongly invertible slice knots obtained above with non-trivial equivariant Rasmussen invariant. Nonetheless, we managed to obtain the following:

\begin{proposition}
\label{prop:experiment4}
    The Whitehead doubles of the (slice) pretzel knots $P(-3, 3, -3)$ and $P(-5, 5, -5)$ have non-trivial equivariant Rasmussen invariant, both equal to $(0, 2)$.
\end{proposition}

Note that for a general strongly invertible knot $K$, there are two ways of forming its Whitehead double corresponding to the two fixed points of $\tau$ on $K$ (see \cite[Figure 4]{DMT:2026}); but for the two pretzel knots, they are equivalent. The diagram of $\Wh(P(-5, 5, -5))$ has $72$ crossings, and the computation for the equivariant Rasmussen invariant was performed on a 64-core machine with 3 TB of memory, finishing in 17 hours. 

\Cref{thm:wh-pretzel-exotic} follows by combining \Cref{thm:pretzel-equiv-s,prop:experiment4} with previously known results. First, a \textit{$\ZZ$-slice disk} $D \subset D^4$ for a knot $K \subset S^3$ is a slice disk such that $\pi_1(D^4 \setminus D) \isom \ZZ$, and a knot is \textit{$\ZZ$-slice} if it possesses a $\ZZ$-slice disk. From \cite[Proposition 2.1]{GHK+:2023}, if $D$ is a slice disk of a knot $K$, then the doubled disk $\Wh(D)$ is a $\ZZ$-slice disk of the Whitehead double $\Wh(K)$; and from \cite[Theorem 1.2]{CP:2021} and \cite[Theorem 2.2]{Hay:2021}, any two $\ZZ$-slice disks are topologically isotopic rel boundary. This completes the proof of \Cref{thm:wh-pretzel-exotic}. 

\begin{question}
    For every pretzel knot $P(-p, p, -p)$ with $p \geq 3$ odd, does its Whitehead double have non-trivial equivariant Rasmussen invariant? 
\end{question}

\begin{question}
\label{q:wh-ssi}
    More strongly, if the equivariant Rasmussen invariant of a strongly invertible slice knot $K$ is non-trivial, is that of $\Wh(K)$ also non-trivial?
\end{question}

If \Cref{q:wh-ssi} has a positive answer, then every strongly invertible slice knot with non-trivial equivariant Rasmussen invariant gives further evidence supporting \Cref{q:wh-slice-disks}. 

\subsection*{Future directions}

In the above examples, every non-trivial equivariant Rasmussen invariant of a strongly invertible slice knot was either $(-2, 0)$ or $(0, 2)$. One may ask whether there is an example $K$ with $|\ds(K)|$ or $|\us(K)|$ greater than $2$. The answer is positive; indeed, the strongly invertible knots appearing in \cite{Hay:2023} provide such examples. This direction will be pursued in future work. 

\subsection*{Organization} 

The paper is organized as follows. In \Cref{sec:kh-tangles}, we review Bar-Natan's formulation of Khovanov homology for tangles, together with the delooping and Gaussian elimination techniques. In \Cref{sec:inv-kh}, we introduce the $\tau$-equivariant Khovanov complex for involutive tangles, prove its invariance under the involutive Reidemeister moves and the $I$-move (\Cref{thm:inv-tang-invariance}), and construct the involutive Khovanov complex for involutive links (\Cref{thm:inv-kh-invariance}), relating it with the construction of \cite{Sano:2025b}. In \Cref{sec:reducing}, we develop the equivariant versions of delooping and Gaussian elimination, and describe the divide-and-conquer procedure that computes the involutive Khovanov homology from a symmetric tangle decomposition. In \Cref{sec:equiv-rasmussen}, we review the definition of the equivariant Rasmussen invariant via the equivariant Lee class, and explain how it can be computed effectively from a linear system. Finally, in \Cref{sec:pen-paper-computation}, we prove \Cref{thm:pretzel-equiv-s} by illustrating the method through explicit diagrammatic computations for pretzel knots.
    
\subsection*{AI declaration} 

In the preparation of this paper, generative AI was used only to improve the English writing and to check for errors; the author takes full responsibility for all mathematical content. In the development of the program \verb|yui| \cite{Sano:YUI}, Claude Code (Opus and Fable) was used extensively to improve the computational efficiency and to carry out the computational experiments; the author has reviewed all of the generated code and takes full responsibility for its correctness. 
    \subsection*{Acknowledgements}
The author thanks Masaki Taniguchi for the helpful discussions. 
The author was supported by JSPS KAKENHI Grant Numbers 23K12982, RIKEN iTHEMS Program and academist crowdfunding.

    \section{Khovanov homology for tangles}
\label{sec:kh-tangles}

First, we briefly review Bar-Natan's reformulation of Khovanov homology for tangles \cite{BN:2005}. 

\subsection{Dotted cobordisms}

For any set $B$ of an even number of points on the boundary $\del D^2$ of the unit disk $D^2$, the preadditive category of \textit{dotted cobordisms} $\Cob^3_\bullet(B)$ is defined as follows: an object of $\Cob^3_\bullet(B)$ is an unoriented compact $1$-manifold $X$ properly embedded in $D^2$ with $\partial X = B$. A morphism $f: X \rightarrow Y$ in $\Cob^3_\bullet(B)$ is a $\ZZ$-linear combination of boundary-fixing isotopy classes of dotted cobordisms from $X$ to $Y$, where a \textit{dotted cobordism} is an unoriented compact $2$-manifold $S$ properly embedded in $D^2 \times I$, possibly carrying dots on some of its components, with boundary $\del S = X \times \{0\} \cup B \times I \cup Y \times \{1\}$. The \textit{$q$-degree} of a dotted cobordism $S$ with $k$ dots is defined by
\[
    \qdeg S = \chi(S) - \frac{1}{2} |B| - 2k
\]
where $\chi$ denotes the Euler characteristic.
The following \textit{local relations} are imposed on the hom-sets of $\Cob^3_\bullet(B)$, and the resulting quotient category is denoted $\Cob^3_{\bullet/l}(B)$:
\begin{center}
    \input{tikzpictures/loc-relation-dot}
\end{center}
Let $h, t$ denote the following closed dotted cobordisms:
\begin{center}
    \tikzset{every picture/.style={line width=0.75pt}} 

\begin{tikzpicture}[x=0.75pt,y=0.75pt,yscale=-1,xscale=1]

\draw  [draw opacity=0] (89.94,40.75) .. controls (89.29,43) and (82.45,44.76) .. (74.11,44.76) .. controls (65.35,44.76) and (58.24,42.81) .. (58.24,40.41) .. controls (58.24,40.26) and (58.27,40.11) .. (58.32,39.97) -- (74.11,40.41) -- cycle ; \draw  [color={rgb, 255:red, 128; green, 128; blue, 128 }  ,draw opacity=1 ] (89.94,40.75) .. controls (89.29,43) and (82.45,44.76) .. (74.11,44.76) .. controls (65.35,44.76) and (58.24,42.81) .. (58.24,40.41) .. controls (58.24,40.26) and (58.27,40.11) .. (58.32,39.97) ;  
\draw  [draw opacity=0][dash pattern={on 0.84pt off 2.51pt}] (58.63,39.33) .. controls (59.9,37.27) and (66.43,35.7) .. (74.28,35.7) .. controls (82.93,35.7) and (89.96,37.6) .. (90.15,39.96) -- (74.28,40.06) -- cycle ; \draw  [color={rgb, 255:red, 128; green, 128; blue, 128 }  ,draw opacity=1 ][dash pattern={on 0.84pt off 2.51pt}] (58.63,39.33) .. controls (59.9,37.27) and (66.43,35.7) .. (74.28,35.7) .. controls (82.93,35.7) and (89.96,37.6) .. (90.15,39.96) ;  
\draw   (58.12,40.41) .. controls (58.12,31.57) and (65.28,24.41) .. (74.11,24.41) .. controls (82.95,24.41) and (90.11,31.57) .. (90.11,40.41) .. controls (90.11,49.24) and (82.95,56.4) .. (74.11,56.4) .. controls (65.28,56.4) and (58.12,49.24) .. (58.12,40.41) -- cycle ;
\draw  [fill={rgb, 255:red, 0; green, 0; blue, 0 }  ,fill opacity=1 ] (67.63,31.68) .. controls (67.63,30.55) and (68.55,29.63) .. (69.68,29.63) .. controls (70.82,29.63) and (71.73,30.55) .. (71.73,31.68) .. controls (71.73,32.81) and (70.82,33.73) .. (69.68,33.73) .. controls (68.55,33.73) and (67.63,32.81) .. (67.63,31.68) -- cycle ;
\draw  [fill={rgb, 255:red, 0; green, 0; blue, 0 }  ,fill opacity=1 ] (77.23,31.68) .. controls (77.23,30.55) and (78.15,29.63) .. (79.28,29.63) .. controls (80.42,29.63) and (81.33,30.55) .. (81.33,31.68) .. controls (81.33,32.81) and (80.42,33.73) .. (79.28,33.73) .. controls (78.15,33.73) and (77.23,32.81) .. (77.23,31.68) -- cycle ;
\draw  [draw opacity=0] (249.56,23.31) .. controls (249,25.25) and (243.11,26.77) .. (235.92,26.77) .. controls (228.36,26.77) and (222.23,25.09) .. (222.23,23.01) .. controls (222.23,22.88) and (222.26,22.76) .. (222.3,22.64) -- (235.92,23.01) -- cycle ; \draw  [color={rgb, 255:red, 128; green, 128; blue, 128 }  ,draw opacity=1 ] (249.56,23.31) .. controls (249,25.25) and (243.11,26.77) .. (235.92,26.77) .. controls (228.36,26.77) and (222.23,25.09) .. (222.23,23.01) .. controls (222.23,22.88) and (222.26,22.76) .. (222.3,22.64) ;  
\draw  [draw opacity=0][dash pattern={on 0.84pt off 2.51pt}] (222.57,22.08) .. controls (223.67,20.31) and (229.29,18.95) .. (236.06,18.95) .. controls (243.51,18.95) and (249.57,20.59) .. (249.75,22.62) -- (236.06,22.71) -- cycle ; \draw  [color={rgb, 255:red, 128; green, 128; blue, 128 }  ,draw opacity=1 ][dash pattern={on 0.84pt off 2.51pt}] (222.57,22.08) .. controls (223.67,20.31) and (229.29,18.95) .. (236.06,18.95) .. controls (243.51,18.95) and (249.57,20.59) .. (249.75,22.62) ;  
\draw   (222.13,23.01) .. controls (222.13,15.4) and (228.3,9.22) .. (235.92,9.22) .. controls (243.53,9.22) and (249.71,15.4) .. (249.71,23.01) .. controls (249.71,30.63) and (243.53,36.8) .. (235.92,36.8) .. controls (228.3,36.8) and (222.13,30.63) .. (222.13,23.01) -- cycle ;
\draw  [fill={rgb, 255:red, 0; green, 0; blue, 0 }  ,fill opacity=1 ] (230.33,15.49) .. controls (230.33,14.51) and (231.12,13.72) .. (232.1,13.72) .. controls (233.07,13.72) and (233.87,14.51) .. (233.87,15.49) .. controls (233.87,16.47) and (233.07,17.26) .. (232.1,17.26) .. controls (231.12,17.26) and (230.33,16.47) .. (230.33,15.49) -- cycle ;
\draw  [fill={rgb, 255:red, 0; green, 0; blue, 0 }  ,fill opacity=1 ] (238.61,15.49) .. controls (238.61,14.51) and (239.4,13.72) .. (240.37,13.72) .. controls (241.35,13.72) and (242.14,14.51) .. (242.14,15.49) .. controls (242.14,16.47) and (241.35,17.26) .. (240.37,17.26) .. controls (239.4,17.26) and (238.61,16.47) .. (238.61,15.49) -- cycle ;

\draw  [draw opacity=0] (249.56,60.11) .. controls (249,62.05) and (243.11,63.57) .. (235.92,63.57) .. controls (228.36,63.57) and (222.23,61.89) .. (222.23,59.81) .. controls (222.23,59.68) and (222.26,59.56) .. (222.3,59.44) -- (235.92,59.81) -- cycle ; \draw  [color={rgb, 255:red, 128; green, 128; blue, 128 }  ,draw opacity=1 ] (249.56,60.11) .. controls (249,62.05) and (243.11,63.57) .. (235.92,63.57) .. controls (228.36,63.57) and (222.23,61.89) .. (222.23,59.81) .. controls (222.23,59.68) and (222.26,59.56) .. (222.3,59.44) ;  
\draw  [draw opacity=0][dash pattern={on 0.84pt off 2.51pt}] (222.57,58.88) .. controls (223.67,57.11) and (229.29,55.75) .. (236.06,55.75) .. controls (243.51,55.75) and (249.57,57.39) .. (249.75,59.42) -- (236.06,59.51) -- cycle ; \draw  [color={rgb, 255:red, 128; green, 128; blue, 128 }  ,draw opacity=1 ][dash pattern={on 0.84pt off 2.51pt}] (222.57,58.88) .. controls (223.67,57.11) and (229.29,55.75) .. (236.06,55.75) .. controls (243.51,55.75) and (249.57,57.39) .. (249.75,59.42) ;  
\draw   (222.13,59.81) .. controls (222.13,52.2) and (228.3,46.02) .. (235.92,46.02) .. controls (243.53,46.02) and (249.71,52.2) .. (249.71,59.81) .. controls (249.71,67.43) and (243.53,73.6) .. (235.92,73.6) .. controls (228.3,73.6) and (222.13,67.43) .. (222.13,59.81) -- cycle ;
\draw  [fill={rgb, 255:red, 0; green, 0; blue, 0 }  ,fill opacity=1 ] (230.33,52.29) .. controls (230.33,51.31) and (231.12,50.52) .. (232.1,50.52) .. controls (233.07,50.52) and (233.87,51.31) .. (233.87,52.29) .. controls (233.87,53.27) and (233.07,54.06) .. (232.1,54.06) .. controls (231.12,54.06) and (230.33,53.27) .. (230.33,52.29) -- cycle ;
\draw  [fill={rgb, 255:red, 0; green, 0; blue, 0 }  ,fill opacity=1 ] (238.61,52.29) .. controls (238.61,51.31) and (239.4,50.52) .. (240.37,50.52) .. controls (241.35,50.52) and (242.14,51.31) .. (242.14,52.29) .. controls (242.14,53.27) and (241.35,54.06) .. (240.37,54.06) .. controls (239.4,54.06) and (238.61,53.27) .. (238.61,52.29) -- cycle ;

\draw  [draw opacity=0] (195.94,39.76) .. controls (195.29,42) and (188.45,43.77) .. (180.11,43.77) .. controls (171.35,43.77) and (164.24,41.82) .. (164.24,39.41) .. controls (164.24,39.26) and (164.27,39.12) .. (164.32,38.98) -- (180.11,39.41) -- cycle ; \draw  [color={rgb, 255:red, 128; green, 128; blue, 128 }  ,draw opacity=1 ] (195.94,39.76) .. controls (195.29,42) and (188.45,43.77) .. (180.11,43.77) .. controls (171.35,43.77) and (164.24,41.82) .. (164.24,39.41) .. controls (164.24,39.26) and (164.27,39.12) .. (164.32,38.98) ;  
\draw  [draw opacity=0][dash pattern={on 0.84pt off 2.51pt}] (164.63,38.33) .. controls (165.9,36.27) and (172.43,34.71) .. (180.28,34.71) .. controls (188.93,34.71) and (195.96,36.6) .. (196.15,38.96) -- (180.28,39.06) -- cycle ; \draw  [color={rgb, 255:red, 128; green, 128; blue, 128 }  ,draw opacity=1 ][dash pattern={on 0.84pt off 2.51pt}] (164.63,38.33) .. controls (165.9,36.27) and (172.43,34.71) .. (180.28,34.71) .. controls (188.93,34.71) and (195.96,36.6) .. (196.15,38.96) ;  
\draw   (164.12,39.41) .. controls (164.12,30.58) and (171.28,23.42) .. (180.11,23.42) .. controls (188.95,23.42) and (196.11,30.58) .. (196.11,39.41) .. controls (196.11,48.25) and (188.95,55.41) .. (180.11,55.41) .. controls (171.28,55.41) and (164.12,48.25) .. (164.12,39.41) -- cycle ;
\draw  [fill={rgb, 255:red, 0; green, 0; blue, 0 }  ,fill opacity=1 ] (170.63,30.49) .. controls (170.63,29.35) and (171.55,28.44) .. (172.68,28.44) .. controls (173.82,28.44) and (174.73,29.35) .. (174.73,30.49) .. controls (174.73,31.62) and (173.82,32.54) .. (172.68,32.54) .. controls (171.55,32.54) and (170.63,31.62) .. (170.63,30.49) -- cycle ;
\draw  [fill={rgb, 255:red, 0; green, 0; blue, 0 }  ,fill opacity=1 ] (178.23,30.49) .. controls (178.23,29.35) and (179.15,28.44) .. (180.28,28.44) .. controls (181.42,28.44) and (182.33,29.35) .. (182.33,30.49) .. controls (182.33,31.62) and (181.42,32.54) .. (180.28,32.54) .. controls (179.15,32.54) and (178.23,31.62) .. (178.23,30.49) -- cycle ;
\draw  [fill={rgb, 255:red, 0; green, 0; blue, 0 }  ,fill opacity=1 ] (185.83,30.49) .. controls (185.83,29.35) and (186.75,28.44) .. (187.88,28.44) .. controls (189.02,28.44) and (189.93,29.35) .. (189.93,30.49) .. controls (189.93,31.62) and (189.02,32.54) .. (187.88,32.54) .. controls (186.75,32.54) and (185.83,31.62) .. (185.83,30.49) -- cycle ;

\draw (204.5,30.9) node [anchor=north west][inner sep=0.75pt]    {$-$};
\draw (21,31.61) node [anchor=north west][inner sep=0.75pt]    {$h\ =\ $};
\draw (127,31.61) node [anchor=north west][inner sep=0.75pt]    {$t\ =\ $};
\draw (97,31.61) node [anchor=north west][inner sep=0.75pt]    {$,$};

\end{tikzpicture}
\end{center}
We have $\qdeg h = -2,\ \qdeg t = -4$. The relation (NC) implies the following relation:
\begin{center}
    \tikzset{every picture/.style={line width=0.75pt}} 

\begin{tikzpicture}[x=0.75pt,y=0.75pt,yscale=-.75,xscale=.75]

\draw  [dash pattern={on 0.84pt off 2.51pt}] (10.93,11.56) -- (56.56,11.56) -- (56.56,55.49) -- (10.93,55.49) -- cycle ;
\draw  [dash pattern={on 0.84pt off 2.51pt}] (122.41,11.06) -- (169.08,11.06) -- (169.08,55.99) -- (122.41,55.99) -- cycle ;
\draw  [dash pattern={on 0.84pt off 2.51pt}] (224.95,12.06) -- (269.54,12.06) -- (269.54,54.99) -- (224.95,54.99) -- cycle ;
\draw  [fill={rgb, 255:red, 0; green, 0; blue, 0 }  ,fill opacity=1 ] (23.5,34.2) .. controls (23.5,32.43) and (24.93,31) .. (26.7,31) .. controls (28.47,31) and (29.91,32.43) .. (29.91,34.2) .. controls (29.91,35.97) and (28.47,37.41) .. (26.7,37.41) .. controls (24.93,37.41) and (23.5,35.97) .. (23.5,34.2) -- cycle ;
\draw  [fill={rgb, 255:red, 0; green, 0; blue, 0 }  ,fill opacity=1 ] (36,34.2) .. controls (36,32.43) and (37.43,31) .. (39.2,31) .. controls (40.97,31) and (42.41,32.43) .. (42.41,34.2) .. controls (42.41,35.97) and (40.97,37.41) .. (39.2,37.41) .. controls (37.43,37.41) and (36,35.97) .. (36,34.2) -- cycle ;
\draw  [fill={rgb, 255:red, 0; green, 0; blue, 0 }  ,fill opacity=1 ] (142.54,33.52) .. controls (142.54,31.75) and (143.97,30.32) .. (145.74,30.32) .. controls (147.51,30.32) and (148.95,31.75) .. (148.95,33.52) .. controls (148.95,35.29) and (147.51,36.73) .. (145.74,36.73) .. controls (143.97,36.73) and (142.54,35.29) .. (142.54,33.52) -- cycle ;

\draw (68,25.32) node [anchor=north west][inner sep=0.75pt]    {$=$};
\draw (183,25.32) node [anchor=north west][inner sep=0.75pt]    {$+$};
\draw (103,25.32) node [anchor=north west][inner sep=0.75pt]    {$h$};
\draw (207,25.32) node [anchor=north west][inner sep=0.75pt]    {$t$};

\end{tikzpicture}
\end{center}
It is convenient to define the \textit{hollow dot} as
\begin{center}
    \tikzset{every picture/.style={line width=0.75pt}} 

\begin{tikzpicture}[x=0.75pt,y=0.75pt,yscale=-.75,xscale=.75]

\draw  [dash pattern={on 0.84pt off 2.51pt}] (10.93,11.56) -- (56.56,11.56) -- (56.56,55.49) -- (10.93,55.49) -- cycle ;
\draw  [dash pattern={on 0.84pt off 2.51pt}] (101.41,12.06) -- (148.08,12.06) -- (148.08,56.99) -- (101.41,56.99) -- cycle ;
\draw  [color={rgb, 255:red, 0; green, 0; blue, 0 }  ,draw opacity=1 ][fill={rgb, 255:red, 255; green, 255; blue, 255 }  ,fill opacity=1 ][line width=0.75]  (29.5,34.2) .. controls (29.5,32.43) and (30.93,31) .. (32.7,31) .. controls (34.47,31) and (35.91,32.43) .. (35.91,34.2) .. controls (35.91,35.97) and (34.47,37.41) .. (32.7,37.41) .. controls (30.93,37.41) and (29.5,35.97) .. (29.5,34.2) -- cycle ;
\draw  [fill={rgb, 255:red, 0; green, 0; blue, 0 }  ,fill opacity=1 ] (121.54,34.52) .. controls (121.54,32.75) and (122.97,31.32) .. (124.74,31.32) .. controls (126.51,31.32) and (127.95,32.75) .. (127.95,34.52) .. controls (127.95,36.29) and (126.51,37.73) .. (124.74,37.73) .. controls (122.97,37.73) and (121.54,36.29) .. (121.54,34.52) -- cycle ;
\draw  [dash pattern={on 0.84pt off 2.51pt}] (209.95,11.91) -- (254.54,11.91) -- (254.54,54.84) -- (209.95,54.84) -- cycle ;

\draw (68,25.32) node [anchor=north west][inner sep=0.75pt]    {$=$};
\draw (167,24.24) node [anchor=north west][inner sep=0.75pt]    {$-$};
\draw (190,24.17) node [anchor=north west][inner sep=0.75pt]    {$h$};

\end{tikzpicture}
\end{center}
Then (NC) can be rewritten as 
\begin{center}
    \input{tikzpictures/hollow-dot-NC}
\end{center}

When $B = \emptyset$, the \textit{tautological functor}
\[
    \Hom_{\bullet/l}(\emptyset, -)\colon
    \Cob^3_{\bullet/l}(\emptyset) \to \GrMod
\]
recovers the TQFT for the \textit{$U(2)$-equivariant Khovanov homology}, i.e.\ it gives rise to the graded ring 
\[
    \Hom_{\bullet/l}(\emptyset, \emptyset) \isom R = \ZZ[h, t] 
\]
and the graded Frobenius algebra
\[
    \Hom_{\bullet/l}(\emptyset, \bigcirc) \isom A = R[X]/(X^2 - hX - t).
\]
The two generators $1, X$ of $A$ as a rank 2 free $R$-module are identified with
\begin{center}
    \tikzset{every picture/.style={line width=0.75pt}} 

\begin{tikzpicture}[x=0.75pt,y=0.75pt,yscale=-1,xscale=1]

\draw  [draw opacity=0][dash pattern={on 0.84pt off 2.51pt}] (25.7,33.61) .. controls (23.49,31.96) and (21.86,26.52) .. (21.86,20.07) .. controls (21.86,13.63) and (23.48,8.2) .. (25.7,6.54) -- (27.13,20.07) -- cycle ; \draw  [dash pattern={on 0.84pt off 2.51pt}] (25.7,33.61) .. controls (23.49,31.96) and (21.86,26.52) .. (21.86,20.07) .. controls (21.86,13.63) and (23.48,8.2) .. (25.7,6.54) ;  
\draw  [draw opacity=0] (26.55,6.1) .. controls (26.74,6.04) and (26.93,6.02) .. (27.13,6.02) .. controls (30.04,6.02) and (32.4,12.31) .. (32.4,20.07) .. controls (32.4,27.84) and (30.04,34.13) .. (27.13,34.13) .. controls (26.93,34.13) and (26.74,34.11) .. (26.55,34.05) -- (27.13,20.07) -- cycle ; \draw   (26.55,6.1) .. controls (26.74,6.04) and (26.93,6.02) .. (27.13,6.02) .. controls (30.04,6.02) and (32.4,12.31) .. (32.4,20.07) .. controls (32.4,27.84) and (30.04,34.13) .. (27.13,34.13) .. controls (26.93,34.13) and (26.74,34.11) .. (26.55,34.05) ;  
\draw  [draw opacity=0] (27.71,6.02) .. controls (27.51,6.02) and (27.32,6.02) .. (27.13,6.02) .. controls (18,6.02) and (10.61,12.31) .. (10.61,20.07) .. controls (10.61,27.84) and (18,34.13) .. (27.13,34.13) .. controls (27.32,34.13) and (27.51,34.13) .. (27.71,34.12) -- (27.13,20.07) -- cycle ; \draw   (27.71,6.02) .. controls (27.51,6.02) and (27.32,6.02) .. (27.13,6.02) .. controls (18,6.02) and (10.61,12.31) .. (10.61,20.07) .. controls (10.61,27.84) and (18,34.13) .. (27.13,34.13) .. controls (27.32,34.13) and (27.51,34.13) .. (27.71,34.12) ;  

\draw  [draw opacity=0][dash pattern={on 0.84pt off 2.51pt}] (117.7,34.61) .. controls (115.49,32.96) and (113.86,27.52) .. (113.86,21.07) .. controls (113.86,14.63) and (115.48,9.2) .. (117.7,7.54) -- (119.13,21.07) -- cycle ; \draw  [dash pattern={on 0.84pt off 2.51pt}] (117.7,34.61) .. controls (115.49,32.96) and (113.86,27.52) .. (113.86,21.07) .. controls (113.86,14.63) and (115.48,9.2) .. (117.7,7.54) ;  
\draw  [draw opacity=0] (118.55,7.1) .. controls (118.74,7.04) and (118.93,7.02) .. (119.13,7.02) .. controls (122.04,7.02) and (124.4,13.31) .. (124.4,21.07) .. controls (124.4,28.84) and (122.04,35.13) .. (119.13,35.13) .. controls (118.93,35.13) and (118.74,35.11) .. (118.55,35.05) -- (119.13,21.07) -- cycle ; \draw   (118.55,7.1) .. controls (118.74,7.04) and (118.93,7.02) .. (119.13,7.02) .. controls (122.04,7.02) and (124.4,13.31) .. (124.4,21.07) .. controls (124.4,28.84) and (122.04,35.13) .. (119.13,35.13) .. controls (118.93,35.13) and (118.74,35.11) .. (118.55,35.05) ;  
\draw  [draw opacity=0] (119.71,7.02) .. controls (119.51,7.02) and (119.32,7.02) .. (119.13,7.02) .. controls (110,7.02) and (102.61,13.31) .. (102.61,21.07) .. controls (102.61,28.84) and (110,35.13) .. (119.13,35.13) .. controls (119.32,35.13) and (119.51,35.13) .. (119.71,35.12) -- (119.13,21.07) -- cycle ; \draw   (119.71,7.02) .. controls (119.51,7.02) and (119.32,7.02) .. (119.13,7.02) .. controls (110,7.02) and (102.61,13.31) .. (102.61,21.07) .. controls (102.61,28.84) and (110,35.13) .. (119.13,35.13) .. controls (119.32,35.13) and (119.51,35.13) .. (119.71,35.12) ;  
\draw  [fill={rgb, 255:red, 0; green, 0; blue, 0 }  ,fill opacity=1 ] (105.49,20.82) .. controls (105.49,19.96) and (106.19,19.27) .. (107.04,19.27) .. controls (107.9,19.27) and (108.6,19.96) .. (108.6,20.82) .. controls (108.6,21.68) and (107.9,22.37) .. (107.04,22.37) .. controls (106.19,22.37) and (105.49,21.68) .. (105.49,20.82) -- cycle ;

\draw (39,10.4) node [anchor=north west][inner sep=0.75pt]    {$\mapsto \ 1,$};
\draw (130,11.4) node [anchor=north west][inner sep=0.75pt]    {$\mapsto \ X.$};

\end{tikzpicture}
\end{center}
Put $Y = X - h$. Then, elementary cobordisms in $\Cob^3_{\bullet/l}$ correspond to operations on $A$ as follows:
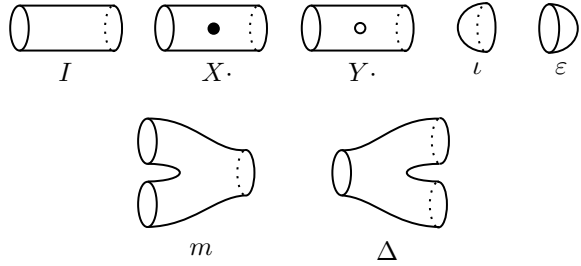
\begin{center}
    \input{tikzpictures/elem-cob-symbols}    
\end{center}

\subsection{Khovanov homology for tangles}
\label{subsec:kh-for-tangles}

Let $\Diag(B)$ denote the category whose objects are oriented tangle diagrams $T$ with $\del T = B$, and whose morphisms\footnote{
    In \cite{BN:2005}, the domain category of $[\cdot]$ is denoted $\Cob^4(B)$. Here we adopt the notation and definitions of \cite{BHL:2019} for clarity and to avoid a genericity argument for the cobordisms.
} are oriented movies between tangle diagrams that fix the boundary $B$. Here, we assume that each tangle diagram $T$ is equipped with an ordering of the crossings. Its algebraic counterpart is the category of complexes in the additive closure $\Mat(\Cob^3_{\bullet/l}(B))$ of $\Cob^3_{\bullet/l}(B)$, denoted
\[
    \Kob_{\bullet/l}(B) := \Kom(\Mat(\Cob^3_{\bullet/l}(B))).
\]

\begin{figure}[t]
    \centering
    \resizebox{.95\textwidth}{!}{
        \input{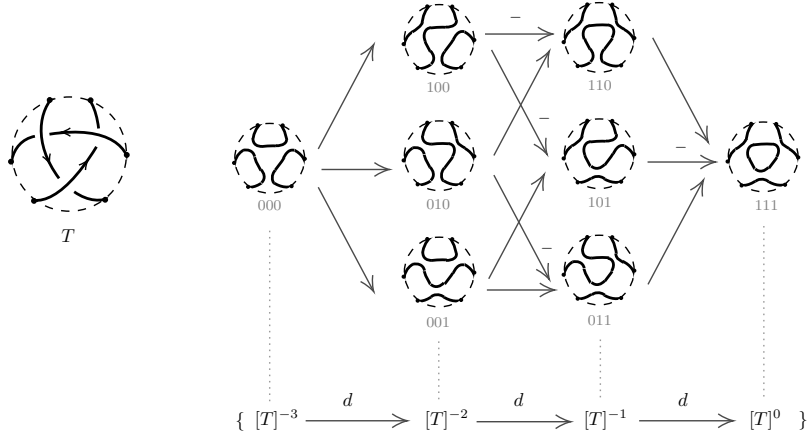}
    }
    \caption{The formal Khovanov bracket $[T]$}
    \label{fig:formal-ckh}
\end{figure}

Given a tangle diagram $T$ in $\Diag(B)$, the formal Khovanov bracket $[T]$ is defined as a bigraded chain complex in the additive closure $\Mat(\Cob^3_{\bullet/l}(B))$ of $\Cob^3_{\bullet/l}(B)$. The construction is identical to Khovanov's original one, except that each ``chain group'' is a formal direct sum of diagrams, and the differential $d$ is made of matrices whose entries are elementary cobordisms. See \Cref{fig:formal-ckh} for an example. Similarly, for a movie $S$ between tangle diagrams, a chain map $[S]$ between the corresponding complexes is defined as the composition of the chain maps associated with the elementary moves. This gives a functor 
\[
    [\cdot]\colon \Diag(B) \to \Kob_{\bullet/l}(B).
\]
It is proved in \cite{BN:2005} that the chain homotopy type of $[T]$ is an invariant of the tangle. In particular, when $\partial T = \emptyset$, the \textit{tautological functor} (TQFT)
\[
    \Hom_{\bullet/l}([\emptyset], -)\colon \Kob_{\bullet/l}(\emptyset) \longrightarrow \Kom({R\GrMod})
\]
recovers the \textit{$U(2)$-equivariant Khovanov complex} $\CKh_{h, t}(T) = \Hom_{\bullet/l}([\emptyset], [T])$, from which the original Khovanov homology \cite{Kh:2000} and its deformations \cite{Lee:2005,BN:2005,Kh:2006} can be fully recovered by substituting $h$ and $t$.

Furthermore, \cite[Theorem 5]{BN:2005} states that $[\cdot]$ descends to a functor, well-defined up to sign,
\[
    [\cdot]\colon \Diag_{/i}(B) \to h\Kob_{\bullet/l}(B)
\]
where $\Diag_{/i}(B)$ denotes the quotient category of $\Diag(B)$ modulo the 15 \textit{movie moves} of Carter and Saito \cite{CS:1993}, and $h\Kob_{\bullet/l}(B)$ denotes the homotopy category of $\Kob_{\bullet/l}(B)$. This proves that the cobordism map of Khovanov homology is isotopy invariant up to sign, and that Khovanov homology is functorial up to sign with respect to tangle and link cobordisms.

\subsection{Delooping and Gaussian elimination}

For $n \in \ZZ$, we write $q^n X$ for the object $X$ with its $q$-grading shifted by $n$. The following two propositions, \textit{Delooping} and \textit{Gaussian elimination}, were originally proved in \cite{BN:2007} and will be used repeatedly in this paper.

\begin{proposition}[Delooping]
\label{prop:delooping}
    The following matrices of morphisms give mutually inverse isomorphisms in $\Mat(\Cob^3_{\bullet/l}(B))$:
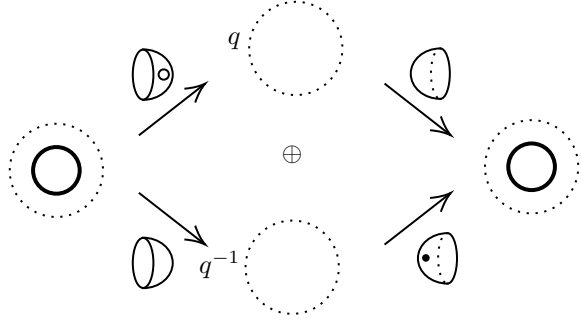
    \begin{center}
        \input{tikzpictures/delooping}
    \end{center}
\end{proposition}

\begin{definition}
    Consider the category of chain complexes in any additive category. A chain map $r\colon \Omega \rightarrow \Omega'$ is a \textit{strong deformation retraction}
    if there is a chain map $i\colon \Omega' \rightarrow \Omega$ (called the \textit{inclusion}) and a homotopy $h$ on $\Omega$ satisfying (i) $ri = 1$, (ii) $ir - 1 = dh + hd$, (iii) $hi = 0$, (iv) $rh = 0$ and (v) $h^2 = 0$. Such $\Omega'$ is called a \textit{strong deformation retract} of $\Omega$. The following diagram will be frequently used to describe these maps. 
    \[
        \begin{tikzcd}
        \Omega 
            \arrow[r, "r", shift left] 
            \arrow["h"', loop, distance=2em, in=305, out=235] & 
        \Omega' 
            \arrow[l, "i", shift left]
        \end{tikzcd}        
    \]
\end{definition}

\begin{proposition}[Gaussian elimination]
\label{prop:gauss-elim}
    Suppose the top row of the following diagram is part of a chain complex $\Omega$ in which the morphism $X \xrightarrow{a} Y$ is invertible. Then there is a strong deformation retract $\Omega'$ of $\Omega$ described by the bottom row, with a strong deformation retraction indicated by the downward vertical arrows, 
    \[
\begin{tikzcd}
[row sep=7em, column sep=4em, ampersand replacement=\&]
U 
    \arrow[r, "{\begin{pmatrix} \alpha \\ \beta \end{pmatrix}}"] 
    \arrow[d, equal] 
\& \begin{pmatrix} X \\ X' \end{pmatrix} 
    \arrow[rr, "{\begin{pmatrix} a & b \\ c & d \end{pmatrix}}", shift left] 
    \arrow[d, "{\begin{pmatrix} 0 & 1 \end{pmatrix}}"', shift right] 
\& \& \begin{pmatrix} Y \\ Y' \end{pmatrix} 
    \arrow[r, "{\begin{pmatrix} \gamma & \delta \end{pmatrix}}"] 
    \arrow[d, "{\begin{pmatrix} -ca^{-1} & 1 \end{pmatrix}}"', shift right] 
    \arrow[ll, "{\begin{pmatrix} -a^{-1} & 0 \\ 0 & 0 \end{pmatrix}}", dashed, shift left]
\& V 
    \arrow[d, equal] 
\\ U 
    \arrow[r, "\beta"] 
\& X' 
    \arrow[rr, "d - c a^{-1} b"] 
    \arrow[u, "{\begin{pmatrix} -a^{-1}b \\ 1 \end{pmatrix}}"', dashed, shift right]
\& \& Y' 
    \arrow[r, "\delta"] 
    \arrow[u, "{\begin{pmatrix} 0 \\ 1 \end{pmatrix}}"', dashed, shift right]
\& V.
\end{tikzcd}
    \]
    The inclusion and the homotopy associated to the retraction are indicated by the dashed arrows.
\end{proposition}
    \section{Involutive Khovanov homology for tangles}
\label{sec:inv-kh}

In \cite{Sano:2025b}, the author introduced the \textit{involutive Khovanov homology} for \textit{involutive links}. Here, we extend the construction to \textit{involutive tangles}, based on Bar-Natan's framework of Khovanov homology for tangles \cite{BN:2005}.

Throughout the rest of this paper, all categories are taken modulo the local relations, and we abbreviate
\[
    \Cob(B) := \Cob^3_{\bullet/l}(B),\quad
    \Kob(B) := \Kob_{\bullet/l}(B),
\]
writing $\Hom(-, -)$ for the hom-sets accordingly. Moreover, unless otherwise stated, we work over $\FF_2$; that is, the morphism groups of $\Cob(B)$ are taken to be $\FF_2$-vector spaces. Signs are nonetheless written out where they naturally arise, both for clarity and in view of a future upgrade to $\ZZ$ coefficients.

\subsection{$\tau$-equivariant Khovanov complex}

Let $B$ be a set of an even number of points on the boundary $\del D^2$ of the unit disk $D^2$. Let $\tau$ denote the 180-degree rotation of $\RR^2$ in $\RR^3$ about the $y$-axis. Its extension $\tau \times \id$ in $\RR^2 \times I$ will also be denoted $\tau$. Since the local relations are preserved by $\tau$, this gives rise to a functor
\[
    \tau\colon
    \Cob(B) \to \Cob(\tau B)
\]
and further induces a functor between the corresponding categories of complexes,
\[
    \tau\colon
    \Kob(B) \to \Kob(\tau B).
\]
By extending $\tau$ to tangle diagrams, we obtain a functor,
\[
    \tau\colon 
    \Diag(B) \to \Diag(\tau B).
\]
Here, the action of $\tau$ on a tangle diagram is defined to preserve the ordering of the crossings. 

Although these definitions seem straightforward and we have $[\tau T] = \tau [T]$ for any tangle $T$, the following diagram of functors is not strictly commutative:
\[
    \begin{tikzcd}
    \Diag(B) \arrow[r, "{[\cdot]}"] \arrow[d, "\tau"] & \Kob(B) \arrow[d, "\tau"] \\
    \Diag(\tau B) \arrow[r, "{[\cdot]}"] & \Kob(\tau B).
    \end{tikzcd}
\]
This can be exhibited by the cobordism map for the R3-move given in \cite[Figure 9]{BN:2005}. The map is defined by selecting the \textit{lowermost} crossing among the three crossings, but being lowermost is not preserved under $\tau$. Nonetheless, the diagram commutes up to homotopy, as the following proposition shows. 

\begin{proposition}
\label{prop:tau-equiv-mor}
    Given a morphism $S\colon T \to T'$ in $\Diag(B)$, the chain maps $\tau[S]$ and $[\tau S]$ are homotopic.
\end{proposition}

\begin{proof}
    This is essentially proved in \cite[Lemma 4.27]{Sano:2025b}. It suffices to treat the case where $S$ is either a Reidemeister move or a Morse move, and in either case we may assume that $T, T'$ are minimal tangle diagrams representing the move. For the R1, R2 moves and the three Morse moves, the two chain maps are strictly equal. For the R3 move, since the diagrams $T, T'$ are \textit{$\Kh$-simple} (see \cite[Definition 8.5]{BN:2005}), it follows that the two chain maps are homotopic.\footnote{Alternatively, we may construct an explicit homotopy between the two chain maps.}
\end{proof}

Hereafter, we fix the homotopy for the R3 move, so that the homotopy exhibiting $\tau[S] \htpy [\tau S]$ is determined without ambiguity.

Now, we move on to the involutive setting. Recall that an \textit{involutive link} in $S^3$ is an oriented link $L$ equipped with an involution $\tau$ on $S^3$ (with fixed point set an unknotted circle) such that $\tau(L) = L$, not necessarily preserving orientation. An \textit{involutive link diagram} is a link diagram that is symmetric with respect to the 180-degree rotation about the $y$-axis. We extend these definitions to tangles. Let $B \subset \partial D^2$ be a set of boundary points satisfying $\tau B = B$.

\begin{definition}
\label{def:inv-tangle}
    An \textit{involutive tangle} with boundary in $B$ is a tangle $T$ in the $3$-ball $D^3$ with $\partial T = B$ that is fixed under the $180$-degree rotation $\tau$ about the $y$-axis, i.e.\ $\tau(T) = T$, not necessarily preserving orientation. Two involutive tangles are \textit{equivalent} if they are isotopic through involutive tangles fixing the boundary.
\end{definition}

\begin{definition}
\label{def:inv-tangle-diag}
    An \textit{involutive tangle diagram} $T$ with boundary in $B$ is a tangle diagram with $\partial T = B$ that is symmetric with respect to the 180-degree rotation about the $y$-axis, i.e.\ $\tau T = T$, not necessarily preserving orientation. 
\end{definition}

Every involutive tangle is presented by an involutive tangle diagram, and two involutive tangle diagrams present equivalent involutive tangles if and only if they are related by a finite sequence of symmetric planar isotopies and \textit{involutive Reidemeister moves} (see \Cref{fig:inv-reidemeister} in \Cref{subsec:invariance}). This equivariant analogue of the Reidemeister theorem follows from the argument of \cite[Theorem 2.3]{LW:2021}, applied with the boundary held fixed.

For an involutive tangle diagram $T$, the two diagrams $T$ and $\tau T$ are identical but have different orderings of the crossings, and therefore, the complexes $[T]$ and $\tau [T] = [\tau T]$ are not strictly equal. Nonetheless, we have the following proposition.

\begin{proposition}
\label{prop:inv-tng-assoc-isom}
    There is a canonical isomorphism $i\colon [T] \to \tau [T]$ such that $\tau(i) = i^{-1}\colon \tau [T] \to [T]$. 
\end{proposition}

\begin{proof}
    The isomorphism $i$ is given by the reordering of the crossings from $[T]$ to $\tau[T]$. To be precise, each crossing $x$ of $T$ corresponds to a unique crossing $x'$ of $\tau T$ lying at the same position. The mapping $x \mapsto x'$ gives a bijection $c(T) \to c(\tau T)$ between the sets of crossings of $T$ and $\tau T$, which induces a bijection $2^{c(T)} \to 2^{c(\tau T)}$ between the sets of states of $T$ and $\tau T$. This gives an isomorphism $i\colon [T] \to \tau[T]$ between the two complexes, which is unique up to overall sign, by the argument given after \cite[Definition 4.5]{LS:2014a}. We choose the one that restricts to $I$ on the initial state $\bar{0}$. 
    
    Now, consider $\tau(i)\colon \tau [T] \to [T]$. The composition $\tau(i) \circ i$ is an automorphism of $[T]$ that restricts to $\pm I$ on each state space. By the same argument, this sign is uniform, and since it restricts to $I$ on $\bar{0}$, we have $\tau(i) \circ i = \id_{[T]}$.
\end{proof}

\begin{figure}
    \centering
    \input{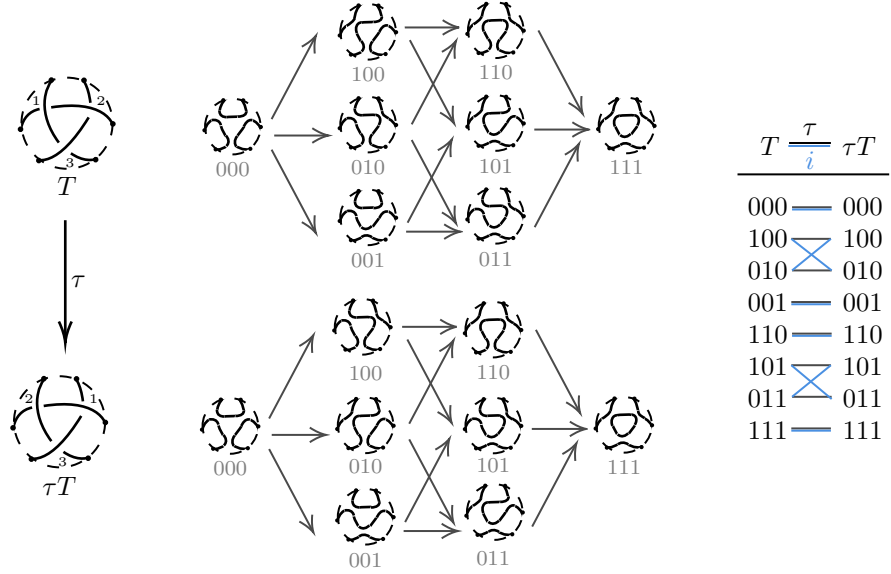}
    \caption{Explicit identification between $[T]$ and $\tau [T]$.}
    \label{fig:tau-equiv-cpx}
\end{figure}

The isomorphism $i$ of \Cref{prop:inv-tng-assoc-isom} is called the \textit{isomorphism associated with the involutive tangle diagram} $T$. \Cref{fig:tau-equiv-cpx} shows the action of $\tau$ on an involutive tangle $T$ together with the associated isomorphism $i\colon [T] \isom \tau [T]$. 

A general \textit{$\tau$-equivariant complex} is obtained by abstracting the structure of $([T], i)$. 

\begin{definition}
\label{def:tau-equiv-complex}
    A \textit{$\tau$-equivariant complex} (or simply a \textit{$\tau$-complex}), is a complex $\Omega \in \Kob(B)$ equipped with an isomorphism $i\colon \Omega \to \tau \Omega$ such that $\tau(i) = i^{-1}$. 
    A \textit{$\tau$-equivariant chain map} (or simply a \textit{$\tau$-map}) between $\tau$-complexes $(\Omega, i)$ and $(\Omega', i')$, is a pair consisting of a morphism $f\colon \Omega \to \Omega'$ and a homotopy $F$ exhibiting $i' f - \tau(f) i = [d, F]$. We express $F$ as a homotopy commutative diagram
    \[
        \begin{tikzcd}[row sep = 3em, column sep = 3em]
        \Omega \arrow[r, "f"] \arrow[d, "i"'] & 
        \Omega' \arrow[d, "i'"] \arrow[ld, "F"', Rightarrow] \\
        \tau \Omega \arrow[r, "\tau f"'] 
        & \tau \Omega'.
        \end{tikzcd}
    \]
    A morphism $(f, F)$ is \textit{strictly $\tau$-equivariant} when $F = 0$. The identity morphism on $(\Omega, i)$ is given by 
    \[
        \id_{(\Omega, i)} := (\id_\Omega, 0).
    \]
    For morphisms $(f, F)\colon (\Omega, i) \to (\Omega', i')$ and $(g, G)\colon (\Omega', i') \to (\Omega'', i'')$, the composition is defined as 
    \[
        (g, G) \circ (f, F) := (g f, Gf + \tau(g) F).
    \]
    The above definitions give rise to the \textit{category of $\tau$-complexes}, denoted $\Kob^\tau(B)$. 
\end{definition}

\begin{proposition}
    $\Kob^\tau(B)$ is indeed a category. 
\end{proposition}

\begin{proof}
    It suffices to check that the composition is well-defined. For the above morphisms $(f, F)$ and $(g, G)$, we have
    \begin{align*}
        i'' (gf) - \tau(gf) i 
        &= (i''g - \tau(g)i')f + \tau(g)(i'f - \tau(f)i) \\ 
        &= [d, G]f + \tau(g)[d, F] \\
        &= [d, Gf + \tau(g)F]
    \end{align*}
    as desired. It is also easy to check that the composition is associative, and that the identity morphism is unital. 
\end{proof}

\begin{definition}
\label{def:inv-kh-bracket}
    For an involutive tangle diagram $T$, the $\tau$-complex $[T]^\tau := ([T], i)$, which is the pair of the Khovanov complex $[T]$ and the associated isomorphism $i\colon [T] \to \tau [T]$, is called the \textit{$\tau$-equivariant Khovanov complex} of $T$.
\end{definition}

In order to state the invariance of $[T]^\tau$, we need a notion of homotopy in $\Kob^\tau(B)$. 

\begin{definition}
\label{def:tau-equiv-homotopy}
    A \textit{$\tau$-equivariant homotopy} $(h, H)$ between $\tau$-chain maps $(f, F)$ and $(g, G)$ is a pair consisting of a degree $-1$ homotopy $h$ with $f - g = [d, h]$ and a degree $-2$ homotopy $H$ exhibiting
    \[
        F - G = [d, H] + i' h - \tau(h) i.
    \]
\end{definition}

The following diagram describes the condition given in \Cref{def:tau-equiv-homotopy}, where $F$ fills the back side of the cylinder, $G$ fills the front, and $H$ fills the interior.
\[
    \begin{tikzcd}[column sep=5em, row sep=3.5em]
    \Omega 
        \arrow[d, "i"]
        \arrow[r, "f"{name=f}, bend left]
        \arrow[r, "g"'{name=g}, bend right] 
        \arrow[Rightarrow, from=f, to=g, "h", shorten <= 2pt, shorten >= 2pt]
    & \Omega' 
        \arrow[d, "i'"] \\
    \tau\Omega 
        \arrow[r, "\tau(f)"{name=tf}, bend left]
        \arrow[r, "\tau(g)"'{name=tg}, bend right] 
        \arrow[Rightarrow, from=tf, to=tg, "\tau(h)", shorten <= 2pt, shorten >= 2pt]
    & \tau\Omega'
    \end{tikzcd}
\]
One verifies that $\tau$-equivariant homotopies give an equivalence relation on the hom-sets of $\Kob^\tau(B)$, compatible with the composition. Concepts such as \textit{$\tau$-equivariant homotopy equivalences} and \textit{$\tau$-equivariant strong deformation retractions} are defined in the obvious way. 

\subsection{Invariance}
\label{subsec:invariance}

\begin{figure}[t]
    \centering
    \resizebox{\textwidth}{!}{
    \input{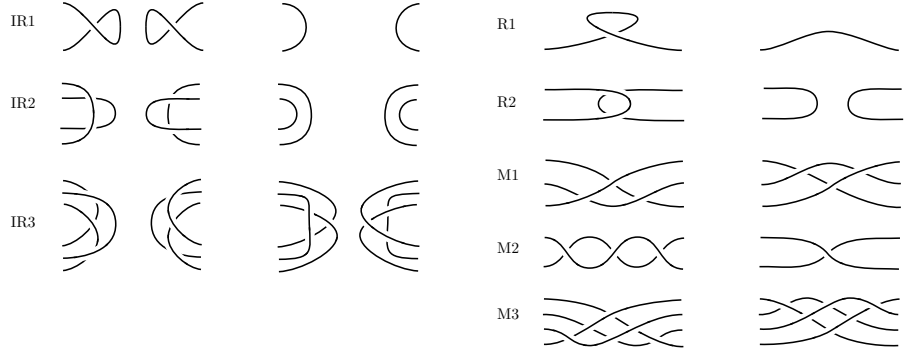}
    }
    \caption{Involutive Reidemeister moves}
    \label{fig:inv-reidemeister}
\end{figure}

In \cite{LW:2021}, Lobb and Watson introduced the \textit{involutive Reidemeister moves}, which are the involutive analogues of the ordinary Reidemeister moves (see \Cref{fig:inv-reidemeister}). It is proved that two involutive links are equivalent (in $\RR^3$) if and only if their diagrams are related by a sequence of equivariant planar isotopies and involutive Reidemeister moves. 

\begin{proposition}
\label{prop:ir-invariance}
    Let $T, T'$ be involutive tangle diagrams with $\partial T = \partial T' = B$ that are related by one of the involutive Reidemeister moves. Then there is a $\tau$-equivariant homotopy equivalence
    \[
        f^\tau\colon [T]^\tau \to [T']^\tau
    \]
    that lifts the standard homotopy equivalence $f\colon [T] \to [T']$ given in \cite{BN:2005}.
\end{proposition}

\begin{proof}
    This is essentially proved in \cite[Section 2.2]{Sano:2025b}. Let $S\colon T \to T'$ be one of the involutive Reidemeister moves. Using the maps defined in \cite{BN:2005} together with $\Kh$-simplicity arguments (see \cite[Definition 2.10]{Sano:2025b}), we obtain preferred chain maps $f, g$ and chain homotopies $h, h', F, G$ in $\Kob(B)$,
    \[
    \begin{tikzcd}[column sep=3em]
    {[T]}
        \arrow[r, "f", shift left] 
        \arrow["h"', loop, distance=2em, in=305, out=235] &
    {[T']}
        \arrow[l, "g", shift left] 
        \arrow["h'"', loop, distance=2em, in=305, out=235]
    \end{tikzcd}
    \]
    satisfying
    \[
        g f \htpy_{h} I, \quad 
        f g \htpy_{h'} I, \quad 
        f \htpy_{F} \tau(f),\quad 
        g \htpy_{G} \tau(g).
    \]
    Here, $gf \htpy_h I$ is short for $gf - I = [d, h]$. For the latter two equations, we have suppressed the canonical isomorphisms $i, i'$. 
    
    Furthermore, again from $\Kh$-simplicity arguments, there are degree $-2$ homotopies $H, H'$ giving
    \begin{align*}
        G f + \tau(g) F
            &\htpy_{H} h - \tau(h), \\
        F g + \tau(f) G
            &\htpy_{H'} h' - \tau(h').
    \end{align*}
    Indeed, on the one hand, we have
    \begin{align*}
        gf - \tau(gf) 
            &= (g - \tau(g))f + \tau(g)(f - \tau(f)) \\ 
            &= [d, G]f + \tau(g)[d, F] \\ 
            &= [d, Gf + \tau(g)F]
    \end{align*}
    and on the other hand,
    \begin{align*}
        gf - \tau(gf) 
            &= [d, h] - [d, \tau(h)] \\
            &= [d, h - \tau(h)]
    \end{align*}
    and therefore $Gf + \tau(g)F$ and $h - \tau(h)$ must be homotopic by $\Kh$-simplicity. Now, define
    \begin{align*}
        f^\tau = (f, F)&\colon [T]^\tau \to [T']^\tau, \\
        g^\tau = (g, G)&\colon [T']^\tau \to [T]^\tau
    \end{align*}
    and
    \begin{align*}
        h^\tau = (h, H)&\colon [T]^\tau \to [T]^\tau, \\
        {h'}^\tau = (h', H')&\colon [T']^\tau \to [T']^\tau.
    \end{align*}
    One easily verifies that $g^\tau f^\tau \htpy_{h^\tau} I$ and $f^\tau g^\tau \htpy_{{h'}^\tau} I$.
\end{proof}

Next, we consider the $I$-move. In \cite[Theorem 1.1]{BDM+:2026b}, it is stated that for isotopy invariance of involutive links in $S^3$ (rather than $\RR^3$), we additionally need to consider this move. Let $L, L'$ be involutive link diagrams that are related by the $I$-move as in \Cref{fig:I-move}. There is an obvious isomorphism $\varphi$ between the Khovanov complexes $[L], [L']$, obtained by matching the resolved diagrams one-to-one. This map already ensures the invariance of the chain homotopy type. However, the $I$-move interchanges the two fixed points on the axis, and hence exchanges the $XY$-labelings of the Seifert circles, which are used to define the (equivariant) Lee cycles (see \cite[Algorithm 3.1]{Sano:2025b}). Accordingly, we compose it with the involution $\sigma$, described for each circle as a dotted cobordism
\begin{center}
    \tikzset{every picture/.style={line width=0.75pt}} 

\begin{tikzpicture}[x=0.75pt,y=0.75pt,yscale=-1,xscale=1]

\draw  [color={rgb, 255:red, 0; green, 0; blue, 0 }  ,draw opacity=1 ] (60.6,28.61) .. controls (60.6,20.9) and (63.18,14.66) .. (66.35,14.66) .. controls (69.53,14.66) and (72.11,20.9) .. (72.11,28.61) .. controls (72.11,36.31) and (69.53,42.56) .. (66.35,42.56) .. controls (63.18,42.56) and (60.6,36.31) .. (60.6,28.61) -- cycle ;
\draw  [draw opacity=0] (65.77,14.56) .. controls (65.97,14.55) and (66.16,14.55) .. (66.35,14.55) .. controls (75.48,14.55) and (82.88,20.84) .. (82.88,28.61) .. controls (82.88,36.37) and (75.48,42.67) .. (66.35,42.67) .. controls (66.16,42.67) and (65.97,42.66) .. (65.77,42.66) -- (66.35,28.61) -- cycle ; \draw   (65.77,14.56) .. controls (65.97,14.55) and (66.16,14.55) .. (66.35,14.55) .. controls (75.48,14.55) and (82.88,20.84) .. (82.88,28.61) .. controls (82.88,36.37) and (75.48,42.67) .. (66.35,42.67) .. controls (66.16,42.67) and (65.97,42.66) .. (65.77,42.66) ;  
\draw  [draw opacity=0][dash pattern={on 0.84pt off 2.51pt}] (118.17,41.61) .. controls (115.95,39.96) and (114.32,34.52) .. (114.32,28.07) .. controls (114.32,21.63) and (115.95,16.2) .. (118.16,14.54) -- (119.59,28.07) -- cycle ; \draw  [dash pattern={on 0.84pt off 2.51pt}] (118.17,41.61) .. controls (115.95,39.96) and (114.32,34.52) .. (114.32,28.07) .. controls (114.32,21.63) and (115.95,16.2) .. (118.16,14.54) ;  
\draw  [draw opacity=0] (119.02,14.1) .. controls (119.21,14.04) and (119.4,14.02) .. (119.59,14.02) .. controls (122.5,14.02) and (124.86,20.31) .. (124.86,28.07) .. controls (124.86,35.84) and (122.5,42.13) .. (119.59,42.13) .. controls (119.4,42.13) and (119.21,42.11) .. (119.02,42.05) -- (119.59,28.07) -- cycle ; \draw   (119.02,14.1) .. controls (119.21,14.04) and (119.4,14.02) .. (119.59,14.02) .. controls (122.5,14.02) and (124.86,20.31) .. (124.86,28.07) .. controls (124.86,35.84) and (122.5,42.13) .. (119.59,42.13) .. controls (119.4,42.13) and (119.21,42.11) .. (119.02,42.05) ;  
\draw  [draw opacity=0] (120.17,14.02) .. controls (119.98,14.02) and (119.79,14.02) .. (119.59,14.02) .. controls (110.47,14.02) and (103.07,20.31) .. (103.07,28.07) .. controls (103.07,35.84) and (110.47,42.13) .. (119.59,42.13) .. controls (119.79,42.13) and (119.98,42.13) .. (120.17,42.12) -- (119.59,28.07) -- cycle ; \draw   (120.17,14.02) .. controls (119.98,14.02) and (119.79,14.02) .. (119.59,14.02) .. controls (110.47,14.02) and (103.07,20.31) .. (103.07,28.07) .. controls (103.07,35.84) and (110.47,42.13) .. (119.59,42.13) .. controls (119.79,42.13) and (119.98,42.13) .. (120.17,42.12) ;  

\draw  [color={rgb, 255:red, 0; green, 0; blue, 0 }  ,draw opacity=1 ] (159.13,28.61) .. controls (159.13,20.9) and (161.71,14.66) .. (164.89,14.66) .. controls (168.07,14.66) and (170.64,20.9) .. (170.64,28.61) .. controls (170.64,36.31) and (168.07,42.56) .. (164.89,42.56) .. controls (161.71,42.56) and (159.13,36.31) .. (159.13,28.61) -- cycle ;
\draw  [draw opacity=0] (164.31,14.56) .. controls (164.5,14.55) and (164.69,14.55) .. (164.89,14.55) .. controls (174.01,14.55) and (181.41,20.84) .. (181.41,28.61) .. controls (181.41,36.37) and (174.01,42.67) .. (164.89,42.67) .. controls (164.69,42.67) and (164.5,42.66) .. (164.31,42.66) -- (164.89,28.61) -- cycle ; \draw   (164.31,14.56) .. controls (164.5,14.55) and (164.69,14.55) .. (164.89,14.55) .. controls (174.01,14.55) and (181.41,20.84) .. (181.41,28.61) .. controls (181.41,36.37) and (174.01,42.67) .. (164.89,42.67) .. controls (164.69,42.67) and (164.5,42.66) .. (164.31,42.66) ;  
\draw  [draw opacity=0][dash pattern={on 0.84pt off 2.51pt}] (216.7,41.61) .. controls (214.49,39.96) and (212.86,34.52) .. (212.86,28.07) .. controls (212.86,21.63) and (214.48,16.2) .. (216.7,14.54) -- (218.13,28.07) -- cycle ; \draw  [dash pattern={on 0.84pt off 2.51pt}] (216.7,41.61) .. controls (214.49,39.96) and (212.86,34.52) .. (212.86,28.07) .. controls (212.86,21.63) and (214.48,16.2) .. (216.7,14.54) ;  
\draw  [draw opacity=0] (217.55,14.1) .. controls (217.74,14.04) and (217.93,14.02) .. (218.13,14.02) .. controls (221.04,14.02) and (223.4,20.31) .. (223.4,28.07) .. controls (223.4,35.84) and (221.04,42.13) .. (218.13,42.13) .. controls (217.93,42.13) and (217.74,42.11) .. (217.55,42.05) -- (218.13,28.07) -- cycle ; \draw   (217.55,14.1) .. controls (217.74,14.04) and (217.93,14.02) .. (218.13,14.02) .. controls (221.04,14.02) and (223.4,20.31) .. (223.4,28.07) .. controls (223.4,35.84) and (221.04,42.13) .. (218.13,42.13) .. controls (217.93,42.13) and (217.74,42.11) .. (217.55,42.05) ;  
\draw  [draw opacity=0] (218.71,14.02) .. controls (218.51,14.02) and (218.32,14.02) .. (218.13,14.02) .. controls (209,14.02) and (201.61,20.31) .. (201.61,28.07) .. controls (201.61,35.84) and (209,42.13) .. (218.13,42.13) .. controls (218.32,42.13) and (218.51,42.13) .. (218.71,42.12) -- (218.13,28.07) -- cycle ; \draw   (218.71,14.02) .. controls (218.51,14.02) and (218.32,14.02) .. (218.13,14.02) .. controls (209,14.02) and (201.61,20.31) .. (201.61,28.07) .. controls (201.61,35.84) and (209,42.13) .. (218.13,42.13) .. controls (218.32,42.13) and (218.51,42.13) .. (218.71,42.12) ;  

\draw  [fill={rgb, 255:red, 0; green, 0; blue, 0 }  ,fill opacity=1 ] (76.32,28.56) .. controls (76.32,27.7) and (77.01,27) .. (77.87,27) .. controls (78.72,27) and (79.42,27.7) .. (79.42,28.56) .. controls (79.42,29.41) and (78.72,30.11) .. (77.87,30.11) .. controls (77.01,30.11) and (76.32,29.41) .. (76.32,28.56) -- cycle ;
\draw  [fill={rgb, 255:red, 0; green, 0; blue, 0 }  ,fill opacity=1 ] (204.49,27.82) .. controls (204.49,26.96) and (205.19,26.27) .. (206.04,26.27) .. controls (206.9,26.27) and (207.6,26.96) .. (207.6,27.82) .. controls (207.6,28.68) and (206.9,29.37) .. (206.04,29.37) .. controls (205.19,29.37) and (204.49,28.68) .. (204.49,27.82) -- cycle ;

\draw (136.07,20.42) node [anchor=north west][inner sep=0.75pt]    {$-$};
\draw (13,18.4) node [anchor=north west][inner sep=0.75pt]    {$\sigma \ =\ $};

\end{tikzpicture}
\end{center}
This corresponds to the Frobenius algebra automorphism
\[
    \sigma\colon A \to A;\quad 1 \mapsto 1,\quad X \mapsto Y = X - h.
\]
(Note that $\sigma$ is indeed a Frobenius algebra isomorphism, since we are working over $\FF_2$. See \cite{KhSano:2026}.) The effect of $\sigma$ is that it swaps $\bullet$ with $\circ$. We define the cobordism map assigned to the $I$-move to be $\sigma \varphi = \varphi \sigma$.

\begin{figure}
    \begin{center}
        \tikzset{every picture/.style={line width=0.75pt}} 

\begin{tikzpicture}[x=0.75pt,y=0.75pt,yscale=-1,xscale=1]

\draw [color={rgb, 255:red, 155; green, 155; blue, 155 }  ,draw opacity=1 ][line width=0.75]  [dash pattern={on 0.84pt off 2.51pt}]  (48.5,16.68) -- (48.5,128.44) ;
\draw  [fill={rgb, 255:red, 255; green, 255; blue, 255 }  ,fill opacity=1 ] (22.4,55.71) -- (74.6,55.71) -- (74.6,89.41) -- (22.4,89.41) -- cycle ;

\draw  [draw opacity=0] (14.2,55.42) .. controls (14.47,44.34) and (29.72,35.41) .. (48.5,35.41) .. controls (67.25,35.41) and (82.49,44.32) .. (82.8,55.37) -- (48.5,55.71) -- cycle ; \draw   (14.2,55.42) .. controls (14.47,44.34) and (29.72,35.41) .. (48.5,35.41) .. controls (67.25,35.41) and (82.49,44.32) .. (82.8,55.37) ;  
\draw  [draw opacity=0] (14.2,55.2) .. controls (14.22,64.56) and (18.16,72.14) .. (23.03,72.21) -- (23.1,55.12) -- cycle ; \draw   (14.2,55.2) .. controls (14.22,64.56) and (18.16,72.14) .. (23.03,72.21) ;  
\draw  [draw opacity=0] (82.8,55.37) .. controls (82.78,64.72) and (79.09,72.41) .. (74.35,73.42) -- (73.29,55.29) -- cycle ; \draw   (82.8,55.37) .. controls (82.78,64.72) and (79.09,72.41) .. (74.35,73.42) ;  

\draw  [fill={rgb, 255:red, 0; green, 0; blue, 0 }  ,fill opacity=1 ] (47.21,35.1) .. controls (47.21,34.26) and (47.89,33.58) .. (48.73,33.58) .. controls (49.56,33.58) and (50.24,34.26) .. (50.24,35.1) .. controls (50.24,35.93) and (49.56,36.61) .. (48.73,36.61) .. controls (47.89,36.61) and (47.21,35.93) .. (47.21,35.1) -- cycle ;

\draw [color={rgb, 255:red, 155; green, 155; blue, 155 }  ,draw opacity=1 ][line width=0.75]  [dash pattern={on 0.84pt off 2.51pt}]  (176.1,16.68) -- (176.1,128.44) ;
\draw  [fill={rgb, 255:red, 255; green, 255; blue, 255 }  ,fill opacity=1 ] (150,55.71) -- (202.2,55.71) -- (202.2,89.41) -- (150,89.41) -- cycle ;

\draw  [draw opacity=0] (141.8,94.99) .. controls (142.09,105.18) and (157.34,113.39) .. (176.1,113.39) .. controls (194.83,113.39) and (210.06,105.21) .. (210.39,95.04) -- (176.1,94.7) -- cycle ; \draw   (141.8,94.99) .. controls (142.09,105.18) and (157.34,113.39) .. (176.1,113.39) .. controls (194.83,113.39) and (210.06,105.21) .. (210.39,95.04) ;  
\draw  [draw opacity=0] (141.8,95.16) .. controls (141.82,86.55) and (145.76,79.57) .. (150.64,79.51) -- (150.7,95.24) -- cycle ; \draw   (141.8,95.16) .. controls (141.82,86.55) and (145.76,79.57) .. (150.64,79.51) ;  
\draw  [draw opacity=0] (210.4,95.01) .. controls (210.37,86.35) and (206.64,79.24) .. (201.87,78.38) -- (200.89,95.09) -- cycle ; \draw   (210.4,95.01) .. controls (210.37,86.35) and (206.64,79.24) .. (201.87,78.38) ;  

\draw  [fill={rgb, 255:red, 0; green, 0; blue, 0 }  ,fill opacity=1 ] (174.81,113.07) .. controls (174.81,112.24) and (175.49,111.56) .. (176.33,111.56) .. controls (177.16,111.56) and (177.84,112.24) .. (177.84,113.07) .. controls (177.84,113.91) and (177.16,114.59) .. (176.33,114.59) .. controls (175.49,114.59) and (174.81,113.91) .. (174.81,113.07) -- cycle ;

\draw (48.5,72.56) node    {$L$};
\draw (176.1,72.56) node    {$L$};

\end{tikzpicture}
    \end{center}
    \caption{The $I$-move}
    \label{fig:I-move}
\end{figure}
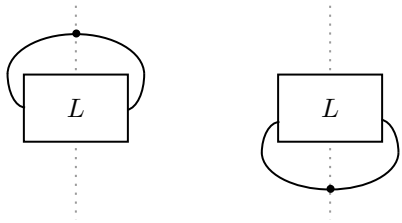

\begin{proposition}
\label{prop:i-move-invariance}
    Let $L, L'$ be involutive link diagrams that are related by the $I$-move. Then $\sigma \varphi$ gives a $\tau$-equivariant homotopy equivalence
    \[
        \sigma \varphi\colon [L]^\tau \to [L']^\tau
    \]
    that lifts the map given in \cite[Proposition 3.6]{SS:2024} (under the TQFT over $\FF_2$).
\end{proposition}

\begin{proof}
    Since $\sigma$ is given by the same local cobordism on every circle, it commutes with $\tau$, and hence $\sigma \varphi$ is strictly $\tau$-equivariant.
\end{proof}

\begin{remark}
\label{rem:I-move}
    \cite[Theorem 2]{Sano:2025b} lacked the observation for the $I$-move, and \Cref{prop:i-move-invariance} fills the gap. Moreover, the equivariant Lee cycles of $L$ and $L'$ correspond under $\sigma \varphi$, hence their behavior under cobordisms \cite[Proposition 4.29]{Sano:2025b} remains valid.
\end{remark}

Combining \Cref{prop:ir-invariance,prop:i-move-invariance}, the proof of \Cref{thm:inv-tang-invariance} is complete.

\subsection{Categorical formulation}

In order to state the correspondence $T \mapsto [T]^\tau$ in the categorical setting, we introduce the \textit{category of involutive tangle diagrams}.

\begin{definition}
    An \textit{involutive movie} $S$ between involutive tangle diagrams $T$ and $T'$ with $\del T = \del T' = B$ is a movie in which every still is an involutive tangle diagram (see \cite[Theorem 1.2]{BDM+:2026b}, applied with the boundary held fixed, which states that such $S$ is a sequence of equivariant isotopies, equivariant Reidemeister moves, and equivariant Morse moves). The involutive tangle diagrams and involutive movies form the \textit{category of involutive tangle diagrams}, denoted $\Diag^\tau(B)$.
\end{definition}

Both $\Diag^\tau(B)$ and $\Kob^\tau(B)$ admit forgetful functors $U$ to $\Diag(B)$ and $\Kob(B)$ respectively. Given a commutative diagram of functors
\[
\begin{tikzcd}
    \Diag^\tau(B) \arrow[d, "U"'] \arrow[r, "\tilde{\Phi}"] & \Kob^\tau(B) \arrow[d, "U"] \\
    \Diag(B) \arrow[r, "\Phi"] & \Kob(B),
\end{tikzcd}
\]
 we say $\tilde{\Phi}$ \textit{lifts} $\Phi$. 

\begin{proposition}
\label{prop:inv-functoriality}
    The $\tau$-equivariant Khovanov bracket extends to a functor
    \[
        [\cdot]^\tau \colon \Diag^\tau(B) \to \Kob^\tau(B)
    \]
    that lifts the Khovanov bracket functor
    \[
        [\cdot] \colon \Diag(B) \to \Kob(B).
    \]
\end{proposition}

\begin{proof}
    Let $S\colon T \to T'$ be an involutive movie, decomposed into elementary movies $S = S_N \circ \cdots \circ S_1$. Each $S_i$ is either an equivariant planar isotopy, an involutive Reidemeister move, or an equivariant Morse move. The corresponding map $[S_i]^\tau$ for the first is the canonical identification, that for the second is given in \Cref{thm:inv-tang-invariance}, and that for the third is the obvious strictly $\tau$-equivariant map. Define $[S]^\tau$ as the composition of these maps. This satisfies $[S' \circ S]^\tau = [S']^\tau \circ [S]^\tau$ from the construction.
\end{proof}

\begin{remark}
    It is natural to expect that the functor $[\cdot]^\tau$ descends to a (projective) functor
    \[
        [\cdot]^\tau \colon \Diag^\tau_{/i}(B) \to h\Kob^\tau(B)
    \]
    where the source category is the quotient category of $\Diag^\tau(B)$ modulo the \textit{equivariant movie moves} introduced in \cite[Theorem 1.4]{BDM+:2026b}, and the target category is the homotopy category of $\Kob^\tau(B)$. This would then promote $[\cdot]^\tau$ to a functor on the \textit{category of involutive tangles}. We shall not pursue the proof in this paper.     
\end{remark}

In \cite{Sano:2025b}, we introduced the concept of \textit{(simple) isotopy-equivariant cobordisms} for involutive links. Here, we define the analogue for tangles.

\begin{definition}
    A movie $S\colon T \to T'$ between involutive tangle diagrams $T, T'$ is \textit{isotopy-equivariant} if $S$ and $\tau S$ are related by (ordinary) movie moves; in other words, if the corresponding cobordisms are (non-equivariantly) isotopic. 
\end{definition}

The following is a generalization of \cite[Proposition 4.28]{Sano:2025b}. 

\begin{proposition}
\label{prop:ie-movie-map}
    Given an isotopy-equivariant movie $S\colon T \to T'$, there is a map between the corresponding $\tau$-equivariant Khovanov complexes,
    \[
        [S]^\tau \colon [T]^\tau \to [T']^\tau
    \]
    that lifts the standard cobordism map
    \[
        [S] \colon [T] \to [T'].
    \]
\end{proposition}

\begin{proof}
    From the assumption, $S$ and $\tau S$ are isotopic, and \cite[Theorem 5]{BN:2005} gives a homotopy $[S] \htpy [\tau S]$. Furthermore, \Cref{prop:tau-equiv-mor} gives a homotopy $[\tau S] \htpy \tau[S]$, so there is a homotopy $h_S$ exhibiting $[S] \htpy \tau [S]$. Define $[S]^\tau := ([S], h_S)$. 
\end{proof}

\begin{remark}
    In the proof of \Cref{prop:ie-movie-map}, there is an ambiguity in the choice of an isotopy from $S$ to $\tau S$, resulting in an ambiguity in the choice of the homotopy $h_S$. Hence the uniqueness of $([S], h_S)$ (up to homotopy) is not established. To fix this, we need higher coherence of the functoriality of Khovanov homology, which we shall not pursue. 
\end{remark}

\subsection{Involutive Khovanov complex}
\label{subsec:tau-cone}

Here, we restrict to the case $B = \emptyset$ and consider involutive link diagrams. First, for any $X \in \Cob(\emptyset)$ there is a natural isomorphism
\[
    J_\tau: X \to \tau X
\]
given by cylinders that connect each circle $C \subset X$ to the counterpart $\tau C \subset \tau X$. Here, since we are considering the dotted category, each cylinder can be neck-cut into pieces, and there is no ambiguity due to the knotting of the cylinders. 
We have $J_\tau \circ J_\tau = \id$, and the action $\tau$ on the morphisms of $\Cob(\emptyset)$ can be realized as the conjugation by $J_\tau$. Namely, for any cobordism $S: X \to Y$, the following diagram commutes
\[
    \begin{tikzcd}
    X \arrow[r, "S"] \arrow[d, "J_\tau"'] & 
    Y \arrow[d, "J_\tau"] \\
    \tau X \arrow[r, "\tau S"] & 
    \tau Y.
    \end{tikzcd}
\]
This extends to a natural isomorphism in $\Kob(\emptyset)$: for any complex $\Omega$, an isomorphism
\[
    J_\tau\colon \Omega \to \tau \Omega
\]
is given by connecting each summand $X$ of $\Omega$ with the summand $\tau X$ of $\tau \Omega$ by $J_\tau$. Again, for any morphism $f: \Omega \to \Omega'$, the following diagram commutes
\[
    \begin{tikzcd}
    \Omega \arrow[r, "f"] \arrow[d, "J_\tau"'] & 
    \Omega' \arrow[d, "J_\tau"] \\
    \tau \Omega \arrow[r, "\tau f"] & 
    \tau \Omega'.
    \end{tikzcd}
\]
Finally, for a $\tau$-complex $(\Omega, i)$ in $\Kob^\tau(\emptyset)$, define an automorphism $I_\tau$ on $\Omega$ by the composition
\[
    \begin{tikzcd}
    \Omega \arrow[r, "J_\tau"]
    & \tau \Omega \arrow[r, "i^{-1}"]
    & \Omega.
    \end{tikzcd}
\]

\begin{proposition}
\label{prop:I-tau-involution}
    $I_\tau$ is an involution on $\Omega$. 
\end{proposition}

\begin{proof}
    Compute 
    \begin{align*}
        I_\tau^2 
            &= i^{-1} (J_\tau i^{-1} J_\tau) \\ 
            &= i^{-1} \tau(i^{-1}) \\ 
            &= i^{-1} (i^{-1})^{-1} \\ 
            &= 1.
    \end{align*}
\end{proof}

\begin{lemma}
\label{lem:tau-induced-maps}
    A $\tau$-map $(f, F)\colon (\Omega, i) \to (\Omega', i')$ in $\Kob^\tau(\emptyset)$ gives a homotopy 
    \[
        \tilde{F} := -i'^{-1} F i^{-1} J_\tau
    \]
    exhibiting 
    \[
        I_\tau f \htpy f I_\tau
    \]
    in $\Kob(\emptyset)$. Furthermore, a homotopy $(h, H)$ between two $\tau$-maps $(f, F), (g, G)$ in $\Kob^\tau(\emptyset)$ gives a degree $-2$ homotopy
    \[
        \tilde{H} := -i'^{-1} H i^{-1} J_\tau
    \]
    exhibiting
    \[
        \tilde{F} - \tilde{G} \htpy I_\tau h - h I_\tau
    \]
    in $\Kob(\emptyset)$.
\end{lemma}

\begin{proof}
    Compute 
    \[
        [d, \tilde{F}] = -i'^{-1} [d, F] i^{-1} J_\tau = -(f i^{-1} - i'^{-1}\tau(f)) J_\tau = -f I_\tau + I_\tau f
    \]
    and similarly for the second statement. 
\end{proof}

The following diagram describes the homotopy $\tilde{F}$ given in \Cref{lem:tau-induced-maps}, where $\bar{F} = -i'^{-1} F i^{-1}$.
\[
\begin{tikzcd}[row sep = 3em]
    \Omega \arrow[r, "f"] \arrow[d, "J_\tau"'] \arrow[dd, "I_\tau"', bend right=45] & \Omega' \arrow[d, "J_\tau"] \arrow[ld, equal] \arrow[dd, "I_\tau", bend left=45] \\
    \tau \Omega \arrow[r, "\tau f"] \arrow[d, "i^{-1}"'] & \tau \Omega' \arrow[d, "i'^{-1}"] \arrow[ld, "\bar{F}"', Rightarrow] \\
    \Omega \arrow[r, "f"'] & \Omega' 
\end{tikzcd}
\]

\begin{proposition}
\label{prop:tau-equiv-to-cone}
    There is a functor
    \[
        \cal{Q}\colon \Kob^\tau(\emptyset) \to h\Kob(\emptyset)
    \]
    into the homotopy category of $\Kob(\emptyset)$, defined for each $\tau$-complex $\Omega^\tau = (\Omega, i)$ by
    \[
        \cal{Q}(\Omega^\tau) := \Cone(\Omega \xrightarrow{\ I - I_\tau \ } \Omega)
    \]
    and for each $\tau$-map $f^\tau = (f, F)$ by the homotopy class of
    \[
        \cal{Q}(f^\tau) := \begin{pmatrix}
            f & \\
            - \tilde{F} & f
        \end{pmatrix}
    \]
    where $\tilde{F}$ is the homotopy given in \Cref{lem:tau-induced-maps}. Furthermore, $\cal{Q}$ maps $\tau$-equivariantly homotopic $\tau$-maps to homotopic chain maps, and hence induces a functor
    \[
        \cal{Q}\colon h\Kob^\tau(\emptyset) \to h\Kob(\emptyset).
    \]
\end{proposition}

\begin{proof}
    The complex $\cal{Q}(\Omega^\tau)$ is given by the underlying object $\Omega \oplus \Omega[1]$ equipped with the differential
    \[
        D = \begin{pmatrix}
            d & \\
            I - I_\tau & -d
        \end{pmatrix}.
    \]
    That $\cal{Q}(f^\tau)$ commutes with $D$ is immediate from \Cref{lem:tau-induced-maps}. Also $\cal{Q}$ preserves the identity morphisms, since $F = 0$ implies $\tilde{F} = 0$.

    Next, we prove that $\cal{Q}$ preserves compositions up to homotopy. Take $\tau$-maps $f^\tau = (f, F)\colon (\Omega, i) \to (\Omega', i')$ and $g^\tau = (g, G)\colon (\Omega', i') \to (\Omega'', i'')$, and put $K := Gf + \tau(g)F$, so that $g^\tau f^\tau = (gf, K)$. Both $\cal{Q}(g^\tau f^\tau)$ and $\cal{Q}(g^\tau)\cal{Q}(f^\tau)$ are lower triangular with diagonal entries $gf$, so their difference is concentrated in the $(2,1)$-entry,
    \[
        \cal{Q}(g^\tau f^\tau) - \cal{Q}(g^\tau)\cal{Q}(f^\tau)
        = \begin{pmatrix}
            0 & \\
            -\Delta & 0
        \end{pmatrix},
        \quad
        \Delta := \tilde{K} - (\tilde{G}f + g\tilde{F}).
    \]
    Compute
    \begin{align*}
        \Delta
        &= -i''^{-1}(Gf + \tau(g)F)i^{-1}J_\tau + i''^{-1}Gi'^{-1} J_\tau f + gi'^{-1}Fi^{-1}J_\tau \\ 
        &= \left(\,-i''^{-1}G\,(fi^{-1} - i'^{-1}\tau(f)) + (gi'^{-1} - i''^{-1}\tau(g))\,Fi^{-1} \,\right) J_\tau \\
        &= \left(\,-i''^{-1}G\,i'^{-1}[d, F]i^{-1} + i''^{-1}[d, G]i'^{-1}\,Fi^{-1} \,\right) J_\tau.
    \end{align*}
    With
    \begin{align*}
        i''^{-1}G\,i'^{-1}[d, F]i^{-1} &= i''^{-1}(\, G\,i'^{-1}dF + G\,i'^{-1}Fd \,)\,i^{-1}, \\
        i''^{-1}[d, G]i'^{-1}\,Fi^{-1} &= i''^{-1}(\, dG\,i'^{-1}F + G\,i'^{-1}dF \,)\,i^{-1},
    \end{align*}
    the two terms involving $dF$ cancel, giving
    \begin{align*}
        \Delta
        &= i''^{-1}(\, dG\, i'^{-1}F - G\, i'^{-1}F d \,)\, i^{-1} J_\tau \\
        &= [d, i''^{-1} G\, i'^{-1} F\, i^{-1} J_\tau].
    \end{align*}
    Therefore, the homotopy
    \[
        \Theta := \begin{pmatrix}
            0 & \\
            i''^{-1} G\, i'^{-1} F\, i^{-1} J_\tau & 0
        \end{pmatrix},
    \]
    gives
    \[
        \cal{Q}(g^\tau f^\tau) - \cal{Q}(g^\tau)\cal{Q}(f^\tau) = [D, \Theta].
    \]

    Finally, let $(h, H)$ be a $\tau$-equivariant homotopy between $\tau$-maps $f^\tau = (f, F)$ and $g^\tau = (g, G)$, and let $\tilde{H}$ be the induced homotopy given in \Cref{lem:tau-induced-maps}. Define
    \[
        \cal{Q}(h^\tau) := \begin{pmatrix}
            h & \\
            \tilde{H} & -h
        \end{pmatrix}.
    \]
    The diagonal entries of $[D, \cal{Q}(h^\tau)]$ are $[d, h] = f - g$, and the $(2,1)$-entry is
    \begin{align*}
        (I - I_\tau)h - h(I - I_\tau) - (d\tilde{H} - \tilde{H}d)
        &= -(I_\tau h - h I_\tau + [d, \tilde{H}]) \\
        &= -(\tilde{F} - \tilde{G}),
    \end{align*}
    where the last equality is the second statement of \Cref{lem:tau-induced-maps}. Hence
    \[
        \cal{Q}(f^\tau) - \cal{Q}(g^\tau) = [D, \cal{Q}(h^\tau)].
    \]
\end{proof}

\begin{definition}
\label{def:inv-kh-complex}
    For an involutive link diagram $L$, we put
    \[
        \cal{Q}(L) := \cal{Q}([L]^\tau)
    \]
    and call it the \textit{formal involutive Khovanov complex} of $L$.
\end{definition}

\begin{proposition}
\label{prop:inv-kh-invariance}
    Let $L, L'$ be involutive link diagrams that are related by one of the involutive Reidemeister moves or the $I$-move. Then there is a homotopy equivalence
    \[
        \cal{Q}(L) \to \cal{Q}(L').
    \]
\end{proposition}

\begin{proof}
    By \Cref{prop:ir-invariance,prop:i-move-invariance}, there is a $\tau$-equivariant homotopy equivalence $[L]^\tau \to [L']^\tau$. Applying the functor $\cal{Q}$ of \Cref{prop:tau-equiv-to-cone} gives the desired homotopy equivalence.
\end{proof}

This completes the proof of \Cref{thm:inv-kh-invariance}.

\subsection{$\tau$-matched complexes}
\label{subsec:tau-matched-cpx}

Here, we introduce a more tractable class of $\tau$-complexes, namely the \textit{$\tau$-matched complexes}, and relate it with the involutive Khovanov complex of \cite[Definition 2.2]{Sano:2025b}.

\begin{definition}
\label{def:tau-matching}
    Let $\Omega$ be a complex in $\Kob(B)$, and let $S(\Omega)$ denote the set of summands of $\Omega$, each of whose elements is an object of $\Cob(B)$. A \textit{$\tau$-matching} on $\Omega$ is a self-bijection $j$ of $S(\Omega)$ satisfying the following conditions: (i) for each $X \in S(\Omega)$, $jX = \tau X$ as objects in $\Cob(B)$, (ii) for each component $d_{X, Y}\colon X \to Y$ of the differential $d$ of $\Omega$, the corresponding component $d_{jX, jY}\colon jX \to jY$ is given by $\tau(d_{X, Y})$, and (iii) $j^2 = \id_{S(\Omega)}$. A complex $\Omega$ equipped with a $\tau$-matching $j$ is called a \textit{$\tau$-matched complex}. 
\end{definition}

Note that the existence of a $\tau$-matching requires $B = \tau B$ and $S(\Omega) = \tau S(\Omega)$. 

\begin{proposition}
\label{prop:tau-matched-cpx}
    A $\tau$-matching $j$ on $\Omega$ induces an isomorphism $i\colon \Omega \to \tau \Omega$ such that $(\Omega, i)$ is a $\tau$-complex. 
\end{proposition}

\begin{proof}
    For each $X \in S(\Omega)$, consider the summand $jX \in S(\Omega)$. From condition (i) of \Cref{def:tau-matching}, the image $\tau(jX)$ regarded as a summand in $\tau \Omega$ is identical to $X$ as objects in $\Cob(B)$. Define $i$ on $X$ to be the identity morphism $i_X\colon X \to \tau(jX)$. It follows from condition (ii) that $i$ is indeed a chain isomorphism. Furthermore, from condition (iii), the map $i$ on $jX \in S(\Omega)$ is given by the identity morphism $i_{jX}\colon jX \to \tau(j^2(X)) = \tau X$. Applying $\tau$ gives the identity morphism $\tau(i_{jX})\colon \tau(jX) \to \tau^2 X = X$, and hence $\tau(i_{jX}) \circ i_X = \id_X$, proving $\tau(i) \circ i = \id_\Omega$. 
\end{proof}

\begin{proposition}
\label{prop:tau-match-for-inv-tangle}
    For an involutive tangle diagram $T$, the complex $[T]$ admits a $\tau$-matching $j$ such that its induced isomorphism $i\colon[T] \to \tau[T]$ coincides with the one given in \Cref{prop:inv-tng-assoc-isom}.
\end{proposition}

\begin{proof}
    The matching $j$ is given by the action of $\tau$ on the states of $T$. The conditions of \Cref{def:tau-matching} and the coincidence of the induced isomorphism with that of \Cref{prop:inv-tng-assoc-isom} are immediate from the constructions.
\end{proof}

\begin{definition}
    The $\tau$-matching $j$ stated in \Cref{prop:tau-match-for-inv-tangle} is called the \textit{associated $\tau$-matching} for $[T]$. 
\end{definition}

\begin{example}
    In \Cref{fig:tau-equiv-cpx}, the associated $\tau$-matching for $[T]$ is given by 
    \begin{gather*}
        000 \leftrightarrow 000,\quad 100 \leftrightarrow 010, \quad 001 \leftrightarrow 001,\\
        110 \leftrightarrow 110,\quad 101 \leftrightarrow 011, \quad 111 \leftrightarrow 111.
    \end{gather*}
\end{example}

\begin{definition}
\label{def:tau-pres-map}
    Let $(\Omega, j), (\Omega', j')$ be $\tau$-matched complexes. A chain map $f\colon \Omega \to \Omega'$ \textit{preserves} the $\tau$-matchings if for each pair of objects $X \in S(\Omega)$ and $Y \in S(\Omega')$, the components $f_{X, Y}\colon X \to Y$ and $f_{jX, j'Y}\colon jX \to j'Y$ satisfy $\tau(f_{X, Y}) = f_{jX, j'Y}$. \textit{$\tau$-preserving chain homotopies} are defined similarly.
\end{definition}

\begin{proposition}
\label{prop:tau-matched-map}
    A $\tau$-preserving chain map $f\colon \Omega \to \Omega'$ induces a strictly equivariant $\tau$-map $(f, 0) \colon (\Omega, i) \to (\Omega', i')$ between the corresponding $\tau$-complexes. A similar statement holds for $\tau$-preserving chain homotopies.
\end{proposition}

\begin{proof}
    For each pair of objects $X \in S(\Omega)$ and $Y \in S(\Omega')$, the following diagram strictly commutes. 
    \[
    \begin{tikzcd}
        X \arrow[r, "{f_{X,Y}}"] \arrow[d, "i_X"'] & Y \arrow[d, "i'_Y"] \\
        \tau(jX) \arrow[r, "{\tau(f_{jX, j'Y})}"] & \tau(j'Y).
    \end{tikzcd}
    \]
\end{proof}

\begin{definition}
    $\tau$-matched complexes and $\tau$-preserving chain maps form the category of $\tau$-matched complexes, denoted $\Kob^{\tau m}(B)$. 
\end{definition}

Combining \Cref{prop:tau-matched-cpx,prop:tau-match-for-inv-tangle,prop:tau-matched-map}, we obtain the following. 

\begin{proposition}
    Over $\FF_2$, there is a functor 
    \[
        \Kob^{\tau m}(B) \to \Kob^\tau_\bullet(B)
    \]
    such that the following diagram commutes:
    \[
    \begin{tikzcd}
        \Diag^\tau(B) \arrow[rr, "{[\cdot]^\tau}"] \arrow[rd, dashed] & & \Kob^\tau(B). \\
        & \Kob^{\tau m}(B) \arrow[ru] & 
    \end{tikzcd}
    \]
    Here, the dashed arrow indicates that the correspondence is defined only on objects. The functor descends to their homotopy categories.
\end{proposition}

Next, we focus on the case $B = \emptyset$. For a $\tau$-matched complex $(\Omega, j)$ in $\Kob^{\tau m}(\emptyset)$ (regarded as a $\tau$-complex), the involution $I_\tau$ on $\Omega$ is defined by
\[
\begin{tikzcd}
    \Omega \arrow[r, "J_\tau"]
    & \tau \Omega \arrow[r, "i^{-1}"]
    & \Omega.
\end{tikzcd}
\]
By construction, its restriction to a summand $X \in S(\Omega)$ is simply
\[
\begin{tikzcd}
    X \arrow[r, "J_\tau"]
    & \tau X \arrow[r, "\id"]
    & jX.
\end{tikzcd}
\]
In particular, when $\Omega = [L]$ for an involutive link diagram $L$, applying the tautological functor $\cal{F}$ to $[L]$ gives the Khovanov complex $\CKh(L)$ (over $\FF_2$), and to $I_\tau$ yields the chain automorphism defined in \cite[Section 2.1]{Sano:2025b} (there, the automorphism is denoted $\tau$). 

\begin{proposition}
\label{prop:recover-ckhi}
    Let $L$ be an involutive link diagram. The involutive Khovanov complex $\CKhI(L)$ of \cite[Definition 2.2]{Sano:2025b} is obtained as the image of $L$ under the following sequence of functors:
    \[
    \begin{tikzcd}
        \Diag^\tau(\emptyset) \arrow[r, "{[\cdot]^\tau}"] & h\Kob^\tau(\emptyset) \arrow[r, "\cal{Q}"] & h\Kob(\emptyset) \arrow[r, "\cal{F}"] & h\Kom({R\GrMod}).
    \end{tikzcd}
    \]
\end{proposition}

\begin{proof}
    The TQFT $\cal{F}$ commutes with the cone: 
    \[
        \cal{F}(\cal{Q}(L)) = \cal{F}(\Cone(I - I_\tau)) = \Cone(\cal{F}(I) - \cal{F}(I_\tau)) = \CKhI(L). 
    \]
\end{proof}

Together with \Cref{thm:inv-tang-invariance}, we recover \cite[Theorem 2, Proposition 2.7]{Sano:2025b}: the chain homotopy type of $\CKhI(L)$ is an invariant of the involutive link. Since $\CKhI(L)$ is a complex in the ordinary category of complexes, its homology is well-defined, and is the \textit{involutive Khovanov homology} $\KhI(L)$ of $L$, defined in \cite[Definition 2.2]{Sano:2025b}. 
    \section{Reducing the involutive Khovanov complex}
\label{sec:reducing}

The idea of Bar-Natan's fast Khovanov homology computation \cite{BN:2007} is to iteratively reduce the Khovanov complex while building it from smaller tangle pieces, using delooping and Gaussian elimination (\Cref{prop:delooping,prop:gauss-elim}). Here, we develop analogous techniques for the involutive Khovanov complex. 

\subsection{Symmetric tangle decomposition}
\label{subsec:sym-tangle-decomp}

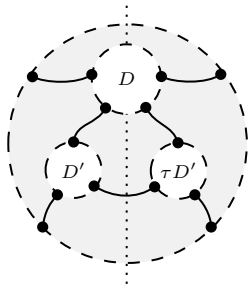
\begin{figure}[t]
    \centering
    \tikzset{every picture/.style={line width=0.75pt}} 

\begin{tikzpicture}[x=0.75pt,y=0.75pt,yscale=-1,xscale=1]

\draw  [fill={rgb, 255:red, 241; green, 241; blue, 241 }  ,fill opacity=1 ][dash pattern={on 4.5pt off 4.5pt}] (13.38,76.46) .. controls (13.38,43.4) and (40.18,16.6) .. (73.24,16.6) .. controls (106.31,16.6) and (133.11,43.4) .. (133.11,76.46) .. controls (133.11,109.52) and (106.31,136.33) .. (73.24,136.33) .. controls (40.18,136.33) and (13.38,109.52) .. (13.38,76.46) -- cycle ;
\draw  [fill={rgb, 255:red, 255; green, 255; blue, 255 }  ,fill opacity=1 ][dash pattern={on 4.5pt off 4.5pt}] (84.9,90.06) .. controls (84.9,82.3) and (91.19,76.01) .. (98.95,76.01) .. controls (106.71,76.01) and (113,82.3) .. (113,90.06) .. controls (113,97.82) and (106.71,104.11) .. (98.95,104.11) .. controls (91.19,104.11) and (84.9,97.82) .. (84.9,90.06) -- cycle ;
\draw  [fill={rgb, 255:red, 0; green, 0; blue, 0 }  ,fill opacity=1 ] (96.26,76.01) .. controls (96.26,74.75) and (97.28,73.72) .. (98.55,73.72) .. controls (99.82,73.72) and (100.84,74.75) .. (100.84,76.01) .. controls (100.84,77.28) and (99.82,78.31) .. (98.55,78.31) .. controls (97.28,78.31) and (96.26,77.28) .. (96.26,76.01) -- cycle ;
\draw  [fill={rgb, 255:red, 0; green, 0; blue, 0 }  ,fill opacity=1 ] (84.54,98.11) .. controls (84.54,96.85) and (85.57,95.82) .. (86.84,95.82) .. controls (88.1,95.82) and (89.13,96.85) .. (89.13,98.11) .. controls (89.13,99.38) and (88.1,100.41) .. (86.84,100.41) .. controls (85.57,100.41) and (84.54,99.38) .. (84.54,98.11) -- cycle ;
\draw    (46.15,75.61) .. controls (47.4,65.66) and (57.8,67.25) .. (62.19,58.54) ;
\draw    (56.44,97.71) .. controls (65.86,103.66) and (77.86,104.51) .. (86.84,98.11) ;
\draw  [fill={rgb, 255:red, 255; green, 255; blue, 255 }  ,fill opacity=1 ][dash pattern={on 4.5pt off 4.5pt}] (32.1,89.66) .. controls (32.1,81.9) and (38.39,75.61) .. (46.15,75.61) .. controls (53.91,75.61) and (60.2,81.9) .. (60.2,89.66) .. controls (60.2,97.42) and (53.91,103.71) .. (46.15,103.71) .. controls (38.39,103.71) and (32.1,97.42) .. (32.1,89.66) -- cycle ;
\draw  [fill={rgb, 255:red, 0; green, 0; blue, 0 }  ,fill opacity=1 ] (43.86,75.61) .. controls (43.86,74.35) and (44.88,73.32) .. (46.15,73.32) .. controls (47.42,73.32) and (48.44,74.35) .. (48.44,75.61) .. controls (48.44,76.88) and (47.42,77.91) .. (46.15,77.91) .. controls (44.88,77.91) and (43.86,76.88) .. (43.86,75.61) -- cycle ;
\draw  [fill={rgb, 255:red, 0; green, 0; blue, 0 }  ,fill opacity=1 ] (54.14,97.71) .. controls (54.14,96.45) and (55.17,95.42) .. (56.44,95.42) .. controls (57.7,95.42) and (58.73,96.45) .. (58.73,97.71) .. controls (58.73,98.98) and (57.7,100.01) .. (56.44,100.01) .. controls (55.17,100.01) and (54.14,98.98) .. (54.14,97.71) -- cycle ;
\draw  [fill={rgb, 255:red, 0; green, 0; blue, 0 }  ,fill opacity=1 ] (117.74,41.68) .. controls (117.74,40.41) and (118.76,39.39) .. (120.03,39.39) .. controls (121.3,39.39) and (122.32,40.41) .. (122.32,41.68) .. controls (122.32,42.94) and (121.3,43.97) .. (120.03,43.97) .. controls (118.76,43.97) and (117.74,42.94) .. (117.74,41.68) -- cycle ;
\draw  [fill={rgb, 255:red, 0; green, 0; blue, 0 }  ,fill opacity=1 ] (23.21,42.86) .. controls (23.21,41.6) and (24.24,40.57) .. (25.5,40.57) .. controls (26.77,40.57) and (27.8,41.6) .. (27.8,42.86) .. controls (27.8,44.13) and (26.77,45.16) .. (25.5,45.16) .. controls (24.24,45.16) and (23.21,44.13) .. (23.21,42.86) -- cycle ;
\draw    (25.5,42.86) .. controls (31.8,46.05) and (48.2,46.85) .. (55.31,41.77) ;
\draw    (90.23,42.77) .. controls (96.52,45.96) and (112.92,46.76) .. (120.03,41.68) ;
\draw  [dash pattern={on 0.84pt off 2.51pt}]  (72.8,2.75) -- (72.8,150.05) ;
\draw  [fill={rgb, 255:red, 255; green, 255; blue, 255 }  ,fill opacity=1 ][dash pattern={on 4.5pt off 4.5pt}] (55.31,44.06) .. controls (55.31,34.42) and (63.12,26.6) .. (72.77,26.6) .. controls (82.41,26.6) and (90.23,34.42) .. (90.23,44.06) .. controls (90.23,53.71) and (82.41,61.52) .. (72.77,61.52) .. controls (63.12,61.52) and (55.31,53.71) .. (55.31,44.06) -- cycle ;

\draw  [fill={rgb, 255:red, 0; green, 0; blue, 0 }  ,fill opacity=1 ] (53.01,41.77) .. controls (53.01,40.5) and (54.04,39.48) .. (55.31,39.48) .. controls (56.57,39.48) and (57.6,40.5) .. (57.6,41.77) .. controls (57.6,43.04) and (56.57,44.06) .. (55.31,44.06) .. controls (54.04,44.06) and (53.01,43.04) .. (53.01,41.77) -- cycle ;
\draw  [fill={rgb, 255:red, 0; green, 0; blue, 0 }  ,fill opacity=1 ] (87.93,42.77) .. controls (87.93,41.5) and (88.96,40.48) .. (90.23,40.48) .. controls (91.49,40.48) and (92.52,41.5) .. (92.52,42.77) .. controls (92.52,44.04) and (91.49,45.06) .. (90.23,45.06) .. controls (88.96,45.06) and (87.93,44.04) .. (87.93,42.77) -- cycle ;
\draw  [fill={rgb, 255:red, 0; green, 0; blue, 0 }  ,fill opacity=1 ] (59.9,58.54) .. controls (59.9,57.28) and (60.93,56.25) .. (62.19,56.25) .. controls (63.46,56.25) and (64.49,57.28) .. (64.49,58.54) .. controls (64.49,59.81) and (63.46,60.84) .. (62.19,60.84) .. controls (60.93,60.84) and (59.9,59.81) .. (59.9,58.54) -- cycle ;
\draw  [fill={rgb, 255:red, 0; green, 0; blue, 0 }  ,fill opacity=1 ] (80.1,58.28) .. controls (80.1,57.01) and (81.12,55.99) .. (82.39,55.99) .. controls (83.66,55.99) and (84.68,57.01) .. (84.68,58.28) .. controls (84.68,59.55) and (83.66,60.57) .. (82.39,60.57) .. controls (81.12,60.57) and (80.1,59.55) .. (80.1,58.28) -- cycle ;
\draw    (98.55,76.01) .. controls (95.8,66.86) and (87.38,67.67) .. (82.39,58.28) ;
\draw  [fill={rgb, 255:red, 0; green, 0; blue, 0 }  ,fill opacity=1 ] (28.35,118.29) .. controls (28.35,117.02) and (29.38,116) .. (30.65,116) .. controls (31.91,116) and (32.94,117.02) .. (32.94,118.29) .. controls (32.94,119.55) and (31.91,120.58) .. (30.65,120.58) .. controls (29.38,120.58) and (28.35,119.55) .. (28.35,118.29) -- cycle ;
\draw    (30.65,118.29) .. controls (31.8,110.06) and (34.29,106.3) .. (38.02,102.13) ;
\draw  [fill={rgb, 255:red, 0; green, 0; blue, 0 }  ,fill opacity=1 ] (35.73,102.13) .. controls (35.73,100.86) and (36.75,99.83) .. (38.02,99.83) .. controls (39.29,99.83) and (40.31,100.86) .. (40.31,102.13) .. controls (40.31,103.39) and (39.29,104.42) .. (38.02,104.42) .. controls (36.75,104.42) and (35.73,103.39) .. (35.73,102.13) -- cycle ;
\draw  [fill={rgb, 255:red, 0; green, 0; blue, 0 }  ,fill opacity=1 ] (112.93,118.29) .. controls (112.93,117.02) and (113.96,116) .. (115.22,116) .. controls (116.49,116) and (117.52,117.02) .. (117.52,118.29) .. controls (117.52,119.55) and (116.49,120.58) .. (115.22,120.58) .. controls (113.96,120.58) and (112.93,119.55) .. (112.93,118.29) -- cycle ;
\draw  [fill={rgb, 255:red, 0; green, 0; blue, 0 }  ,fill opacity=1 ] (104.73,102.13) .. controls (104.73,100.86) and (105.75,99.83) .. (107.02,99.83) .. controls (108.29,99.83) and (109.31,100.86) .. (109.31,102.13) .. controls (109.31,103.39) and (108.29,104.42) .. (107.02,104.42) .. controls (105.75,104.42) and (104.73,103.39) .. (104.73,102.13) -- cycle ;
\draw    (115.22,118.29) .. controls (115,112.06) and (111,105.26) .. (107.02,102.13) ;

\draw (72.77,44.06) node  [font=\scriptsize]  {$D$};
\draw (98.95,90.06) node  [font=\scriptsize]  {$\tau D'$};
\draw (46.15,89.66) node  [font=\scriptsize]  {$D'$};

\end{tikzpicture}
    \caption{A symmetric planar arc diagram.}
    \label{fig:planar-diagram}
\end{figure}

\begin{definition}    
    A \textit{$d$-input planar arc diagram} $D$ is a disk with $d$ smaller disks $D_1, \ldots, D_d$ removed, together with a collection of disjoint properly embedded oriented arcs and circles that have boundary in $\del D$. The outer disk is called the \textit{output disk}, and the inner disks are called the \textit{input disks}. A planar arc diagram $D$ is \textit{symmetric} if $\tau D = D$ up to reordering of the input disks (see \Cref{fig:planar-diagram}). 
\end{definition}

The input disks of a symmetric planar arc diagram can be separated into two types: those that intersect the $y$-axis, called the \textit{on-axis input disks}, and those that do not, called the \textit{off-axis input disks}. Each on-axis input disk $D_i$ satisfies $\tau D_i = D_i$, and each off-axis input disk $D'_j$ has a unique counterpart $\tau D'_j$. Let $d_s, d_a$ denote the numbers of on-axis and off-axis input disks. We usually order the $d = d_s + 2d_a$ input disks of $D$ as 
\[
    (\underbrace{D_1, \ldots, D_{d_s}}_{\text{center}}; \,\underbrace{D'_1, \ldots, D'_{d_a}}_{\text{left}}; \,\underbrace{\tau D'_1, \ldots, \tau D'_{d_a}}_{\text{right}})
\]
so that $D_1, \ldots, D_{d_s}$ are on-axis and the remaining ones are off-axis, with $D'_1, \ldots, D'_{d_a}$ on the left side of $D$ and $\tau D'_1, \ldots, \tau D'_{d_a}$ on the right. The input disks of $\tau D$ are then ordered as
\[
    (\underbrace{D_1, \ldots, D_{d_s}}_{\text{center}}; \,\underbrace{\tau D'_1, \ldots, \tau D'_{d_a}}_{\text{right}}; \,\underbrace{D'_1, \ldots, D'_{d_a}}_{\text{left}}).
\]

A \textit{basic symmetric planar arc diagram} $D$ is a $3$-input symmetric planar arc diagram with $d_s = d_a = 1$ (as in \Cref{fig:planar-diagram}). Observe that any symmetric planar arc diagram can be expressed as a horizontal composition of a basic symmetric planar arc diagram and three smaller planar arc diagrams, by grouping the input disks into three: the on-axis disks, the off-axis disks on the left, and their counterparts on the right (allowing each group to be empty).

Suppose $D$ is a basic symmetric planar arc diagram with input disks $(D_s, D_a, \tau D_a)$. Given an involutive tangle diagram $T_s$ that matches $D_s$ (i.e. $\partial T_s$ coincides with the boundary points on $\partial D_s$) and an arbitrary tangle diagram $T_a$ that matches $D_a$, the composition
\[
    T = D(T_s, T_a, \tau T_a)
\]
is an involutive tangle diagram. Such a description of an involutive tangle diagram $T$ is called a \textit{basic symmetric tangle decomposition}. Applying $\tau$ gives a symmetric tangle decomposition for $\tau T$, 
\[
    \tau T = (\tau D)(\tau T_s, \tau T_a, T_a). 
\]
Our aim is to reduce the $\tau$-complex $[T]^\tau$ by reducing the central $\tau$-complex $[T_s]^\tau$ equivariantly, reducing the complex $[T_a]$ non-equivariantly, and recombining them using $D$. We shall develop such techniques for general $\tau$-complexes.

Suppose we are given a $\tau$-complex $(\Omega_s, i_s)$ that matches $D_s$ and an arbitrary complex $\Omega_a$ that matches $D_a$. Consider the composite complex
\[
    \Omega = D(\Omega_s, \Omega_a, \tau \Omega_a).
\]
Define
\[
    i: \Omega \xrightarrow{\ \isom\ } \tau \Omega
\]
by the composition
\[
    \begin{tikzcd}[column sep=3.25em]
    \Omega \arrow[d, equal] 
        \arrow[rr, "i"]
    & & \tau\Omega 
        \arrow[d, equal]\\
    {D(\Omega_s, \Omega_a, \tau \Omega_a)} 
        \arrow[r, "{D(i_s, I, I)}"]
    & {D(\tau \Omega_s, \Omega_a, \tau \Omega_a) }
        \arrow[r, "R"]
    & {(\tau D)(\tau \Omega_s, \tau \Omega_a, \Omega_a).}
    \end{tikzcd}
\]
Here, the second chain isomorphism $R$ is given by the permutation matrices that realize the reordering of the direct summands of the off-axis complexes. The following two lemmas are obvious from the construction.

\begin{lemma}
\label{lem:tau-complex-from-sym-decomp}
    $(\Omega, i) = (D(\Omega_s, \Omega_a, \tau \Omega_a), R \circ D(i_s, I, I))$ is a $\tau$-complex.
\end{lemma}

\begin{lemma}
    For a symmetric tangle decomposition $T = D(T_s, T_a, \tau T_a)$, the $\tau$-complex $[T]^\tau$ coincides with the complex $D([T_s], [T_a], \tau [T_a])$ equipped with the isomorphism $i$ of \Cref{lem:tau-complex-from-sym-decomp}.
\end{lemma}

Hereafter, we always assume that such a composite complex $\Omega$ is equipped with the above identification when regarded as a $\tau$-complex.

The following proposition states that equivariant chain maps between on-axis complexes, and chain maps between off-axis complexes induce equivariant chain maps between the composite $\tau$-complexes. The same holds for chain homotopies. In particular, equivariant homotopy equivalences between on-axis complexes and homotopy equivalences between off-axis complexes both induce equivariant homotopy equivalences on the composite $\tau$-complex. 

\begin{proposition}
\label{prop:sym-decomp-maps}
    Given a basic symmetric planar arc diagram $D$ with input disks $(D_s, D_a, \tau D_a)$, suppose $\Omega_s, \Omega_s'$ are $\tau$-complexes that match $D_s$, and $\Omega_a, \Omega_a'$ are complexes that match $D_a$.
    \begin{enumerate}
        \item A $\tau$-map
        \[
            f_s^\tau = (f_s, F_s)\colon
                \Omega_s \to \Omega_s'
        \]
        induces a $\tau$-map 
        \[
            \bar{f}_s^\tau\colon
                D(\Omega_s, \Omega_a, \tau \Omega_a) \to
                D(\Omega_s', \Omega_a, \tau \Omega_a)
        \]
        where $\bar{f}_s^\tau = (D(f_s, I, I),\ R \circ D(F_s, I, I))$.
        
        \item A $\tau$-homotopy $h_s^\tau = (h_s, H_s)$ between $\tau$-maps
        \[
            f_s^\tau \htpy f_s'^\tau\colon
                \Omega_s \to \Omega_s'
        \]
        induces a $\tau$-homotopy $\bar{h}_s^\tau = (\bar{h}_s, \bar{H}_s)$ between the induced $\tau$-maps
        \[
            \bar{f}_s^\tau \htpy \bar{f}_s'^\tau \colon
                D(\Omega_s, \Omega_a, \tau \Omega_a) \to
                D(\Omega_s', \Omega_a, \tau \Omega_a)
        \]
        where $\bar{h}_s = D(h_s, I, I),\ \bar{H}_s = R \circ D(H_s, I, I)$.

        \item A chain map
        \[
            f_a\colon \Omega_a \to \Omega_a'
        \]
        induces a $\tau$-map
        \[
            \bar{f}_a^\tau\colon
                D(\Omega_s, \Omega_a, \tau \Omega_a) \to
                D(\Omega_s, \Omega_a', \tau \Omega_a')
        \]
        where $\bar{f}_a^\tau = (D(I, f_a, \tau f_a), 0)$.

        \item A homotopy $h_a$ between chain maps
        \[
            f_a \htpy f_a'\colon
                \Omega_a \to \Omega_a'
        \]
        induces a $\tau$-homotopy $\bar{h}_a^\tau = (\bar{h}_a, \bar{H}_a)$ between the induced equivariant chain maps
        \[
            \bar{f}_a^\tau \htpy \bar{f}_a'^\tau \colon
                D(\Omega_s, \Omega_a, \tau \Omega_a) \to
                D(\Omega_s, \Omega_a', \tau \Omega_a')
        \]
        where
        \[
            \bar{h}_a = D(I, h_a, \tau f_a) + D(I, f_a', \tau h_a),\quad
            \bar{H}_a = R \circ D(i_s, h_a, \tau h_a).
        \]
    \end{enumerate}
\end{proposition}

\begin{proof}
    \begin{enumerate}
        \item Consider the following homotopy commutative diagram:
        \[
            \begin{tikzcd}[row sep=3em, column sep=4em]
            {D(\Omega_s, \Omega_a, \tau \Omega_a) }
                \arrow[r, "{D(f_s, I, I)}"] 
                \arrow[d, "{D(i_s, I, I)}"'] &
            {D(\Omega_s', \Omega_a, \tau \Omega_a)} 
                \arrow[d, "{D(i_s', I, I)}"]
                \arrow[ld, "{D(F_s, I, I)}"', Rightarrow] \\
            {D(\tau \Omega_s, \Omega_a, \tau \Omega_a)} 
                \arrow[r, "{D(\tau f_s, I, I)}"]
                \arrow[d, "R"'] & 
            {D(\tau \Omega_s', \Omega_a, \tau \Omega_a)}
                \arrow[d, "R"] 
                \arrow[ld, equal] \\
            {(\tau D)(\tau \Omega_s, \tau \Omega_a, \Omega_a)}
                \arrow[r, "{(\tau D)(\tau f_s, I, I)}"] &
            {(\tau D)(\tau \Omega_s', \tau \Omega_a, \Omega_a)}.
            \end{tikzcd}
        \]
        \item Consider a diagram similar to that of item (1).

        \item Consider the following strictly commutative diagram:
        \[
            \begin{tikzcd}[row sep=3em, column sep=4em]
            {D(\Omega_s, \Omega_a, \tau \Omega_a)} 
                \arrow[r, "{D(I, f_a, \tau f_a)}"] 
                \arrow[d, "{D(i_s, I, I)}"'] & 
            {D(\Omega_s, \Omega_a', \tau \Omega_a')} 
                \arrow[d, "{D(i_s, I, I)}"] 
                \arrow[ld, equal] \\
            {D(\tau \Omega_s, \Omega_a, \tau \Omega_a)} 
                \arrow[r, "{D(I, f_a, \tau f_a)}"] 
                \arrow[d, "R"'] & 
            {D(\tau \Omega_s, \Omega_a', \tau \Omega_a')}
                \arrow[d, "R"] 
                \arrow[ld, equal] \\
            {(\tau D)(\tau \Omega_s, \tau \Omega_a, \Omega_a)} 
                \arrow[r, "{(\tau D)(I, \tau f_a, f_a)}"] &
            {(\tau D)(\tau \Omega_s, \tau \Omega_a', \Omega_a').}
            \end{tikzcd}
        \]
        \item With $f_a - f_a' = [d, h_a]$, we have
        \begin{align*}
            D(I, f_a, \tau f_a) - D(I, f_a', \tau f_a')
            &= D(I, f_a - f_a', \tau f_a) + D(I, f_a', \tau (f_a - f_a')) \\
            &= D(I, [d, h_a], \tau f_a) + D(I, f_a', \tau [d, h_a]) \\
            &= [d_D, D(I, h_a, \tau f_a) + D(I, f_a', \tau h_a)] \\
            &= [d_D, \bar{h}_a].
        \end{align*}
        For the second component, note that $i_s$ is a chain map, so that
        \begin{align*}
            i' \bar{h}_a - \tau(\bar{h}_a)\, i
            &= R \circ \left(\, D(i_s, h_a, \tau(f_a - f_a')) + D(i_s, f_a - f_a', \tau h_a) \,\right) \\
            &= R \circ \left(\, D(i_s, h_a, \tau [d, h_a]) + D(i_s, [d, h_a], \tau h_a) \,\right) \\
            &= [d_D,\, R \circ D(i_s, h_a, \tau h_a)]
            = [d_D, \bar{H}_a],
        \end{align*}
        which gives the required relation. \qedhere
    \end{enumerate}
\end{proof}

\subsection{Equivariant delooping and Gaussian elimination}

Hereafter, we give explicit $\tau$-equivariant reduction methods for $\tau$-matched complexes (see \Cref{subsec:tau-matched-cpx}) by applying delooping and Gaussian elimination equivariantly. For a $\tau$-matched complex $(\Omega, j)$, we make the matching $j$ implicit, and for each $X \in S(\Omega)$, we abuse notation and let $\tau X$ denote $jX \in S(\Omega)$.

\begin{definition}
    For a $\tau$-matched complex $\Omega$, a summand $X \in S(\Omega)$ is \textit{$\tau$-invariant} if $\tau X = X$.
\end{definition}

\begin{remark}
    The component $d_{X, Y}$ of the differential $d$ between $\tau$-invariant summands $X, Y$ is necessarily $\tau$-invariant, i.e. $\tau(d_{X, Y}) = d_{X, Y}$. 
\end{remark}

\begin{example}
    In \Cref{fig:tau-equiv-cpx}, the $\tau$-invariant summands of $[T]$ are $000, 001, 110, 111$. Note that the cobordisms $000 \to 001$ and $110 \to 111$ are $\tau$-invariant. 
\end{example}

\begin{proposition}[Equivariant delooping]
\label{prop:equiv-deloop}
    Suppose $\Omega$ is a $\tau$-matched complex, $X$ is a summand of $\Omega$, and $C$ is a circle in $X$.
    \begin{enumerate}
        \item If $X$ is $\tau$-invariant and $C$ is symmetric (i.e. $\tau C = C$), then delooping $C$ in $X$ gives a $\tau$-matched complex $\Omega'$ that is isomorphic to $\Omega$. 
        \item If $X$ is $\tau$-invariant and $C$ is asymmetric, then $C' = \tau C$ is another circle in $X$ disjoint from $C$, and delooping both $C$ and $C'$ in $X$ gives a $\tau$-matched complex $\Omega'$ that is isomorphic to $\Omega$. 
        \item If $X$ is not $\tau$-invariant, then $\tau X$ contains $\tau C$, and delooping both $C$ in $X$ and $\tau C$ in $\tau X$ gives a $\tau$-matched complex $\Omega'$ that is isomorphic to $\Omega$.
    \end{enumerate}
    In all three cases, the isomorphisms preserve the $\tau$-matchings. 
\end{proposition}

\begin{proof}
    Let $j$ denote the $\tau$-matching of $\Omega$. 
    \begin{enumerate}
        \item Let $X'$ denote the diagram obtained from $X$ by removing $C$. The delooping isomorphism gives $X \isom qX' \oplus q^{-1}X'$. Let $\Omega'$ be the resulting complex. The matching is defined on the two summands as identity so that they are $\tau$-invariant, and by $j$ on the complement. Hence the isomorphism preserves the $\tau$-matching.

        \item Let $X'$ denote the diagram obtained from $X$ by removing $C$ and $C'$. The delooping isomorphism gives
        \begin{align*}
            X 
            &\isom q(X \setminus C) \oplus q^{-1}(X \setminus C) \\
            &\isom (q^2 X' \oplus X') \oplus (X' \oplus q^{-2}X').
        \end{align*}
        Let $\Omega'$ be the resulting complex. The matching $j'$ is defined on the summands as in the following diagram, where the upper half depicts $f\colon \Omega \to \Omega'$ restricted to $X$, the lower half its $\tau$-image, and the middle vertical arrows the matching $j'$ on $S(\Omega')$.

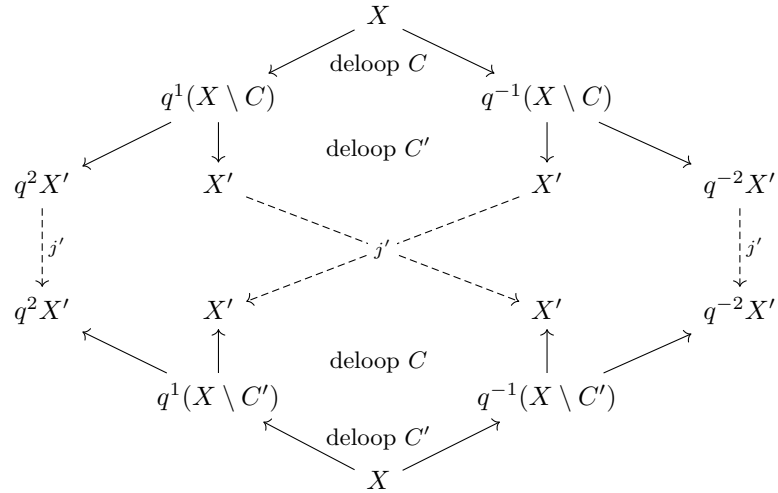
        \[
        \begin{tikzcd}[row sep=1.5em]
            & & X \arrow[ld] \arrow[rd] \arrow[d, "\small \text{deloop $C$}", phantom] & & \\
            & q^1 (X \setminus C) \arrow[ld] \arrow[d] & {} \arrow[d, "\small \text{deloop $C'$}", phantom] & q^{-1} (X \setminus C) \arrow[d] \arrow[rd] & \\
            q^2 X' \arrow[dd, "j'", dashed] & X' \arrow[rrdd, dashed] & {} & X' \arrow[lldd, "j'" description, dashed] & q^{-2}X' \arrow[dd, "j'", dashed] \\
            \\
            q^2 X' & X' & {} & X' & q^{-2}X' \\
            & q^1 (X \setminus C') \arrow[lu] \arrow[u] & {} \arrow[d, "\small \text{deloop $C'$}", phantom] \arrow[u, "\small \text{deloop $C$}", phantom] & q^{-1} (X \setminus C') \arrow[u] \arrow[ru] & \\
            & & X \arrow[lu] \arrow[ru] & & 
        \end{tikzcd}
        \]
        One verifies that $j'$ is indeed a $\tau$-matching, and that the isomorphism preserves the matchings. 

        \item Put $Y = \tau X$. Let $X'$ denote the diagram obtained from $X$ by removing $C$, and $Y'$ the diagram obtained from $Y$ by removing $C' = \tau C$. By the delooping isomorphism, we have
        \[
            X \oplus Y \isom (qX' \oplus q^{-1}X') \oplus (qY' \oplus q^{-1}Y').
        \]
        Let $\Omega'$ be the resulting complex. The matching $j'$ is defined as in the following diagram. 
        
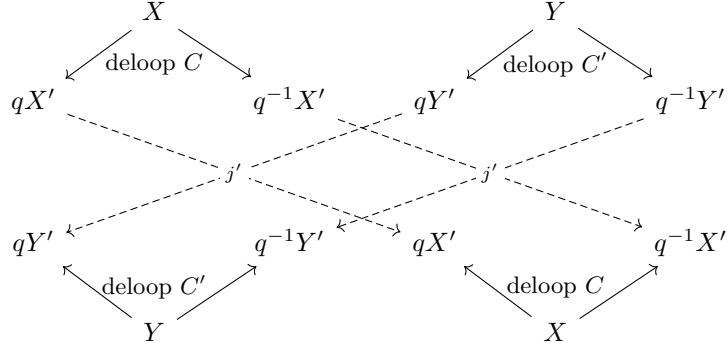
\[
        \begin{tikzcd}
            & X \arrow[ld] \arrow[rd] \arrow[d, "\small \text{deloop $C$}", phantom] & & & Y \arrow[ld] \arrow[rd] \arrow[d, "\small \text{deloop $C'$}", phantom] & \\
            qX' \arrow[rrrdd, dashed] & {} & q^{-1}X' \arrow[rrrdd, dashed] & qY' \arrow[llldd, "j'" description, dashed] & {} & q^{-1}Y' \arrow[llldd, "j'" description, dashed] \\
            & & & & & \\
            qY' & {} & q^{-1}Y' & qX' & {} & q^{-1}X' \\
            & Y \arrow[lu] \arrow[ru] \arrow[u, "\small \text{deloop $C'$}", phantom] & & & X \arrow[lu] \arrow[ru] \arrow[u, "\small \text{deloop $C$}", phantom] & 
        \end{tikzcd}
        \]
        Again, one verifies the necessary conditions. \qedhere
    \end{enumerate}
\end{proof}

\begin{proposition}[Equivariant Gaussian elimination]
\label{prop:equiv-gauss-elim}
    Suppose $\Omega$ is a $\tau$-matched complex that contains summands $X, Y$ such that the component $d_{X, Y}: X \to Y$ of the differential $d$ is invertible. 
    \begin{enumerate}
        \item If both $X$ and $Y$ are $\tau$-invariant, then applying Gaussian elimination to $d_{X, Y}$ gives a strong deformation retract $\Omega'$ of $\Omega$ with a $\tau$-matching that is compatible with the strong deformation retraction. 

        \item If neither $X$ nor $Y$ is $\tau$-invariant, and the components of $d$ for $X \to \tau Y$ and $\tau X \to Y$ are both zero, then applying Gaussian elimination to $X \to Y$ and to $\tau X \to \tau Y$ gives a strong deformation retract $\Omega'$ of $\Omega$ with a $\tau$-matching that is compatible with the strong deformation retraction. 
    \end{enumerate}
\end{proposition}

\begin{proof} 
    Let $j$ denote the $\tau$-matching of $\Omega$, and put $a = d_{X, Y}\colon X \to Y$.
    \begin{enumerate}
        \item From the assumption, the morphism $a$ is necessarily $\tau$-invariant. The effect of applying Gaussian elimination to $a$ is that the summands $X, Y$ are removed, and morphisms are modified as in \Cref{prop:gauss-elim}, giving a strong deformation retract $\Omega'$ of $\Omega$. We define the matching $j'$ on $\Omega'$ by restricting $j$ to the remaining summands. To check that $j'$ satisfies the necessary conditions, consider any pair of summands $X', Y'$ of $\Omega$ (which may or may not be $\tau$-invariant) that forms the following arrows:
        \[
        \begin{tikzcd}[column sep = 3em]
            X \arrow[r, "a"] \arrow[rd, "c"', pos=.8] & Y \\
            X' \arrow[r, "d"'] \arrow[ru, "b", pos=.8] & Y'.
        \end{tikzcd}
        \]
        Then we have arrows
        \[
        \begin{tikzcd}[column sep = 3em]
            X \arrow[r, "a"] \arrow[rd, "\tau c"', pos=.8] & Y \\
            \tau X' \arrow[r, "\tau d"'] \arrow[ru, "\tau b", pos=.8] & \tau Y'.
        \end{tikzcd}
        \]
        After the elimination, the arrow between $X'$ and $Y'$ is given by $d - c a^{-1} b$, and the arrow between $\tau X'$ and $\tau Y'$ is given by 
        \[
            \tau d - (\tau c) a^{-1} (\tau b) = \tau (d - c a^{-1} b). 
        \]
        Thus $j'$ gives a matching on $\Omega'$, which is compatible with the maps defining the strong deformation retraction. 
        
        \item The proof is similar to the previous case, except that $X \xrightarrow{a} Y$ and $\tau X \xrightarrow{\tau a} \tau Y$ are distinct arrows. From the assumption, eliminating one of the arrows will not affect the other. \qedhere
    \end{enumerate}    
\end{proof}

\begin{remark}
    In the second case of \Cref{prop:equiv-gauss-elim}, the condition that $X \to \tau Y$ is zero is equivalent to $\tau X \to Y$ being zero. Without this assumption, eliminating the arrow $a$ will subtract a non-trivial morphism from the other arrow, which may make it non-invertible.
    \[
        \begin{tikzcd}[column sep = 3em]
            X \arrow[r, "a"] \arrow[rd, dashed] & Y \\
            {\tau X} \arrow[r, "\tau a"'] \arrow[ru, dashed] & {\tau Y}.
        \end{tikzcd}
    \]
    The remaining case, where exactly one of $X, Y$ is $\tau$-invariant, admits no equivariant elimination.
\end{remark}

\subsection{Symmetry-breaking reduction}

Finally, we restrict to the case $B = \emptyset$ and consider the reduction after constructing the cone. Recall 
\[
    \cal{Q}(\Omega) := \Cone(\Omega \xrightarrow{I - I_\tau} \Omega)
\]
defined in \Cref{prop:tau-equiv-to-cone}. To distinguish the two copies of each summand of $\Omega$ inside $\cal{Q}(\Omega)$, for each summand $X$ of $\Omega$, we let $\underline{X} = X$ and $\overline{X} = X[1]$ denote the corresponding pair of objects in $\cal{Q}(\Omega)$. If $X$ is not $\tau$-invariant, there are invertible arrows in $\cal{Q}(\Omega)$ of the form: 
\[
    \begin{tikzcd}[column sep = 3em]
    \underline{X} 
        \arrow[rd] \arrow[d, "I"'] & 
    \underline{\tau X} 
        \arrow[ld, "I_\tau" description] 
        \arrow[d, "I"] \\
    \overline{X} & 
    \overline{\tau X}.
    \end{tikzcd}
\]
If $X$ is $\tau$-invariant, there is an arrow in $\cal{Q}(\Omega)$ of the form: 
\[
    \begin{tikzcd}[column sep = 3em]
    \underline{X} 
        \arrow[d, "I - I_\tau"] \\
    \overline{X}
    \end{tikzcd}
\]
which is neither zero nor invertible in general. In particular, if all circles of $X$ are symmetric, then $I = I_\tau$.

\begin{proposition}
\label{prop:vertical-reduction}
    Let $\Omega$ be a $\tau$-matched complex in $\Kob(\emptyset)$. Suppose $X$ is a non-$\tau$-invariant summand of $\Omega$. Then we may apply Gaussian elimination to the invertible arrow
    \[
    \begin{tikzcd}
        {\underline{\tau X}} \arrow[r, "I"] & 
        {\overline{\tau X}},
    \end{tikzcd}
    \]
    in $\Gamma := \cal{Q}(\Omega)$ to get a strong deformation retract $\Gamma'$ of $\Gamma$. The effect of this elimination is as follows: suppose there are summands $Y, Z$ in $\Omega$ forming a partial diagram
    \[
        \begin{tikzcd}[row sep=.5em]
        & X \arrow[rd, "c"] & \\
        Y \arrow[ru, "a"] \arrow[rd, "b"'] & & Z. \\
        & \tau X \arrow[ru, "d"'] & 
        \end{tikzcd}        
    \]
    Then $\Gamma$ contains the following partial diagram
    \[
        \begin{tikzcd}
        \underline{Y} \arrow[rr, "a"] \arrow[rd, "b"'] & & \underline{X} \arrow[rd, "c"] \arrow[dd] \arrow[lddd] & \\
         & \underline{\tau X} \arrow[rr, "d"', pos=0.7] \arrow[dd, "I" description, pos=0.3] \arrow[rd, "I_\tau" description] & & \underline{Z} \\
        \overline{Y} \arrow[rr, "a", pos=0.3] \arrow[rd, "b"'] & & \overline{X} \arrow[rd, "c"] & \\
         & \overline{\tau X} \arrow[rr, "d"'] & & \overline{Z}. 
        \end{tikzcd}        
    \]
    After the elimination, this part transforms into 
    \[
        \begin{tikzcd}
        \underline{Y} \arrow[rr, "a"] \arrow[rd, no head, dotted] & & \underline{X} \arrow[rd, "c - d I_\tau"] \arrow[dd, no head, dotted, ] & \\
         & \arrow[rr, no head, dotted] & & \underline{Z} \\
        \overline{Y} \arrow[rr, "a - I_\tau b"] \arrow[rrru, "-db" description, bend left = 10] \arrow[rd, no head, dotted] & & \overline{X} \arrow[rd, "c"] & \\
         & \arrow[rr, no head, dotted] & & \overline{Z}.
        \end{tikzcd}
    \]
    Such a modification must be applied for each pair $(Y, Z)$ forming a non-trivial square with $(X, \tau X)$ in $\Omega$ as above. This process may be repeated for each pair $\{X, \tau X\}$ of non-$\tau$-invariant summands of $\Omega$.
\end{proposition}

\begin{proof}
    Immediate from \Cref{prop:gauss-elim}.
\end{proof}

\subsection{Overall procedure}
\label{subsec:reduce-procedure}

\begin{figure}[t]
    \centering
    \begin{tikzcd}[row sep=2.5em, column sep=3em]
        {[L]^\tau} \arrow[r, "{1, 2, 3}"] \arrow[d, "\cal{Q}", maps to] & {\Omega^\tau} \arrow[r, "4"] & \Omega^\tau_0 \arrow[d, "\cal{Q}", maps to] & \\
        {\cal{Q}(L)} \arrow[rr, "5"] & &  \Gamma \arrow[r, "{6, 7}"] &  \Gamma_0
        \end{tikzcd}
    \caption{Reduction of the involutive complex.}
    \label{fig:reduce-scheme}
\end{figure}

Combining the above techniques, we obtain a general procedure for reducing the involutive complex $\CKhI(L)$ for an involutive link $L$. 

\begin{algorithm}
\label{alg:inv-kh-reduction}
    Given an involutive link diagram $L$ with a symmetric tangle decomposition
    \[
        L = D(T_s, T_a, \tau T_a),
    \]
    a strong deformation retract of the involutive Khovanov complex 
    \[
        \cal{Q}(L) := \cal{Q}([L]^\tau)
    \]
    is obtained as follows:
    \begin{enumerate}
        \item Reduce the on-axis $\tau$-complex $[T_s]^\tau$ by applying the equivariant delooping and Gaussian elimination
        \[
            [T_s]^\tau \to \Omega_s^\tau. 
        \]
        \item Reduce the off-axis complex $[T_a]$ by applying the ordinary delooping and Gaussian elimination
        \[
            [T_a] \to \Omega_a.
        \]
        \item Compose the above complexes and maps by $D$ (\Cref{prop:sym-decomp-maps}), to get a $\tau$-equivariant strong deformation retraction
        \[
            [L]^\tau \to D(\Omega_s, \Omega_a, \tau \Omega_a) = \Omega^\tau.
        \]
        Here $\Omega^\tau$ is $\tau$-matched, with the matching induced from that of $\Omega_s^\tau$ together with the interchange of the two off-axis factors.
        \item Further reduce the $\tau$-equivariant complex by applying the equivariant delooping and Gaussian elimination (\Cref{prop:equiv-deloop,prop:equiv-gauss-elim})
        \[
            \Omega^\tau \to \Omega^\tau_0.
        \]
        \item Applying the functor $\cal{Q}$ to get a strong deformation retract of $\cal{Q}(L)$ (the maps obtained above are strictly $\tau$-equivariant, on which $\cal{Q}$ is strictly functorial),
        \[
            \cal{Q}(L) := \cal{Q}([L]^\tau) \to \cal{Q}(\Omega^\tau_0) =: \Gamma.
        \]
        \item Apply symmetry-breaking reductions of \Cref{prop:vertical-reduction},
        \[
            \Gamma \to \Gamma'.
        \]
        \item Further apply delooping and Gaussian elimination until all summands are empty diagrams, and there are no invertible arrows left in the differential,
        \[
            \Gamma' \to \Gamma_0.
        \]
        \item Apply the TQFT $\cal{F} = \Hom([\emptyset], -)$ to obtain a strong deformation retract $\cal{F}(\Gamma_0)$ of $\CKhI(L)$.
    \end{enumerate}
    The procedure up to step 7 is depicted in \Cref{fig:reduce-scheme}.
\end{algorithm}

    \section{Equivariant Rasmussen invariant}
\label{sec:equiv-rasmussen}

Here, we apply the techniques of \Cref{sec:reducing} to efficiently compute the equivariant Rasmussen invariant $(\ds, \us)$ for strongly invertible knots and links, introduced in \cite{Sano:2025b}. 

\subsection{Lee class and Rasmussen invariant}

Here, we review the description of the ordinary $s$-invariant in terms of the \textit{$h$-divisibility of the Lee class}, as given in \cite{KimSano:2025}. First, we additionally impose a local relation $t = 0$, and denote the quotient category by the same symbol $\Cob(B)$. With
\begin{center}
    \tikzset{every picture/.style={line width=0.75pt}} 

\begin{tikzpicture}[x=0.75pt,y=0.75pt,yscale=-1,xscale=1]

\draw  [draw opacity=0] (96.94,23.75) .. controls (96.29,26) and (89.45,27.76) .. (81.11,27.76) .. controls (72.35,27.76) and (65.24,25.81) .. (65.24,23.41) .. controls (65.24,23.26) and (65.27,23.11) .. (65.32,22.97) -- (81.11,23.41) -- cycle ; \draw  [color={rgb, 255:red, 128; green, 128; blue, 128 }  ,draw opacity=1 ] (96.94,23.75) .. controls (96.29,26) and (89.45,27.76) .. (81.11,27.76) .. controls (72.35,27.76) and (65.24,25.81) .. (65.24,23.41) .. controls (65.24,23.26) and (65.27,23.11) .. (65.32,22.97) ;  
\draw  [draw opacity=0][dash pattern={on 0.84pt off 2.51pt}] (65.63,22.33) .. controls (66.9,20.27) and (73.43,18.7) .. (81.28,18.7) .. controls (89.93,18.7) and (96.96,20.6) .. (97.15,22.96) -- (81.28,23.06) -- cycle ; \draw  [color={rgb, 255:red, 128; green, 128; blue, 128 }  ,draw opacity=1 ][dash pattern={on 0.84pt off 2.51pt}] (65.63,22.33) .. controls (66.9,20.27) and (73.43,18.7) .. (81.28,18.7) .. controls (89.93,18.7) and (96.96,20.6) .. (97.15,22.96) ;  
\draw   (65.12,23.41) .. controls (65.12,14.57) and (72.28,7.41) .. (81.11,7.41) .. controls (89.95,7.41) and (97.11,14.57) .. (97.11,23.41) .. controls (97.11,32.24) and (89.95,39.4) .. (81.11,39.4) .. controls (72.28,39.4) and (65.12,32.24) .. (65.12,23.41) -- cycle ;
\draw  [fill={rgb, 255:red, 0; green, 0; blue, 0 }  ,fill opacity=1 ] (74.63,14.68) .. controls (74.63,13.55) and (75.55,12.63) .. (76.68,12.63) .. controls (77.82,12.63) and (78.73,13.55) .. (78.73,14.68) .. controls (78.73,15.81) and (77.82,16.73) .. (76.68,16.73) .. controls (75.55,16.73) and (74.63,15.81) .. (74.63,14.68) -- cycle ;
\draw  [fill={rgb, 255:red, 0; green, 0; blue, 0 }  ,fill opacity=1 ] (84.23,14.68) .. controls (84.23,13.55) and (85.15,12.63) .. (86.28,12.63) .. controls (87.42,12.63) and (88.33,13.55) .. (88.33,14.68) .. controls (88.33,15.81) and (87.42,16.73) .. (86.28,16.73) .. controls (85.15,16.73) and (84.23,15.81) .. (84.23,14.68) -- cycle ;

\draw (14,14.61) node [anchor=north west][inner sep=0.75pt]    {$h\ =\ $};

\end{tikzpicture}
\end{center}
and the hollow dot
\begin{center}
    \tikzset{every picture/.style={line width=0.75pt}} 

\begin{tikzpicture}[x=0.75pt,y=0.75pt,yscale=-.75,xscale=.75]

\draw  [dash pattern={on 0.84pt off 2.51pt}] (10.93,11.56) -- (56.56,11.56) -- (56.56,55.49) -- (10.93,55.49) -- cycle ;
\draw  [dash pattern={on 0.84pt off 2.51pt}] (101.41,12.06) -- (148.08,12.06) -- (148.08,56.99) -- (101.41,56.99) -- cycle ;
\draw  [color={rgb, 255:red, 0; green, 0; blue, 0 }  ,draw opacity=1 ][fill={rgb, 255:red, 255; green, 255; blue, 255 }  ,fill opacity=1 ][line width=0.75]  (29.5,34.2) .. controls (29.5,32.43) and (30.93,31) .. (32.7,31) .. controls (34.47,31) and (35.91,32.43) .. (35.91,34.2) .. controls (35.91,35.97) and (34.47,37.41) .. (32.7,37.41) .. controls (30.93,37.41) and (29.5,35.97) .. (29.5,34.2) -- cycle ;
\draw  [fill={rgb, 255:red, 0; green, 0; blue, 0 }  ,fill opacity=1 ] (121.54,34.52) .. controls (121.54,32.75) and (122.97,31.32) .. (124.74,31.32) .. controls (126.51,31.32) and (127.95,32.75) .. (127.95,34.52) .. controls (127.95,36.29) and (126.51,37.73) .. (124.74,37.73) .. controls (122.97,37.73) and (121.54,36.29) .. (121.54,34.52) -- cycle ;
\draw  [dash pattern={on 0.84pt off 2.51pt}] (209.95,11.91) -- (254.54,11.91) -- (254.54,54.84) -- (209.95,54.84) -- cycle ;

\draw (68,25.32) node [anchor=north west][inner sep=0.75pt]    {$=$};
\draw (167,24.24) node [anchor=north west][inner sep=0.75pt]    {$-$};
\draw (190,24.17) node [anchor=north west][inner sep=0.75pt]    {$h$};

\end{tikzpicture}
\end{center}
the following equations hold:
\begin{center}
    \tikzset{every picture/.style={line width=0.75pt}} 

\begin{tikzpicture}[x=0.75pt,y=0.75pt,yscale=-.75,xscale=.75]

\draw  [dash pattern={on 0.84pt off 2.51pt}] (10.93,11.06) -- (56.56,11.06) -- (56.56,54.99) -- (10.93,54.99) -- cycle ;
\draw  [dash pattern={on 0.84pt off 2.51pt}] (115.41,11.06) -- (162.08,11.06) -- (162.08,55.99) -- (115.41,55.99) -- cycle ;
\draw  [fill={rgb, 255:red, 0; green, 0; blue, 0 }  ,fill opacity=1 ] (23.5,34.2) .. controls (23.5,32.43) and (24.93,31) .. (26.7,31) .. controls (28.47,31) and (29.91,32.43) .. (29.91,34.2) .. controls (29.91,35.97) and (28.47,37.41) .. (26.7,37.41) .. controls (24.93,37.41) and (23.5,35.97) .. (23.5,34.2) -- cycle ;
\draw  [fill={rgb, 255:red, 0; green, 0; blue, 0 }  ,fill opacity=1 ] (36,34.2) .. controls (36,32.43) and (37.43,31) .. (39.2,31) .. controls (40.97,31) and (42.41,32.43) .. (42.41,34.2) .. controls (42.41,35.97) and (40.97,37.41) .. (39.2,37.41) .. controls (37.43,37.41) and (36,35.97) .. (36,34.2) -- cycle ;
\draw  [fill={rgb, 255:red, 0; green, 0; blue, 0 }  ,fill opacity=1 ] (135.54,33.52) .. controls (135.54,31.75) and (136.97,30.32) .. (138.74,30.32) .. controls (140.51,30.32) and (141.95,31.75) .. (141.95,33.52) .. controls (141.95,35.29) and (140.51,36.73) .. (138.74,36.73) .. controls (136.97,36.73) and (135.54,35.29) .. (135.54,33.52) -- cycle ;
\draw  [dash pattern={on 0.84pt off 2.51pt}] (204.93,11.06) -- (250.56,11.06) -- (250.56,54.99) -- (204.93,54.99) -- cycle ;
\draw  [dash pattern={on 0.84pt off 2.51pt}] (322.41,11.06) -- (369.08,11.06) -- (369.08,55.99) -- (322.41,55.99) -- cycle ;
\draw  [color={rgb, 255:red, 0; green, 0; blue, 0 }  ,draw opacity=1 ][fill={rgb, 255:red, 255; green, 255; blue, 255 }  ,fill opacity=1 ] (217.5,34.06) .. controls (217.5,32.29) and (218.93,30.85) .. (220.7,30.85) .. controls (222.47,30.85) and (223.91,32.29) .. (223.91,34.06) .. controls (223.91,35.83) and (222.47,37.26) .. (220.7,37.26) .. controls (218.93,37.26) and (217.5,35.83) .. (217.5,34.06) -- cycle ;
\draw  [color={rgb, 255:red, 0; green, 0; blue, 0 }  ,draw opacity=1 ][fill={rgb, 255:red, 255; green, 255; blue, 255 }  ,fill opacity=1 ] (230,34.06) .. controls (230,32.29) and (231.43,30.85) .. (233.2,30.85) .. controls (234.97,30.85) and (236.41,32.29) .. (236.41,34.06) .. controls (236.41,35.83) and (234.97,37.26) .. (233.2,37.26) .. controls (231.43,37.26) and (230,35.83) .. (230,34.06) -- cycle ;
\draw  [color={rgb, 255:red, 0; green, 0; blue, 0 }  ,draw opacity=1 ][fill={rgb, 255:red, 255; green, 255; blue, 255 }  ,fill opacity=1 ] (342.54,34.38) .. controls (342.54,32.61) and (343.97,31.17) .. (345.74,31.17) .. controls (347.51,31.17) and (348.95,32.61) .. (348.95,34.38) .. controls (348.95,36.15) and (347.51,37.58) .. (345.74,37.58) .. controls (343.97,37.58) and (342.54,36.15) .. (342.54,34.38) -- cycle ;
\draw  [dash pattern={on 0.84pt off 2.51pt}] (415.93,11.06) -- (461.56,11.06) -- (461.56,54.99) -- (415.93,54.99) -- cycle ;
\draw  [color={rgb, 255:red, 0; green, 0; blue, 0 }  ,draw opacity=1 ][fill={rgb, 255:red, 0; green, 0; blue, 0 }  ,fill opacity=1 ] (428.5,34.56) .. controls (428.5,32.79) and (429.93,31.35) .. (431.7,31.35) .. controls (433.47,31.35) and (434.91,32.79) .. (434.91,34.56) .. controls (434.91,36.33) and (433.47,37.76) .. (431.7,37.76) .. controls (429.93,37.76) and (428.5,36.33) .. (428.5,34.56) -- cycle ;
\draw  [color={rgb, 255:red, 0; green, 0; blue, 0 }  ,draw opacity=1 ][fill={rgb, 255:red, 255; green, 255; blue, 255 }  ,fill opacity=1 ] (441,34.56) .. controls (441,32.79) and (442.43,31.35) .. (444.2,31.35) .. controls (445.97,31.35) and (447.41,32.79) .. (447.41,34.56) .. controls (447.41,36.33) and (445.97,37.76) .. (444.2,37.76) .. controls (442.43,37.76) and (441,36.33) .. (441,34.56) -- cycle ;

\draw (68,26.21) node [anchor=north west][inner sep=0.75pt]    {$=$};
\draw (177,36.21) node [anchor=north west][inner sep=0.75pt]    {$,$};
\draw (95,26.21) node [anchor=north west][inner sep=0.75pt]    {$h$};
\draw (262,26.21) node [anchor=north west][inner sep=0.75pt]    {$=$};
\draw (288,26.21) node [anchor=north west][inner sep=0.75pt]    {$-h$};
\draw (387,36.21) node [anchor=north west][inner sep=0.75pt]    {$,$};
\draw (470,26.21) node [anchor=north west][inner sep=0.75pt]    {$=\ 0.$};

\end{tikzpicture}
\end{center}
The corresponding Frobenius extension is 
\[
    R = \ZZ[h],\quad A = R[X]/(X^2 - hX),
\]
and with $Y = X - h$, we have equations
\begin{gather*}
    X^2 = hX,\ XY = 0,\ Y^2 = -hY,\\
    \Delta X = X \otimes X,\ \Delta Y = Y \otimes Y.
\end{gather*}
The link homology theory obtained from the corresponding TQFT is the \textit{$U(1)$-equivariant Khovanov homology}, also called the \textit{bigraded Bar-Natan homology}. 

\begin{figure}[t]
    \centering
    \begin{subfigure}[t]{0.3\linewidth}
        \centering
        \tikzset{every picture/.style={line width=0.75pt}} 

\begin{tikzpicture}[x=0.75pt,y=0.75pt,yscale=-.8,xscale=.8]

\clip (0,0) rectangle + (100, 100);

\draw [line width=1.5]    (70.58,41.39) .. controls (63.5,-18) and (4,25.5) .. (42.36,77.01) ;
\draw [line width=1.5]    (53.15,86.58) .. controls (94,115) and (123,32) .. (35.5,45.5) ;
\draw [line width=1.5]    (25.13,46.95) .. controls (-12.3,54.35) and (15.68,141.29) .. (69.06,54.44) ;
\draw    (54.5,13) -- (47.5,13) ;
\draw [shift={(45.5,13)}, rotate = 360] [color={rgb, 255:red, 0; green, 0; blue, 0 }  ][line width=0.75]    (10.93,-4.9) .. controls (6.95,-2.3) and (3.31,-0.67) .. (0,0) .. controls (3.31,0.67) and (6.95,2.3) .. (10.93,4.9)   ;

\end{tikzpicture}
        \caption{}
        \label{fig:intro-a}
    \end{subfigure}
    \begin{subfigure}[t]{0.3\linewidth}
        \centering
        \tikzset{every picture/.style={line width=0.75pt}} 

\begin{tikzpicture}[x=0.75pt,y=0.75pt,yscale=-.8,xscale=.8]

\clip (0,0) rectangle + (100, 100);

\draw  [color={rgb, 255:red, 208; green, 2; blue, 27 }  ,draw opacity=1 ][line width=1.5]  (49.5,13) .. controls (62.5,12) and (67.5,20) .. (69.5,37) .. controls (71.5,54) and (79.27,44.04) .. (89.5,55) .. controls (99.73,65.96) and (98.5,74) .. (96.5,83) .. controls (94.5,92) and (82.5,100) .. (65.5,91) .. controls (48.5,82) and (55.5,83) .. (40.5,92) .. controls (25.5,101) and (16.5,91) .. (10.5,84) .. controls (4.5,77) and (5.5,62) .. (12.5,57) .. controls (19.5,52) and (30.5,52) .. (32.5,37) .. controls (34.5,22) and (36.5,14) .. (49.5,13) -- cycle ;
\draw  [color={rgb, 255:red, 0; green, 116; blue, 255 }  ,draw opacity=1 ][line width=1.5]  (49.5,44) .. controls (61.5,44) and (75,41) .. (68.5,63) .. controls (62,85) and (45,87) .. (35,67) .. controls (25,47) and (37.5,44) .. (49.5,44) -- cycle ;

\draw (72,5.4) node [anchor=north west][inner sep=0.75pt]    {$\textcolor[rgb]{0.82,0.01,0.11}{X}$};
\draw (70.5,58.4) node [anchor=north west][inner sep=0.75pt]  [color={rgb, 255:red, 0; green, 117; blue, 255 }  ,opacity=1 ]  {$Y$};

\end{tikzpicture}
        \caption{}
        \label{fig:intro-b}
    \end{subfigure}
    \begin{subfigure}[t]{0.35\linewidth}
        \centering
        \tikzset{every picture/.style={line width=0.75pt}} 

\begin{tikzpicture}[x=0.75pt,y=0.75pt,yscale=-.8,xscale=.8]

\clip (0,0) rectangle + (150, 100);

\draw [color={rgb, 255:red, 208; green, 2; blue, 27 }  ,draw opacity=1 ][line width=1.5]    (98.22,13.41) .. controls (107.84,12.45) and (111.54,20.08) .. (113.02,36.28) .. controls (114.5,52.49) and (120.25,43) .. (127.82,53.44) .. controls (135.38,63.88) and (134.47,71.55) .. (132.99,80.13) .. controls (131.51,88.71) and (128.99,89.47) .. (124.14,90.98) ;
\draw [color={rgb, 255:red, 0; green, 116; blue, 255 }  ,draw opacity=1 ][line width=1.5]    (98.22,42.96) .. controls (107.1,42.96) and (117.09,40.1) .. (112.28,61.07) .. controls (107.47,82.03) and (104.49,78.54) .. (99.21,78.03) ;
\draw [color={rgb, 255:red, 208; green, 2; blue, 27 }  ,draw opacity=1 ][line width=1.5]    (98.22,13.41) .. controls (71.82,13.41) and (9,1.67) .. (9,56) .. controls (9,110.33) and (96.96,90.49) .. (124.14,90.98) ;
\draw  [fill={rgb, 255:red, 0; green, 0; blue, 0 }  ,fill opacity=1 ] (26.3,58.45) .. controls (26.3,56.76) and (27.67,55.39) .. (29.35,55.39) .. controls (31.04,55.39) and (32.41,56.76) .. (32.41,58.45) .. controls (32.41,60.13) and (31.04,61.5) .. (29.35,61.5) .. controls (27.67,61.5) and (26.3,60.13) .. (26.3,58.45) -- cycle ;
\draw [color={rgb, 255:red, 208; green, 2; blue, 27 }  ,draw opacity=1 ][line width=1.5]  [dash pattern={on 1.69pt off 2.76pt}]  (124.14,90.98) .. controls (108.4,90.99) and (108.4,81.84) .. (102.3,82.23) .. controls (96.2,82.61) and (91.24,87.94) .. (91.57,88.71) .. controls (91.89,89.47) and (75.23,88.71) .. (69.89,81.84) .. controls (64.55,74.98) and (65.67,60.11) .. (70.85,55.35) .. controls (76.03,50.58) and (84.17,50.58) .. (85.65,36.28) .. controls (87.13,21.99) and (88.28,16.13) .. (95.31,13.97) ;
\draw [color={rgb, 255:red, 0; green, 116; blue, 255 }  ,draw opacity=1 ][line width=1.5]  [dash pattern={on 1.69pt off 2.76pt}]  (100.46,78.03) .. controls (95.91,78.45) and (91.02,70.52) .. (87.43,61.26) .. controls (83.83,51.99) and (87.65,43.7) .. (95.54,43.06) ;
\draw [color={rgb, 255:red, 74; green, 144; blue, 226 }  ,draw opacity=1 ][fill={rgb, 255:red, 255; green, 255; blue, 255 }  ,fill opacity=0.8 ][line width=1.5]    (98.22,42.96) .. controls (71.82,42.96) and (49.78,48.32) .. (49.3,59.99) .. controls (48.82,71.67) and (65.84,72.77) .. (100.46,78.03) ;
\draw  [color={rgb, 255:red, 0; green, 0; blue, 0 }  ,draw opacity=1 ][fill={rgb, 255:red, 255; green, 255; blue, 255 }  ,fill opacity=1 ] (57.92,60.71) .. controls (57.92,59.02) and (59.28,57.66) .. (60.97,57.66) .. controls (62.66,57.66) and (64.02,59.02) .. (64.02,60.71) .. controls (64.02,62.4) and (62.66,63.76) .. (60.97,63.76) .. controls (59.28,63.76) and (57.92,62.4) .. (57.92,60.71) -- cycle ;

\draw (117.65,6.76) node [anchor=north west][inner sep=0.75pt]    {$\textcolor[rgb]{0.82,0.01,0.11}{X}$};
\draw (113.91,63.95) node [anchor=north west][inner sep=0.75pt]  [color={rgb, 255:red, 0; green, 117; blue, 255 }  ,opacity=1 ]  {$Y$};

\end{tikzpicture}
        \caption{}
        \label{fig:intro-c}
    \end{subfigure}
    \caption{The Lee cycle $\ca(K)$ regarded as a dotted cobordism.}
    \label{fig:intro-lee-cycle-cob}    
\end{figure}
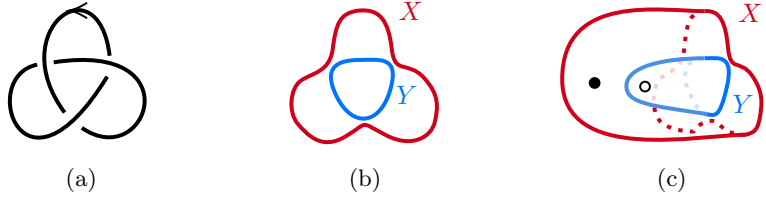

First, consider the case $B = \emptyset$ and let $L$ be a link diagram. Let $L^\lout$ denote the \textit{Seifert resolution} of $L$, i.e.\ the diagram obtained by resolving each crossing of $L$ in the orientation-preserving way. Let
\[
    S_L: \varnothing \to L^\lout
\]
be the cobordism consisting of a cup $U$ for each circle $C \subset L^\lout$ with boundary $\del U = C$. The \textit{Lee cycle} $\ca(L)$ is an element of
\[
    \Hom([\emptyset], [L]) \isom \CKh(L),
\]
defined as follows: to each component of $S_L$, add a dot $(\bullet)$ or a hollow dot $(\circ)$, according to the \textit{$XY$-labeling} on the Seifert circles (see \Cref{fig:intro-lee-cycle-cob} for a simple example, and \cite[Section 3]{KimSano:2025} for the precise definition). From the relation $(\bullet \circ) = 0$, it follows that $\ca(L)$ is indeed a cycle, meaning
\[
    d \circ \ca(L) = 0.
\]
The \textit{Lee class} $[\ca(L)]$ is defined as the homotopy class of $\ca(L)$ (or the homology class, if we consider it in $\Kh(L)$), and its \textit{$h$-divisibility}, denoted $d_h(L)$, is defined as the maximal $k \geq 0$ such that 
\[
    \ca(L) \htpy_s h^k z 
        \, :\Leftrightarrow \,
    \exists l \geq 0, h^l \ca(L) \htpy h^{k + l} z
\]
for some cycle $z$. Then the $s$-invariant of $L$ is defined by the quantity
\[
    s(L) = 2 d_h(L) + w(L) - r(L) + 1
\]
where $w(L)$ denotes the writhe of $L$ and $r(L)$ the number of Seifert circles of $L$ (i.e. the number of circles in $L^\lout$). In \cite[Proposition 3.9]{KimSano:2025}, it is proved that $s$ is indeed an invariant of the link, and in particular when $K$ is a knot and $\ca(K)$ is considered over an arbitrary field $F$, the knot invariant $s(K)$ coincides with the Rasmussen invariant $s^F(K)$ of $K$ over $F$. 

\begin{figure}[t]
    \centering
    \input{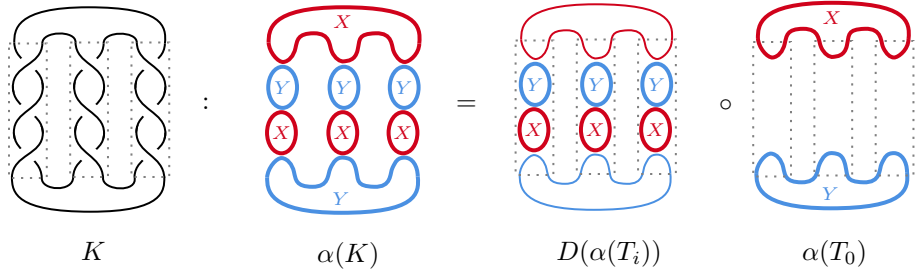}
    \caption{A tangle decomposition of the Lee cycle.}
    \label{fig:intro-lee-decomp}
\end{figure}

The above reformulation of the $s$-invariant allows one to compute $s$ from smaller tangle pieces. Suppose we are given a tangle decomposition of $L$ as
\[
    L = D(T_1, T_2, \ldots, T_d).
\]
Here, each $T_i$ is a tangle diagram and $D$ is a $d$-input planar arc diagram. The Lee cycle for each $T_i$ is defined similarly as a cobordism
\[
    \ca(T_i): [T_i^\lin] \to [T_i^\lout]
\]
where $T_i^\lin$ denotes the crossingless tangle diagram obtained by removing the circle components of $T_i^\lout$. Let $T_0$ be the crossingless unlink diagram given by the composition
\[
    T_0 = D(T_1^\lin, T_2^\lin, \ldots, T_d^\lin). 
\]
Then we may write
\begin{equation}
\label{eqn:lee-cycle-decomp}
    \ca(L) = D(\ca(T_1), \ca(T_2), \ldots, \ca(T_d)) \circ \ca(T_0).    
\end{equation}
See \Cref{fig:intro-lee-decomp} for an example. By applying reductions to the local complexes $[T_i]$, and then to the global composite complex, we may efficiently compute $d_h(L)$, and hence determine $s(L)$. See \cite[Section 3]{KimSano:2025} for the details. 

\subsection{Equivariant Lee class and equivariant Rasmussen invariant}

In this subsection, we focus on \textit{strongly invertible links}. First, we review the definition of the \textit{equivariant Rasmussen invariant} $(\ds(L), \us(L))$ of a strongly invertible link $L$, rephrased in the framework of \cite{KimSano:2025}.

A \textit{strongly invertible link} $(L, \tau)$ is an involutive link such that the involution $\tau$ reverses the orientation, i.e. $\tau(L) = L^r$. A \textit{strongly invertible link diagram} $L$ is an involutive link diagram such that $\tau(L) = L^r$. Let $L$ be a strongly invertible link diagram with Lee cycle $\ca(L) \in \CKh(L)$. From \cite[Proposition 3.3]{Sano:2025b},\footnote{
    \cite[Proposition 3.3]{Sano:2025b} in the published version falsely states that the Lee cycle $\ca(L, o)$ for \textit{any} orientation $o$ of the link is $\tau$-invariant; this holds only for orientations satisfying $\tau(o) = -o$. Thus the claim remains valid for the given orientation of the strongly invertible link. An erratum has been submitted to the journal, and the error has been fixed in the 5th version of the arXiv preprint, where all main theorems remain valid.
}
$\ca(L)$ is $\tau$-invariant, and hence the two copies
\[
    \dca(L) \in \CKhI^0(L),\quad \uca(L) \in \CKhI^1(L)
\]
in the involutive complex $\CKhI(L)$ are cycles, called the \textit{equivariant Lee cycles}. The $h$-divisibilities of those homotopy classes are denoted $\dd_h(L), \ud_h(L)$, and the quantities
\begin{align*}
    \ds(L) &:= 2 \dd_h(L) + w(L) - r(L) + 1,\\
    \us(L) &:= 2 \ud_h(L) + w(L) - r(L) + 1
\end{align*}
are invariants of the strongly invertible link $L$, which are together called the \textit{equivariant Rasmussen invariant} of $L$ (\cite[Definition 3.23]{Sano:2025b}). 

Now, suppose $L$ is a strongly invertible link diagram with a symmetric tangle decomposition
\[
    L = D(T_s, T_a, \tau T_a).
\]
The Lee cycle $\ca(L)$ in $\CKh(L)$ is decomposed as 
\[
    \ca(L) = D(\ca(T_s), \ca(T_a), \tau \ca(T_a)) \circ \ca(T_0)
\]
using the notation of \eqref{eqn:lee-cycle-decomp}. Here, we used $\ca(\tau T_a) = \tau \ca(T_a)$. Suppose we have obtained a reduced complex $\Gamma_0$ as in \Cref{alg:inv-kh-reduction}, so that $\cal{F}(\Gamma_0)$ is a strong deformation retract of the involutive complex $\CKhI(L)$. Let us observe how the equivariant Lee cycles $\dca(L), \uca(L)$ transform under the reduction. 

First, under the $\tau$-equivariant reduction 
\[
    [L] = D([T_s], [T_a], \tau [T_a]) \longrightarrow D(\Omega_s, \Omega_a, \tau \Omega_a) \longrightarrow \Omega_0,
\]
the Lee cycle transforms into 
\[
    \ca(L) \mapsto h^k z
\]
for some $\tau$-invariant cycle $z \in \cal{F}(\Omega_0)$ and $k \geq 0$. The above reduction is strictly $\tau$-equivariant, so after applying the functor $\cal{Q}$ to get $\Gamma = \cal{Q}(\Omega_0)$, we obtain two copies of the cycle $z$,
\[
    \underline{z}, \overline{z} \,\in\, \cal{F}(\Gamma)
\]
as cycles of homological degree $0$ and $1$. Under the subsequent reduction
\[
    \Gamma \longrightarrow \Gamma_0,
\]
the two cycles transform into 
\[
    \underline{z} \mapsto h^{l_0} z_0,\quad 
     \overline{z} \mapsto h^{l_1} z_1
\]
for some cycles $z_0, z_1 \in \cal{F}(\Gamma_0)$ and $l_0, l_1 \geq 0$. Since the above reductions are homotopy equivalences, \cite[Proposition 3.6]{KimSano:2025} gives
\[
    \dd_h(L) = k + l_0 + d_h([z_0]),\quad
    \ud_h(L) = k + l_1 + d_h([z_1]).
\]
Thus it remains to determine the $h$-divisibilities of the homotopy classes $[z_0], [z_1]$ in order to determine the equivariant Rasmussen invariant.

\subsection{Computing $s$ from a linear system}
\label{subsec:compute-s}

When the reduced complex $\Gamma_0$ is simple enough, we may determine the $h$-divisibilities of $z_0, z_1$ by hand, as we shall see in \Cref{sec:pen-paper-computation}. For more complicated knots and links, we pass to algorithmic computations. 

Since $\cal{F}(\Gamma_0)$ is an ordinary chain complex over $R = \FF_2[h]$ (which is a PID), we may compute its homology and extract the $h$-divisibilities (modulo torsion) of the homology classes $[z_0], [z_1]$, and determine $(\ds, \us)$. However, when $L$ is large, this process could be computationally infeasible. Here, we give an alternative method based on $\FF_2$-linear systems, which is similar to the method given in \cite{SS:2020}. To make the arguments clear, we proceed in an abstract setting. 

Let $\FF$ be a field, and $R := \FF[h]$ with $\deg(h) = -2$. Let $C$ be a bounded complex of finitely generated free $R$-modules with distinguished $R$-basis $G^i$ of $C^i$, together with a secondary degree function $q \colon G^i \to \ZZ$. Further assume that the differential $d$ of $C$ has bidegree $(1, 0)$, and that all $x \in G$ have $q(x) \equiv q_0 \pmod{2}$ for some fixed $q_0 \in \ZZ$. We write $C^{i,j} \subset C^i$ for the homogeneous part as an $\FF$-vector space, namely
\[
    C^{i, j} := \operatorname{span}_{\FF} \{\ h^k x \mid x \in G^i,\ k \ge 0,\ q(x) - 2k = j \ \}.
\]
Set $\bar{C} := C / (h - 1) C$, and let $\pi \colon C \to \bar{C}$ be the projection. For each $x \in C$, its image under $\pi$ is denoted $\bar{x} = \pi(x)$, and $\bar{G}^i := \pi G^i$ form an $\FF$-basis of $\bar{C}^i$. For $j \in \ZZ$, set
\[
    F^j \bar{C}^i := \operatorname{span}_{\FF} \{\ \bar{x} \mid x \in G^i,\ q(x) \geq j \ \}
\]
which gives a descending filtration on $\bar{C}$. 

\begin{lemma}[Rees correspondence]
\label{lem:rees-corresp}
    Let $j \in \ZZ$ with $j \equiv q_0 \pmod{2}$. Then $C^{i, j}$ and $F^j \bar{C}^i$ are isomorphic as $\FF$-vector spaces, with an isomorphism given by $\pi|_{C^{i, j}}$ and its inverse by
    \[
        L^j\colon F^j \bar{C}^i \to C^{i, j},\quad \bar{x} \mapsto h^{(q(x) - j)/2} x.
    \]
    More strongly, $C^{*, j}$ and $F^j \bar{C}^*$ are isomorphic as $\FF$-chain complexes. 
\end{lemma}

Let $H$ denote the homology of $C$. For a homology class $[z] \in H$, define its \textit{$h$-divisibility} (modulo torsion) by
\[
    d_h([z]) := \max \{\ k \geq 0 \mid [[z]] \in h^k (H / \Tor(H)) \ \}
\]
where $[[z]] := [z] + \Tor(H)$. (When $[z]$ is torsion, we put $d_h([z]) = \infty$.) From the structure theorem of finitely generated graded modules over the graded PID $R = \FF[h]$, torsion summands of $H$ are of the form $R/(h^l)$ with $l \geq 1$, so the above definition can be rewritten as 
\[
    d_h([z]) = \max \{\ k \geq 0 \mid \exists l \geq 0,\ h^l [z] \in h^{k + l} H \ \}. 
\]

\begin{proposition}
    Let $z \in C$ be a homogeneous cycle of bidegree $(i, j)$. The following are equivalent: 
    \begin{enumerate}
        \item $d_h([z]) \geq k$.
        \item There exists $\bar{w} \in \bar{C}^{i - 1}$ such that $\bar{z} + d\bar{w} \in F^{j + 2k} \bar{C}^i$.
    \end{enumerate}
\end{proposition}

\begin{proof}
    Put $j' := j + 2k$. 

    \vspace{.5em} \noindent
    \textbf{(1) $\Rightarrow$ (2)}. Write $h^l [z] = h^{k + l}[c]$ for some homogeneous cycle $c \in C^{i, j'}$ and some $l \geq 0$. On the level of chains, there is some $w \in C^{i - 1, j - 2l}$ such that $h^l z + dw = h^{k + l} c$. Apply $\pi$ and we get 
    \[
        \bar{z} + d\bar{w} = \bar{c}.
    \] 
    Here, $\bar{w} \in \bar{C}^{i - 1}$ and $\bar{c} \in F^{j'} \bar{C}^i$ as desired. 

    \vspace{.5em} \noindent
    \textbf{(2) $\Rightarrow$ (1)}. Put $\bar{c} := \bar{z} + d\bar{w} \in F^{j'} \bar{C}^i$. Take $j_0 \leq j$ with $j_0 \equiv q_0 \pmod{2}$ so that $\bar{w} \in F^{j_0}\bar{C}^{i - 1}$, and apply $L^{j_0}$ to get
    \[
        L^{j_0}(\bar{c}) = L^{j_0}(\bar{z}) + d L^{j_0}(\bar{w}) 
        \,\in\, C^{i, j_0}. 
    \]
    With $l := \frac{j - j_0}{2}$ and $c := L^{j'}(\bar{c})$, we have $L^{j_0}(\bar{c}) = h^{k + l} c$ and $L^{j_0}(\bar{z}) = h^l z$. Passing to homology, we get 
    \[
        h^l [z] = h^{k + l}[c]
    \] 
    which implies $d_h([z]) \geq k$. 
\end{proof}

\begin{corollary}
\label{cor:h-div-by-linear-sys}
    Let $z \in C$ be a homogeneous cycle of bidegree $(i, j)$. $d_h([z])$ can be computed as the maximal $k \geq 0$ such that the $\FF$-linear system 
    \[
        d \bar{w} = \bar{z}
    \]
    in $\bar{C}^i / F^{j + 2k} \bar{C}^i$ has a solution $\bar{w}$.
\end{corollary}

Back to the case of the equivariant Lee classes, recall that $z_i$ for $i = 0, 1$ is a homogeneous cycle of $\cal{F}(\Gamma_0)$ of homological degree $i$, and let $j_i$ denote its $q$-degree. Setting $h = 1$, i.e.\ passing to $\bar{z}_i := \pi(z_i)$, \Cref{cor:h-div-by-linear-sys} computes $d_h([z_i])$ as the maximal $k \geq 0$ such that the $\FF_2$-linear system
\[
    d \bar{w} = \bar{z}_i
\]
in $\bar{C}^i / F^{j_i + 2k} \bar{C}^i$ has a solution $\bar{w}$. This is far more efficient than directly computing the homology of $\cal{F}(\Gamma_0)$ over $\FF_2[h]$. This method can also be applied to computing the ordinary Rasmussen invariant over an arbitrary field $F$. 

The program \verb|yui| \cite{Sano:YUI} has been improved so that it now computes the involutive Khovanov homology using \Cref{alg:inv-kh-reduction}, and the equivariant Rasmussen invariant using the algorithm described in this section.
    \section{Diagrammatic computation}
\label{sec:pen-paper-computation}

Finally, we prove \Cref{thm:pretzel-equiv-s}, demonstrating how the procedure described in \Cref{subsec:reduce-procedure} can be applied to compute the equivariant Rasmussen invariants by hand. We refer to the proof of \cite[Theorem 3]{KimSano:2025} for a similar computation in the non-equivariant setting. 

\begin{restate-theorem}[thm:pretzel-equiv-s]
    For any odd positive integer $p \geq 3$, the strongly invertible pretzel knot $K = P(-p, p, -p)$ has
    \[
        (\ds(K), \us(K)) = (0, 2).
    \]
\end{restate-theorem}

First, we recall some notation introduced in \cite{KimSano:2025}. For any $q \in \ZZ \setminus \{0\}$, let $T_q$ denote the (unoriented) tangle diagram obtained by adding $|q|$ half-twists to a pair of crossingless vertical strands, positive or negative depending on the sign of $q$. 
\begin{center}
    \tikzset{every picture/.style={line width=0.75pt}} 

\begin{tikzpicture}[x=0.75pt,y=0.75pt,yscale=-1,xscale=1]

\begin{knot}[
  clip width=3pt,
  end tolerance=1pt,
]

\strand [line width=1.5]    (91.05,10.83) .. controls (90.61,25.83) and (71.06,24.83) .. (71.33,40.16) .. controls (71.61,55.49) and (91.33,56.16) .. (91.33,70.83) .. controls (91.33,85.49) and (71.33,84.83) .. (70.67,100.16) ;
\strand [line width=1.5]    (72.24,10.16) .. controls (72.64,25.16) and (91.59,24.83) .. (91.33,40.16) .. controls (91.08,55.49) and (71.98,55.49) .. (71.98,70.16) .. controls (71.98,84.83) and (90.05,85.49) .. (90.67,100.83) ;


\end{knot}

\draw  [color={rgb, 255:red, 255; green, 255; blue, 255 }  ,draw opacity=1 ][fill={rgb, 255:red, 255; green, 255; blue, 255 }  ,fill opacity=1 ] (67,37.17) -- (95,37.17) -- (95,70.83) -- (67,70.83) -- cycle ;
\draw  [color={rgb, 255:red, 128; green, 128; blue, 128 }  ,draw opacity=1 ] (104,93) .. controls (108.67,93) and (111,90.67) .. (111,86) -- (111,65.73) .. controls (111,59.06) and (113.33,55.73) .. (118,55.73) .. controls (113.33,55.73) and (111,52.4) .. (111,45.73)(111,48.73) -- (111,25.45) .. controls (111,20.78) and (108.67,18.45) .. (104,18.45) ;

\draw (9,44.4) node [anchor=north west][inner sep=0.75pt]    {$T_{q} \ =$};
\draw (72.67,44.57) node [anchor=north west][inner sep=0.75pt]    {$\vdots $};
\draw (125,44.4) node [anchor=north west][inner sep=0.75pt]    {$q$};

\end{tikzpicture}
\end{center}
Let $T^\pm_q$ denote the tangle diagram $T_q$ equipped with the orientation given so that each crossing has the sign indicated by the superscript. Let $\sf{E}_0$ and $\sf{E}_1$ denote the following unoriented $4$-end crossingless tangle diagrams:
\begin{center}
    \tikzset{every picture/.style={line width=0.75pt}} 

\begin{tikzpicture}[x=0.75pt,y=0.75pt,yscale=-.75,xscale=.75]

\draw  [dash pattern={on 4.5pt off 4.5pt}] (206,34.77) .. controls (206,20.71) and (217.4,9.31) .. (231.46,9.31) .. controls (245.52,9.31) and (256.91,20.71) .. (256.91,34.77) .. controls (256.91,48.83) and (245.52,60.22) .. (231.46,60.22) .. controls (217.4,60.22) and (206,48.83) .. (206,34.77) -- cycle ;
\draw  [draw opacity=0][line width=1.5]  (251.56,20.25) .. controls (246.19,23.83) and (239.17,26) .. (231.48,26) .. controls (223.83,26) and (216.84,23.85) .. (211.48,20.3) -- (231.48,2.73) -- cycle ; \draw  [line width=1.5]  (251.56,20.25) .. controls (246.19,23.83) and (239.17,26) .. (231.48,26) .. controls (223.83,26) and (216.84,23.85) .. (211.48,20.3) ;  
\draw  [draw opacity=0][line width=1.5]  (211.4,49.29) .. controls (216.77,45.71) and (223.79,43.54) .. (231.48,43.54) .. controls (239.13,43.54) and (246.12,45.69) .. (251.48,49.24) -- (231.48,66.81) -- cycle ; \draw  [line width=1.5]  (211.4,49.29) .. controls (216.77,45.71) and (223.79,43.54) .. (231.48,43.54) .. controls (239.13,43.54) and (246.12,45.69) .. (251.48,49.24) ;  

\draw  [dash pattern={on 4.5pt off 4.5pt}] (86.48,9.29) .. controls (100.54,9.29) and (111.94,20.69) .. (111.94,34.75) .. controls (111.94,48.81) and (100.54,60.21) .. (86.48,60.21) .. controls (72.42,60.21) and (61.02,48.81) .. (61.02,34.75) .. controls (61.02,20.69) and (72.42,9.29) .. (86.48,9.29) -- cycle ;
\draw  [draw opacity=0][line width=1.5]  (101,54.85) .. controls (97.42,49.48) and (95.25,42.46) .. (95.25,34.77) .. controls (95.25,27.12) and (97.4,20.13) .. (100.95,14.77) -- (118.52,34.77) -- cycle ; \draw  [line width=1.5]  (101,54.85) .. controls (97.42,49.48) and (95.25,42.46) .. (95.25,34.77) .. controls (95.25,27.12) and (97.4,20.13) .. (100.95,14.77) ;  
\draw  [draw opacity=0][line width=1.5]  (71.96,14.69) .. controls (75.54,20.06) and (77.71,27.08) .. (77.71,34.77) .. controls (77.71,42.42) and (75.56,49.41) .. (72.01,54.77) -- (54.43,34.77) -- cycle ; \draw  [line width=1.5]  (71.96,14.69) .. controls (75.54,20.06) and (77.71,27.08) .. (77.71,34.77) .. controls (77.71,42.42) and (75.56,49.41) .. (72.01,54.77) ;

\draw (7,22.4) node [anchor=north west][inner sep=0.75pt]    {$\mathsf{E}_{0} \ =\ $};
\draw (149,23.4) node [anchor=north west][inner sep=0.75pt]    {$\mathsf{E}_{1} \ =\ $};
\draw (121,24) node [anchor=north west][inner sep=0.75pt]   [align=left] {,};

\end{tikzpicture}
\end{center}
Consider two $q$-degree $-2$ endomorphisms $a, b$ on $\sf{E}_1$ defined by 
\begin{align*}
    a &= u_X + l_Y = u_Y + l_X, \\
    b &= u_X - l_X = u_Y - l_Y
\end{align*}
where $u_X, l_X$ are endomorphisms defined as
\begin{center}
    \tikzset{every picture/.style={line width=0.75pt}} 

\begin{tikzpicture}[x=0.75pt,y=0.75pt,yscale=-.8,xscale=.8]

\draw  [dash pattern={on 4.5pt off 4.5pt}] (69,32.27) .. controls (69,18.21) and (80.4,6.81) .. (94.46,6.81) .. controls (108.52,6.81) and (119.91,18.21) .. (119.91,32.27) .. controls (119.91,46.33) and (108.52,57.72) .. (94.46,57.72) .. controls (80.4,57.72) and (69,46.33) .. (69,32.27) -- cycle ;
\draw  [draw opacity=0][line width=1.5]  (114.56,17.75) .. controls (109.19,21.33) and (102.17,23.5) .. (94.48,23.5) .. controls (86.83,23.5) and (79.84,21.35) .. (74.48,17.8) -- (94.48,0.23) -- cycle ; \draw  [line width=1.5]  (114.56,17.75) .. controls (109.19,21.33) and (102.17,23.5) .. (94.48,23.5) .. controls (86.83,23.5) and (79.84,21.35) .. (74.48,17.8) ;  
\draw  [draw opacity=0][line width=1.5]  (74.4,46.79) .. controls (79.77,43.21) and (86.79,41.04) .. (94.48,41.04) .. controls (102.13,41.04) and (109.12,43.19) .. (114.48,46.74) -- (94.48,64.31) -- cycle ; \draw  [line width=1.5]  (74.4,46.79) .. controls (79.77,43.21) and (86.79,41.04) .. (94.48,41.04) .. controls (102.13,41.04) and (109.12,43.19) .. (114.48,46.74) ;  

\draw  [fill={rgb, 255:red, 0; green, 0; blue, 0 }  ,fill opacity=1 ] (91,23.27) .. controls (91,21.5) and (92.43,20.07) .. (94.2,20.07) .. controls (95.97,20.07) and (97.41,21.5) .. (97.41,23.27) .. controls (97.41,25.04) and (95.97,26.47) .. (94.2,26.47) .. controls (92.43,26.47) and (91,25.04) .. (91,23.27) -- cycle ;
\draw  [dash pattern={on 4.5pt off 4.5pt}] (209,31.27) .. controls (209,17.21) and (220.4,5.81) .. (234.46,5.81) .. controls (248.52,5.81) and (259.91,17.21) .. (259.91,31.27) .. controls (259.91,45.33) and (248.52,56.72) .. (234.46,56.72) .. controls (220.4,56.72) and (209,45.33) .. (209,31.27) -- cycle ;
\draw  [draw opacity=0][line width=1.5]  (254.56,16.75) .. controls (249.19,20.33) and (242.17,22.5) .. (234.48,22.5) .. controls (226.83,22.5) and (219.84,20.35) .. (214.48,16.8) -- (234.48,-0.77) -- cycle ; \draw  [line width=1.5]  (254.56,16.75) .. controls (249.19,20.33) and (242.17,22.5) .. (234.48,22.5) .. controls (226.83,22.5) and (219.84,20.35) .. (214.48,16.8) ;  
\draw  [draw opacity=0][line width=1.5]  (214.4,45.79) .. controls (219.77,42.21) and (226.79,40.04) .. (234.48,40.04) .. controls (242.13,40.04) and (249.12,42.19) .. (254.48,45.74) -- (234.48,63.31) -- cycle ; \draw  [line width=1.5]  (214.4,45.79) .. controls (219.77,42.21) and (226.79,40.04) .. (234.48,40.04) .. controls (242.13,40.04) and (249.12,42.19) .. (254.48,45.74) ;  

\draw  [fill={rgb, 255:red, 0; green, 0; blue, 0 }  ,fill opacity=1 ] (231.25,40.47) .. controls (231.25,38.7) and (232.69,37.27) .. (234.46,37.27) .. controls (236.23,37.27) and (237.66,38.7) .. (237.66,40.47) .. controls (237.66,42.24) and (236.23,43.68) .. (234.46,43.68) .. controls (232.69,43.68) and (231.25,42.24) .. (231.25,40.47) -- cycle ;

\draw (17,24.67) node [anchor=north west][inner sep=0.75pt]    {$u_{X} \ =\ $};
\draw (127,25.77) node [anchor=north west][inner sep=0.75pt]   [align=left] {,};
\draw (157,23.67) node [anchor=north west][inner sep=0.75pt]    {$l_{X} \ =\ $};

\end{tikzpicture}
\end{center}
and $u_Y, l_Y$ similarly using $\circ$. 

Now, we proceed to the proof of \Cref{thm:pretzel-equiv-s}, with $p = 3$ for simplicity. The pretzel knot $K = P(-3, 3, -3)$ has an obvious symmetric tangle decomposition
\[
    K = D(T^+_{-3}, T^-_{3}, T^+_{-3}).
\]
From \cite[Proposition 4.3]{KimSano:2025}, the $\tau$-complex $[K]$ reduces to 
\[
    [K] \to D(E^+_{-3}, E^-_{3}, E^+_{-3}) =: \Omega,
\]
where $E^+_{-3}, E^-_{3}$ are sequential complexes of the form
\[
\begin{tikzcd}[row sep=.5em]
    E^+_{-3} \arrow[r, "=", phantom] & \{\ \underline{\sf{E}_1} \arrow[r, "a"] & \sf{E}_1 \arrow[r, "b"] & \sf{E}_1 \arrow[r, "s"] & \sf{E}_0\ \}, \\
    E^-_{3} \arrow[r, "=", phantom]  & \{\ \sf{E}_0 \arrow[r, "s"] & \sf{E}_1 \arrow[r, "b"] & \sf{E}_1 \arrow[r, "a"] & \underline{\sf{E}_1}\ \}.
\end{tikzcd}    
\]
The underlined objects indicate the homological grading $0$ part. Further, from \cite[Proposition 4.4]{KimSano:2025}, the Lee cycle $\ca(K)$ is transformed under this retraction to 
\[
    \ca(K) \mapsto D(z^+_{-3}, h^2 z^-_{3}, z^+_{-3}) = h^2 z,
\]
where $z^\pm_q$ denote the cycles given therein. Here, we put $z = D(z^+_{-3}, z^-_{3}, z^+_{-3})$. In the non-equivariant case, this shows $d_h(K) \geq 2$, and from $w(K) = 3$ and $r(K) = 8$, we have 
\[
    s(K) = 2d_h(K) + w(K) - r(K) + 1 \geq 0.
\]
However, since $K$ is slice, this is an equality, and hence the $h$-divisibility of $\ca(K)$ is exactly $2$. 

Now, pass to the involutive complexes, and consider the equivariant Lee cycles $\dca(K), \uca(K)$ and the two copies $\underline{z}, \bar{z}$ of the cycle $z$. The above $\tau$-equivariant reduction induces a reduction
\[
    \cal{Q}([K]) \longrightarrow \cal{Q}(\Omega)
\]
under which the equivariant Lee cycles transform as
\[
    \dca(K) \mapsto h^2 \underline{z},\quad \uca(K) \mapsto h^2 \bar{z}.
\]
This gives $\dd_h(K),\ \ud_h(K) \geq 2$ and $\ds(K),\ \us(K) \geq 0$. With the inequality $\ds(K) \leq s(K)$ of \cite[Corollary 3.25]{Sano:2025b} and $s(K) = 0$, we obtain $\ds(K) = 0$. To prove \Cref{thm:pretzel-equiv-s} for $p = 3$, it remains to show that $\uca(K)$ has $h$-divisibility exactly $3$, or equivalently, that $\bar{z}$ has $h$-divisibility exactly $1$.

\begin{figure}[p]
    \centering
    \vspace{-5em}
    \input{tikzpictures/m9_46-diagrams}
    \caption{Resolved diagrams}
    \label{fig:m9_46-diag}
    \vspace{2em}
    \input{tikzpictures/m9_46-cpx}
    \caption{Structure of $\Omega$}
    \label{fig:m9_46-cpx}
    \vspace{2em}
    \tikzset{every picture/.style={line width=0.75pt}} 

\begin{tikzpicture}[x=0.75pt,y=0.75pt,yscale=-.8,xscale=.8]

\draw [color={rgb, 255:red, 155; green, 155; blue, 155 }  ,draw opacity=1 ][line width=0.75] (233,51.32) -- (311,297.16) ;
\draw [color={rgb, 255:red, 155; green, 155; blue, 155 }  ,draw opacity=1 ]   (240.33,197.6) -- (311,297.16) ;
\draw [color={rgb, 255:red, 155; green, 155; blue, 155 }  ,draw opacity=1 ]   (244,271.16) -- (309,341.71) ;
\draw [color={rgb, 255:red, 155; green, 155; blue, 155 }  ,draw opacity=1 ][line width=0.75] (233,51.32) -- (309,341.71) ;
\draw [color={rgb, 255:red, 155; green, 155; blue, 155 }  ,draw opacity=1 ][line width=0.75]    (81,51.32) -- (155.67,197.6) ;
\draw [color={rgb, 255:red, 155; green, 155; blue, 155 }  ,draw opacity=1 ][line width=0.75]    (81,51.32) -- (155,267.04) ;
\draw [color={rgb, 255:red, 208; green, 2; blue, 27 }  ,draw opacity=1 ][line width=1.5]    (81,51.32) -- (153,51.32) ;

\draw (34.33,15) node [anchor=north west][inner sep=0.75pt]   [align=left] {\underline{-1}};
\draw (185.67,15) node [anchor=north west][inner sep=0.75pt]   [align=left] {\underline{0}};
\draw (336.33,15) node [anchor=north west][inner sep=0.75pt]   [align=left] {\underline{1}};
\draw (44.89,44.86) node [anchor=north] [inner sep=0.75pt]  [font=\scriptsize,color={rgb, 255:red, 0; green, 0; blue, 0 }  ,opacity=1 ]  {$\sf{A}(0,-1,0)$};
\draw (44.89,69.31) node [anchor=north] [inner sep=0.75pt]  [font=\scriptsize,color={rgb, 255:red, 155; green, 155; blue, 155 }  ,opacity=1 ]  {$\sf{A}(1,-2,0)$};
\draw (44.89,93.76) node [anchor=north] [inner sep=0.75pt]  [font=\scriptsize,color={rgb, 255:red, 155; green, 155; blue, 155 }  ,opacity=1 ]  {$\sf{A}(0,-2,1)$};
\draw (44.39,118.21) node [anchor=north] [inner sep=0.75pt]  [font=\scriptsize,color={rgb, 255:red, 155; green, 155; blue, 155 }  ,opacity=1 ]  {$\sf{B}(2,-3,0)$};
\draw (44.39,142.66) node [anchor=north] [inner sep=0.75pt]  [font=\scriptsize,color={rgb, 255:red, 155; green, 155; blue, 155 }  ,opacity=1 ]  {$\sf{B}(1,-3,1)$};
\draw (44.39,167.1) node [anchor=north] [inner sep=0.75pt]  [font=\scriptsize,color={rgb, 255:red, 155; green, 155; blue, 155 }  ,opacity=1 ]  {$\sf{B}(0,-3,2)$};
\draw (191.23,44.6) node [anchor=north] [inner sep=0.75pt]  [font=\scriptsize]  {$\underline{\sf{A}(0,0,0)}$};
\draw (194.73,69.41) node [anchor=north] [inner sep=0.75pt]  [font=\scriptsize,color={rgb, 255:red, 155; green, 155; blue, 155 }  ,opacity=1 ]  {$\sf{A}(1,-1,0)$};
\draw (194.73,93.97) node [anchor=north] [inner sep=0.75pt]  [font=\scriptsize,color={rgb, 255:red, 155; green, 155; blue, 155 }  ,opacity=1 ]  {$\sf{A}(0,-1,1)$};
\draw (194.73,118.53) node [anchor=north] [inner sep=0.75pt]  [font=\scriptsize,color={rgb, 255:red, 155; green, 155; blue, 155 }  ,opacity=1 ]  {$\sf{A}(2,-2,0)$};
\draw (194.73,143.09) node [anchor=north] [inner sep=0.75pt]  [font=\scriptsize,color={rgb, 255:red, 155; green, 155; blue, 155 }  ,opacity=1 ]  {$\sf{A}(1,-2,1)$};
\draw (194.73,167.65) node [anchor=north] [inner sep=0.75pt]  [font=\scriptsize,color={rgb, 255:red, 155; green, 155; blue, 155 }  ,opacity=1 ]  {$\sf{A}(0,-2,2)$};
\draw (194.73,192.21) node [anchor=north] [inner sep=0.75pt]  [font=\scriptsize]  {$\sf{C}_{0} (3,-3,0)$};
\draw (194.23,216.77) node [anchor=north] [inner sep=0.75pt]  [font=\scriptsize,color={rgb, 255:red, 155; green, 155; blue, 155 }  ,opacity=1 ]  {$\sf{B}(2,-3,1)$};
\draw (194.23,241.32) node [anchor=north] [inner sep=0.75pt]  [font=\scriptsize,color={rgb, 255:red, 155; green, 155; blue, 155 }  ,opacity=1 ]  {$\sf{B}(1,-3,2)$};
\draw (195.23,265.89) node [anchor=north] [inner sep=0.75pt]  [font=\scriptsize]  {$\sf{C}'_{0} (0,-3,3)$};
\draw (348.06,69.19) node [anchor=north] [inner sep=0.75pt]  [font=\scriptsize,color={rgb, 255:red, 155; green, 155; blue, 155 }  ,opacity=1 ]  {$\sf{A}(1,0,0)$};
\draw (348.06,93.76) node [anchor=north] [inner sep=0.75pt]  [font=\scriptsize,color={rgb, 255:red, 155; green, 155; blue, 155 }  ,opacity=1 ]  {$\sf{A}(0,0,1)$};
\draw (351.56,117.77) node [anchor=north] [inner sep=0.75pt]  [font=\scriptsize,color={rgb, 255:red, 155; green, 155; blue, 155 }  ,opacity=1 ]  {$\sf{A}(2,-1,0)$};
\draw (351.56,142.89) node [anchor=north] [inner sep=0.75pt]  [font=\scriptsize,color={rgb, 255:red, 155; green, 155; blue, 155 }  ,opacity=1 ]  {$\sf{A}(1,-1,1)$};
\draw (351.56,167.46) node [anchor=north] [inner sep=0.75pt]  [font=\scriptsize,color={rgb, 255:red, 155; green, 155; blue, 155 }  ,opacity=1 ]  {$\sf{A}(0,-1,2)$};
\draw (352.06,192.03) node [anchor=north] [inner sep=0.75pt]  [font=\scriptsize,color={rgb, 255:red, 155; green, 155; blue, 155 }  ,opacity=1 ]  {$\sf{D}(3,-2,0)$};
\draw (351.56,216.59) node [anchor=north] [inner sep=0.75pt]  [font=\scriptsize,color={rgb, 255:red, 155; green, 155; blue, 155 }  ,opacity=1 ]  {$\sf{A}(2,-2,1)$};
\draw (351.56,241.16) node [anchor=north] [inner sep=0.75pt]  [font=\scriptsize,color={rgb, 255:red, 155; green, 155; blue, 155 }  ,opacity=1 ]  {$\sf{A}(1,-2,2)$};
\draw (353.56,265.73) node [anchor=north] [inner sep=0.75pt]  [font=\scriptsize,color={rgb, 255:red, 155; green, 155; blue, 155 }  ,opacity=1 ]  {$\sf{D}'(0,-2,3)$};
\draw (351.56,290.29) node [anchor=north] [inner sep=0.75pt]  [font=\scriptsize]  {$\sf{C}_{0} (3,-3,1)$};
\draw (351.06,314.86) node [anchor=north] [inner sep=0.75pt]  [font=\scriptsize,color={rgb, 255:red, 128; green, 128; blue, 128 }  ,opacity=1 ]  {$\sf{B}(2,-3,2)$};
\draw (353.06,339.44) node [anchor=north] [inner sep=0.75pt]  [font=\scriptsize]  {$\sf{C}'_{0} (1,-3,3)$};

\end{tikzpicture}
    \caption{Structure of $\Omega_0$}
    \label{fig:m9_46-cpx-red}
\end{figure}

\Cref{fig:m9_46-cpx} depicts the structure of the reduced $\tau$-complex $\Omega$ around homological grading $0$, where $\sf{A}, \sf{B}, \sf{C}, \sf{C}', \sf{D}, \sf{D}'$ represent diagrams as in \Cref{fig:m9_46-diag}. Note that $\sf{A}, \sf{B}$ are self-symmetric, and $(\sf{C}, \sf{C}'), (\sf{D}, \sf{D}')$ are symmetric pairs. Summands of $\Omega$ are indexed by vertices $(u, v, w)$ with $-3 \leq v \leq 0$ and $0 \leq u, w \leq 3$. We refer to the components of the differential as \textit{edges}, and edges of the form
\[
    (u, v, w) \to (u, v+1, w)
\]
\textit{central edges}. In \Cref{fig:m9_46-cpx}, morphisms corresponding to the central edges are indicated by the colors: red for $a$, blue for $b$, green for $m$, and yellow for $\Delta$. The cycle $z$ has target in $\sf{A}(0, 0, 0)$. 

Using \cite[Lemma 4.2]{KimSano:2025} together with \Cref{prop:equiv-deloop,prop:equiv-gauss-elim}, we can equivariantly eliminate most of the summands around homological grading $0$, resulting in a reduced $\tau$-complex $\Omega_0$ as in \Cref{fig:m9_46-cpx-red}. Here, we have left the edge $\sf{A}(0, -1, 0) \to \sf{A}(0, 0, 0)$ unchanged. $\sf{C}_0$ and $\sf{C}'_0$ represent diagrams obtained from $\sf{C}$ and $\sf{C}'$ by removing the inner circles, which arise as the results of reducing the central edges
\[
    \sf{C} \xrightarrow{\ m\ } \sf{D},\quad  
    \sf{C}' \xrightarrow{\ m\ } \sf{D}'.
\]
Under this reduction, the cycle $z$ transforms into a cycle $z_0$, which still has target in $\sf{A}(0, 0, 0)$. 

Let us observe what happens if we further try to reduce the unchanged edge. By delooping the inner circles of the two end diagrams of $\sf{A}(0, -1, 0) \to \sf{A}(0, 0, 0)$, from \cite[Lemma 4.2]{KimSano:2025}, the edge transforms into 
\[
\begin{tikzcd}[row sep=.15em]
    & & & &  \sf{A}_0^+\arrow[rrdd,"I" description] \arrow[rr,"Y"] && \sf{A}_0^+ \\
    \sf{A}\arrow[rr,"a"]&& \sf{A} & \leadsto& &&\\
    & & & & \sf{A}_0^- \arrow[rr,"X"]&& \sf{A}_0^-
\end{tikzcd}
\]
The image of $z_0$ will now have components into $\sf{A}_0^+(0, 0, 0)$ and $\sf{A}_0^-(0, 0, 0)$. Then, after eliminating the diagonal edge $I$, the resulting cycle will have non-zero cobordisms into $\sf{C}_0(3, -3, 0)$ and $\sf{C}'_0(0, -3, 3)$, as a result of \Cref{prop:gauss-elim}. This is what obstructs $z_0$ from being $h$-divisible. 

Now, let us consider the involutive complex $\Gamma = \cal{Q}(\Omega_0)$. Apply the symmetry-breaking reduction (\Cref{prop:vertical-reduction}) to eliminate the vertical pairs 
\begin{align*}
    \underline{\sf{C}'_0}(0, -3, 3) &\xrightarrow{\ I \ } \overline{\sf{C}'_0}(0, -3, 3), \\
    \underline{\sf{C}'_0}(1, -3, 3) &\xrightarrow{\ I \ } \overline{\sf{C}'_0}(1, -3, 3)
\end{align*}
of $\Gamma$, and let $\Gamma'$ denote the resulting complex. We claim,

\begin{claim}
    In $\Gamma'$, 
    \[
        \overline{\sf{A}}(0, -1, 0) \xrightarrow{\ a \ } \overline{\sf{A}}(0, 0, 0)
    \]
    is the only non-trivial arrow starting from $\overline{\sf{A}}(0, -1, 0)$.
\end{claim}

\begin{proof}
    Let $e, e'$ denote the edges
    
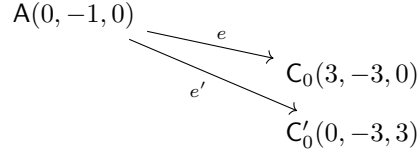
\[
\begin{tikzcd}[row sep=.5em]
{\sf{A}(0, -1, 0)} \arrow[rrd, "e"] \arrow[rrdd, "e'"'] &  &               \\
&  & {\sf{C}_0(3, -3, 0)} \\
&  & {\sf{C}'_0(0, -3, 3)}
\end{tikzcd}
    \]
    in $\Omega_0$. From the construction, we have $e' = \tau e$. From \Cref{prop:vertical-reduction}, the effect of applying the vertical reduction in $\Gamma$ to the edge
    \[
        \underline{\sf{C}'_0}(0, -3, 3) \xrightarrow{\ I \ } \overline{\sf{C}'_0}(0, -3, 3)
    \]
    is that, (i) the two summands will be eliminated, (ii) the edge $e$ will be modified as 
    \[
        e - I_\tau e': \overline{\sf{A}}(0, -1, 0) \to \overline{\sf{C}_0}(3, -3, 0),
    \]
    and (iii) since $\sf{C}'_0(1, -3, 3)$ is the only target of an edge starting from $\sf{C}'_0(0, -3, 3)$, a single new edge will be produced,
    \[
\begin{tikzcd}[row sep=.5em]
&  & {\underline{\sf{C}'_0}(1, -3, 3).} \\
{\overline{\sf{A}}(0, -1, 0)} \arrow[rru] &  &
\end{tikzcd}
    \]
    (Furthermore, the edge $\underline{\sf{C}_0}(3, -3, 0) \to \underline{\sf{C}'_0}(1, -3, 3)$ will be modified, and the vertical arrow $\underline{\sf{C}_0}(3, -3, 0) \to \overline{\sf{C}_0}(3, -3, 0)$ will vanish. Neither is an edge starting from $\overline{\sf{A}}(0, -1, 0)$.)

    For the modified edge in (ii), we have
    \[
        e - I_\tau e' = e (I - I_\tau) = 0
    \]
    since $\sf{A}$ is $\tau$-invariant with no asymmetric circles. The newly produced edge in (iii) will disappear after eliminating
    \[
        \underline{\sf{C}'_0}(1, -3, 3) \xrightarrow{\ I \ } \overline{\sf{C}'_0}(1, -3, 3).
    \]
    Thus, the only remaining edge from $\overline{\sf{A}}(0, -1, 0)$ is the one into $\overline{\sf{A}}(0, 0, 0)$. 
\end{proof}

Now, the edge in $\Gamma'$
\[
    \overline{\sf{A}}(0, -1, 0) \xrightarrow{\ a \ } \overline{\sf{A}}(0, 0, 0),
\]
reduces to
\[
\begin{tikzcd}[row sep=.3em]
    0 \arrow[rrd, dotted] \arrow[rr, dotted] && \overline{\sf{A}_0^+}(0, 0, 0) \\
    \overline{\sf{A}_0^-}(0, -1, 0) \arrow[rr, dotted] && 0.
\end{tikzcd}
\]
Under this reduction, the image of $\bar{z}_0$ has target only in $\overline{\sf{A}_0^+}(0, 0, 0)$, and one can verify from the explicit descriptions of the delooping isomorphism of \Cref{prop:delooping} and the retraction of \Cref{prop:gauss-elim} that this cobordism has a single $h$ factor, coming from $(\bullet \bullet) = h(\bullet)$ or $(\circ \circ) = h(\circ)$ (depending on the $XY$-coloring on $\sf{A}$). Since there are no other incoming arrows into the target object, it follows that $\bar{z}_0$ is $h$-divisible exactly once. Thus we have $\ud_h(K) = 3$, $\us(K) = 2$ and the proof of \Cref{thm:pretzel-equiv-s} for $p = 3$ is complete. The proof for the general case proceeds similarly. 
    
    \printbibliography
    \addresses
\end{document}